\documentclass[10pt,amsfonts,psfig]{amsart}
\usepackage{amsmath,amssymb}
\begin{document}

\newtheorem{thm}{Theorem}[section]
\newtheorem{cor}[thm]{Corollary}
\newtheorem{lem}[thm]{Lemma}
\newtheorem{prop}[thm]{Proposition}
\newtheorem{obs}[thm]{Observation}
\theoremstyle{remark}
\newtheorem{rem}{Remark}[section]
\newtheorem{notation}{Notation}
\newtheorem{exc}[thm]{Exercise}
\newtheorem{exe}[thm]{Exercise}
\newtheorem{problem}[thm]{Problem}

\theoremstyle{definition}
\newtheorem{defn}{Definition}[section]

\theoremstyle{plain}
\newtheorem*{thmA}{Main Theorem (physical version)}
\newtheorem*{thmB}{Main Theorem (dynamical version)}
\newtheorem*{thmC}{Theorem (Global Lee-Yang-Fisher Current)}
\newtheorem*{thmRIG}{Local Rigidity Theorem}
\newtheorem*{LY Theorem}{Lee-Yang Theorem}
\newtheorem*{LY Theorem2}{General Lee-Yang Theorem}
\newtheorem*{LY TheoremBC}{Lee-Yang Theorem with Boundary Conditions}
\newtheorem*{LY Theorem2BC}{General Lee-Yang Theorem with Boundary Conditions}
\newtheorem*{conj}{Conjecture}

\numberwithin{equation}{section}
\numberwithin{figure}{section}

\newcommand{\figref}[1]{Figure~\ref{#1}}
\newcommand{\Rmig}{R} 
\newcommand{\Cmig}{C} 
\newcommand{\Cmigbl}{C_0}
\newcommand{\TOPmig}{\mathrm{T}} 
\newcommand{\BOTTOMmig}{\mathrm{B}} 
\newcommand{\FIXmig}{b} 
\newcommand{\CFIXmig}{e} 
\newcommand{\INDmig}{a} 
\newcommand{\Smig}{S} 

\newcommand{\parab} {{\mathrm{par}}}
\newcommand{\thin} {{\mathrm{thin}}}
\newcommand{\thick} {{ \mathrm{thick}}}
\newcommand{\loc}{{\mathrm {loc}}}
\newcommand{\hor}{{\mathrm {hor}}}
\newcommand{\ver}{{\mathrm {ver}}}
\newcommand{\ess}{{\mathrm{ess}}}
\newcommand{\sym}{{\mathrm{sym}}}

\newcommand{\Cyl}{{\mathrm{Cyl}}}

\newcommand{\lp}{{\,\mathrm {dominates}\,}}

\newcommand{\Rphys}{\mathcal R} 
\newcommand{\Cphys}{\mathcal C} 
\newcommand{\Solid}{\mathcal{SC}}
\newcommand{\Solidmig}{SC}
\newcommand{\Secphys}{\Pi}
\newcommand{\Secmig}{P}
\newcommand{\TOPphys}{\mathcal T} 
\newcommand{\Cphystl}{{\mathcal C}_1}
\newcommand{\Cphysbl}{{\mathcal C}_0}
\newcommand{\Cphyslow}{{\mathcal C}_*}
\newcommand{\BOTTOMphys}{\mathcal B} 
\newcommand{\FIXphys}{\beta} 
\newcommand{\CFIXphys}{\eta} 
\newcommand{\INDphys}{\alpha} 
\newcommand{\Sphys}{\mathcal S} 
\newcommand{\Icurve}{\mathcal G} 
\newcommand{\Imig}{G} 
\newcommand{\PI}{{\mathcal A}}
\newcommand{\crittemps}{{\mathscr C}} 
\newcommand{\epoints}{{\mathscr E}} 

\newcommand{\Line}{L}
\newcommand{\Lzero}{L_0}
\newcommand{\Lone}{L_1}
\newcommand{\Ltwo}{L_2}
\newcommand{\Lthree}{L_3}
\newcommand{\Lfour}{L_4}

\newcommand{\LLzero}{\LL_0}
\newcommand{\LLone}{\LL_1}

\newcommand{\Div}{{\rm Div}}
\newcommand{\Tongue}{\Upsilon}

\newcommand{\Par}{{{\mathcal P}}}
\newcommand{\inv}{{\iota}}

\def\cF{{\mathcal{F}}}
\def\cA{{\mathcal{A}}}

\newcommand{\leg}{{\mathrm leg}}
\newcommand{\il}{{\mathrm il}}

\newcommand{\Mconv}{{\rm M}}
\newcommand{\Hconv}{{\rm H}}
\newcommand{\tconv}{{\rm t}}

\newcommand{\Weight}{W}

\newcommand{\KSQRT}{\widehat \KK}

\newcommand{\ex}{{\mathrm{exc}}}
\newcommand{\out}{{\mathrm{out}}}

\newcommand{\vertbondpp}{\, \underset{\oplus}{\overset{\oplus}{|}} \, }
\newcommand{\vertbondpm}{\, \underset{\ominus}{\overset{\oplus}{|}} \,}
\newcommand{\vertbondmp}{\, \underset{\oplus}{\overset{\ominus}{|}} \, }
\newcommand{\vertbondmm}{\, \underset{\ominus}{\overset{\ominus}{|}} \, }

\font\nt=cmr7


 \def\note#1{}

\long\def\bignote#1{}


\newcommand{\QED}{\rlap{$\sqcup$}$\sqcap$\smallskip}

\def\sss{\subsubsection}

\newcommand{\correspond}{\Psi}
\newcommand{\conjugacy}{\psi}

\newcommand{\rank}{\rm rank}
\newcommand{\di}{\partial}
\newcommand{\dibar}{\bar\partial}
\newcommand{\hookra}{\hookrightarrow}
\newcommand{\ra}{\rightarrow}
\newcommand{\hra}{\hookrightarrow}
\newcommand{\imply}{\Rightarrow}
\def\lra{\longrightarrow}
\newcommand{\wc}{\underset{w}{\to}}
\newcommand{\tu}{\textup}

\def\ssk{\smallskip}
\def\msk{\medskip}
\def\bsk{\bigskip}
\def\noi{\noindent}
\def\nin{\noindent}
\def\lqq{\lq\lq}
\def\sm{\setminus}
\def\bolshe{\succ}
\def\ssm{\smallsetminus}
\def\tr{{\text{tr}}}
\def\Crit{{\mathrm{Crit}}}

\newcommand{\Exp}{\operatorname{Exp}}
\newcommand{\ang}{\operatorname{ang}}
\newcommand{\ctg}{\operatorname{ctg}}
\newcommand{\diam}{\operatorname{diam}}
\newcommand{\dist}{\operatorname{dist}}
\newcommand{\Hdist}{\operatorname{H-dist}}
\newcommand{\cl}{\operatorname{cl}}
\newcommand{\inter}{\operatorname{int}}
\renewcommand{\mod}{\operatorname{mod}}
\newcommand{\card}{\operatorname{card}}
\newcommand{\tl}{\tilde}
\newcommand{\ind}{ \operatorname{ind} }
\newcommand{\Dist}{\operatorname{Dist}}
\newcommand{\Graph}{\operatorname{Graph}}
\newcommand{\len}{\operatorname{\l}}
\newcommand{\vol}{\operatorname{vol}}

\renewcommand{\Re}{\operatorname{Re}}
\renewcommand{\Im}{\operatorname{Im}}

\newcommand{\orb}{\operatorname{orb}}
\newcommand{\HD}{\operatorname{HD}}
\newcommand{\supp}{\operatorname{supp}}
\newcommand{\id}{\operatorname{id}}
\newcommand{\length}{\operatorname{length}}
\newcommand{\dens}{\operatorname{dens}}
\newcommand{\meas}{\operatorname{meas}}
\newcommand{\area}{\operatorname{area}}
\renewcommand{\Im}{\operatorname{Im}}

\renewcommand{\d}{{\diamond}}

\newcommand{\lef}{{\mathrm{left}}}
\newcommand{\righ}{{\mathrm{right}}}

\newcommand{\Dil}{\operatorname{Dil}}
\newcommand{\Ker}{\operatorname{Ker}}
\newcommand{\tg}{\operatorname{tg}}
\newcommand{\codim}{\operatorname{codim}}
\newcommand{\isom}{\approx}
\newcommand{\comp}{\circ}
\newcommand{\esssup}{\operatorname{ess-sup}}
\newcommand{\Rat}{{\mathrm{Rat}}}
\newcommand{\hot}{{\mathrm{h.o.t.}}}
\newcommand{\Conf}{{\mathrm{Conf}}}

\newcommand{\SLa}{\underset{\La}{\Subset}}

\newcommand{\const}{\mathrm{const}}
\def\loc{{\mathrm{loc}}}
\def\fib{{\mathrm{fib}}}

\newcommand{\eps}{{\epsilon}}
\newcommand{\epsi}{{\epsilon}}
\newcommand{\veps}{{\varepsilon}}

\newcommand{\Ga}{{\Gamma}}
\newcommand{\De}{{\Delta}}
\newcommand{\de}{{\delta}}
\newcommand{\la}{{\lambda}}
\newcommand{\La}{{\Lambda}}
\newcommand{\si}{{\sigma}}
\newcommand{\Si}{{\Sigma}}
\newcommand{\Om}{{\Omega}}
\newcommand{\om}{{\omega}}
\newcommand{\Ups}{{\Upsilon}}

\newcommand{\al}{{\alpha}}
\newcommand{\ba}{{\mbox{\boldmath$\alpha$} }}
\newcommand{\be}{{\beta}}
\newcommand{\bbe}{{\mbox{\boldmath$\beta$} }}
\newcommand{\bk}{{\boldsymbol{\kappa}}}
\newcommand{\bg}{{\boldsymbol{\gamma}}}

\newcommand{\BFF} {{\boldsymbol{\FF} }}

\newcommand{\bare}{{\bar\eps}}

\newcommand{\Ray}{{\mathcal R}}
\newcommand{\Eq}{{\mathcal E}}
\newcommand{\PR}{PR}

\newcommand{\IS}{{\mathcal IS}}
\newcommand{\uN}{{\underline{N}}}
\newcommand{\cp}{{cp}}
\newcommand{\Sp}{{Sp}}
\newcommand{\Sieg}{{S}}
\newcommand{\Siegel}{\mathrm{Sieg}}
\newcommand{\Cir}{\operatorname{{\mathrm{Cir}}}}
\newcommand{\Sec}{{ \mathfrak{S}}}
\newcommand{\cyl}{{\mathrm{cyl}}}

\newcommand{\Disc}{{D}}
\newcommand{\hull}{{\mathrm{hull}}}
\newcommand{\num}{{\mathbf{p}}}

\newcommand{\AAA}{{\mathcal A}}
\newcommand{\BB}{{\mathcal B}}
\newcommand{\CC}{{\mathcal C}}
\newcommand{\DD}{{\mathcal D}}
\newcommand{\EE}{{\mathcal E}}
\newcommand{\EEE}{{\mathcal O}}
\newcommand{\II}{{\mathcal I}}
\newcommand{\FF}{{\mathcal F}}
\newcommand{\GG}{{\mathcal G}}
\newcommand{\JJ}{{\mathcal J}}
\newcommand{\HH}{{\mathcal H}}
\newcommand{\KK}{{\mathcal K}}
\newcommand{\LL}{{\mathcal L}}
\newcommand{\MM}{{\mathcal M}}
\newcommand{\NN}{{\mathcal N}}
\newcommand{\OO}{{\mathcal O}}
\newcommand{\PP}{{\mathcal P}}
\newcommand{\QQ}{{\mathcal Q}}
\newcommand{\QM}{{\mathcal QM}}
\newcommand{\QP}{{\mathcal QP}}
\newcommand{\QL}{{\mathcal Q}}

\newcommand{\RR}{{\mathcal R}}
\newcommand{\SSS}{{\mathcal S}}
\newcommand{\TT}{{\mathcal T}}
\newcommand{\TTT}{{\mathcal P}}
\newcommand{\UU}{{\mathcal U}}
\newcommand{\VV}{{\mathcal V}}
\newcommand{\WW}{{\mathcal W}}
\newcommand{\XX}{{\mathcal X}}
\newcommand{\YY}{{\mathcal Y}}
\newcommand{\ZZ}{{\mathcal Z}}

\newcommand{\A}{{\Bbb A}}
\newcommand{\BBB}{{\Bbb B}}
\newcommand{\C}{{\Bbb C}}
\newcommand{\bC}{{\bar{\Bbb C}}}
\newcommand{\D}{{\Bbb D}}
\newcommand{\Hyp}{{\Bbb H}}
\newcommand{\J}{{\Bbb J}}
\newcommand{\Ll}{{\Bbb L}}
\renewcommand{\L}{{\Bbb L}}
\newcommand{\M}{{\Bbb M}}
\newcommand{\N}{{\Bbb N}}
\newcommand{\Q}{{\Bbb Q}}
\renewcommand{\P}{{\Bbb{P}}}
\newcommand{\R}{{\Bbb R}}
\newcommand{\T}{{\Bbb T}}
\newcommand{\V}{{\Bbb V}}
\newcommand{\U}{{\Bbb U}}
\newcommand{\W}{{\Bbb W}}
\newcommand{\X}{{\Bbb X}}
\newcommand{\Z}{{\Bbb Z}}

\newcommand{\VVV}{{\mathbf U}}
\newcommand{\UUU}{{\mathbf U}}

\newcommand{\tT}{{\mathrm{T}}}
\newcommand{\tD}{{D}}
\newcommand{\hyp}{{\mathrm{hyp}}}
\newcommand{\Flower}{{\mathrm{Fl}}}
\newcommand{\Petal}{{\mathrm{Pet}}}

\newcommand{\fix}{{b}}
\newcommand{\cxfix}{{\xi}}
\newcommand{\LINV}{L_{\rm inv}}
\newcommand{\LLINV}{{\mathcal L}_{\rm inv}}
\newcommand{\f}{{\bf f}}
\newcommand{\g}{{\bf g}}
\newcommand{\h}{{\bf h}}
\renewcommand{\i}{{\bar i}}
\renewcommand{\j}{{\bar j}}

\newcommand{\Neck} {{\mathcal{N}}}

\newcommand{\Bc}{{\mathbf{c}}}
\newcommand{\Bp}{{\mathbf{p}}}
\newcommand{\Bq}{{\mathbf{q}}}
\newcommand{\Bf}{{\mathbf{f}}}
\newcommand{\Bg}{{\mathbf{g}}}
\newcommand{\BG}{{\mathbf{G}}}
\newcommand{\Bh}{{\mathbf{h}}}
\def\BH{{\mathbf{H}}}
\def\BJ{{\mathbf{J}}}
\def\Bl{{\mathbf{l}}}
\def\Bm{{\mathbf{m}}}
\def\BI{{\mathbf{I}}}
\def\BT{{\mathbf{T}}}

\def\Bk{{\mathbf{k}}}
\def\Bj{{\mathbf{j}}}
\def\BJ{{\mathbf{J}}}
\def\Bphi{{\mathbf{\Phi}}}
\def\Bpsi{{\mathbf{\Psi}}}

\newcommand{\BA}{{\mathbf{A}}}
\newcommand{\BC}{{\mathbf{C}}}
\newcommand{\BD}{{\boldsymbol{D}}}
\newcommand{\BO} {{\mathbf{O}}}
\newcommand{\BP} {{\boldsymbol{P}}}
\newcommand{\BS} {{\mathbf {S} }}
\newcommand{\BE}{{\boldsymbol{E}}}
\newcommand{\BF}{{\boldsymbol{F}}}
\newcommand{\BU}{{\mathbf{U}}}
\newcommand{\BW}{{\mathbf{W}}}
\newcommand{\BV}{{\mathbf{V}}}
\newcommand{\BOM}{{\boldsymbol{\Om}}}

\newcommand{\BY}{{\mathbf{Y}}}
\newcommand{\BZ}{{\mathbf{Z}}}

\newcommand{\BPi}{{\boldsymbol{\Pi}}}
\def\BUps{{\boldsymbol{\Upsilon}}}
\def\BLa{{\boldsymbol{\La}}}
\def\BGa{{\boldsymbol\Gamma}}
\def\BDe{{\boldsymbol\Delta}}
\def\BUps{{\boldsymbol\Upsilon}}
\def\BThe{{\boldsymbol\Theta}}
\def\BOm{{\boldsymbol \Om}}
\def\BPsi{{\boldsymbol\Psi}}

\def\Baleph{{\boldsymbol\aleph}}

\newcommand{\Ba}{{\mathbf{a}}}
\newcommand{\Bs}{{\mathbf{s}}}
\newcommand{\Bt}{{\mathbf{t}}}
\newcommand{\Bn}{{\mathbf{n}}}
\newcommand{\Bw}{{\mathbf{w}}}
\newcommand{\Bz}{{\mathbf{z}}}

\newcommand{\Bal}{{\boldsymbol{\alpha}}}
\newcommand{\Bde}{{\boldsymbol{\delta}}}
\newcommand{\Bga}{{\boldsymbol{\gamma}}}
\newcommand{\Bsi}{{\boldsymbol{\sigma}}}
\newcommand{\Bla}{{\boldsymbol{\lambda}}}
\def\Be{\mathbf{e}}
\def\Dia{{\Diamond}}

\def\SB{{\boldsymbol{\BB}}}

\def\ext{{\mathrm{ex}}}
\def\mouth{\operatorname{mouth}}
\def\tail{\operatorname{tail}}

\newcommand{\ul} {\underline{l}}
\newcommand{\ukappa} {\underline{\kappa}}
\newcommand{\ut} {\underline{t}}

\newcommand{\Comb}{{\it Comb}}
\newcommand{\Top}{{\operatorname{Top}}}
\newcommand{\Bottom}{{\operatorname{Bot}}}
\newcommand{\QC}{\mathcal QC}
\newcommand{\Def}{\mathcal Def}
\newcommand{\Teich}{\mathcal Teich}
\newcommand{\PPL}{{\mathcal P}{\mathcal L}}

\newcommand{\per }{{\mathfrak{p}}}
\newcommand{\pp }{{\mathfrak{p}}}
\newcommand{\p}{{\boldsymbol{\mathfrak{ {p}}}}}
\newcommand{\m}{{\mathfrak{m}}}

\newcommand{\Jac}{\operatorname{Jac}}
\newcommand{\Homeo}{\operatorname{Homeo}}
\newcommand{\AC}{\operatorname{AC}}
\newcommand{\Dom}{\operatorname{Dom}}

\newcommand{\depth}{\operatorname{depth}} 

\newcommand{\Hol}{{\rm Hol}}

\newcommand{\Aff}{\operatorname{Aff}}
\newcommand{\Euc}{{\mathrm{Euc}}}
\newcommand{\MobC}{\operatorname{M\ddot{o}b}({\Bbb C}) }
\newcommand{\PSL}{ {\mathcal{PSL}} }
\newcommand{\SL}{ {\mathcal{SL}} }
\newcommand{\CP}{ {\Bbb{CP}}   }

\newcommand{\hf}{{\hat f}}
\newcommand{\hz}{{\hat z}}
\newcommand{\hM}{{\hat M}}

\renewcommand{\lq}{``}
\renewcommand{\rq}{''}

\newcommand{\Ch}{\textrm{Ch}}


\catcode`\@=12

\def\Empty{}
\newcommand\oplabel[1]{
  \def\OpArg{#1} \ifx \OpArg\Empty {} \else
  	\label{#1}
  \fi}
		
%

\long\def\realfig#1#2#3#4{
\begin{figure}[htbp]
\centerline{\psfig{figure=#2,width=#4}}
\caption[#1]{#3}
\oplabel{#1}
\end{figure}}

%

\newcommand{\comm}[1]{}
\newcommand{\comment}[1]{}

\title{ 
      Lebesgue measure of Feigenbaum Julia sets}

\author {Artur Avila and Mikhail Lyubich}

\bigskip\bigskip

\date{}

\begin{abstract} 
  We construct Feigenbaum quadratic-like maps
  with a Julia set of positive
  Lebesgue measure.  Indeed, in the quadratic family $P_c: z \mapsto z^2+c$
  the corresponding set of parameters $c$ is shown to have
  positive Hausdorff dimension.
  Our examples include renormalization fixed points,
  and the corresponding quadratic polynomials in their stable manifold
  are the first known rational maps for which the hyperbolic dimension is
  different from the Hausdorff dimension of the Julia set.
\end{abstract}

\setcounter{tocdepth}{1}
 
\maketitle
\tableofcontents

\section{Introduction}

One of the major successes of the theory of one-dimensional dynamical systems
was the conceptual explanation, in terms of the dynamics of a
renormalization operator, of the striking universality phenomena
discovered by Feigenbaum and Coullet-Tresser in 1970's. 
At the center of the 
picture lies the concept of a {\it Feigenbaum map}, which is a quadratic-like map that can be
renormalized infinitely many times with bounded combinatorics and
{\it a priori} bounds (a certain control on the nonlinearity).  
The successive renormalizations are then exponentially asymptotic to a
{\it renormalization attractor}, see \cite{S,McM2,FCT}.
In the simplest case of 
stationary combinatorics, the renormalization attractor consists of a single
renormalization fixed point.  As a consequence, the dynamics of such Feigenbaum maps 
display remarkable self-similarity
reflected in  the geometry of the corresponding Julia sets.

In fact, understanding  the geometry of Feigenbaum Julia sets already played a key role in
the first proof of exponential convergence of the
renormalization \cite {McM2}.  However, 
for a long time the theory had been unable to tackle 
natural geometric problems:  do Feigenbaum Julia sets have full Hausdorff dimension or
even positive area?
 (See \cite {McM2}, page 177, question 3).
In \cite {AL}, a new approach to these problems was developed,
which allowed us to show, in particular, that Feigenbaum Julia sets
can have Hausdorff dimension strictly less than two, while leaving open the problem of whether
they can ever have positive area.  The 
goal of this work is to settle the latter question affirmatively.
Namely, we will show that Julia sets of positive area appear already
among Feigenbaum quadratic polynomials with stationary combinatorics
(note that there are only countably many such polynomials).
At the same time, we construct a set of parameters $c$ of positive
Hausdorff dimension such that the quadratic polynomials $P_c: z\mapsto
z^2+c$ are Feigenbaum maps with Julia sets of positive area.

\comm{
\begin{thm}

There exist Feigenbaum quadratic maps, indeed with stationary combinatorics,
whose Julia set has positive area.

\end{thm}

\begin{rem}

The property of having a Julia set of positive area
depends only on the renormalization combinatorics of a Feigenbaum map.
Thus Theorem \ref {} also provides renormalization fixed points with a Julia set
of positive area.

\end{rem}
}

Note  that our results (as well as the earlier results of \cite {AL})
go against 
intuition coming from hyperbolic geometry.
Indeed, according to the philosophy known as {\it Sullivan's dictionary},
there is a correspondence between certain objects and results in complex dynamics and hyperbolic
geometry. As McMullen suggested in \cite {McM2}  (see especially the last
paragraph on page 177),   Feigenbaum maps 
are analogous to $3$-manifolds with two ends, one of which is
geometrically finite, while the other one is asymptotically fibered over the circle. 
The limit sets $\La(\Gamma)$  of the corresponding  Kleinian groups have 
zero area but full Hausdorff dimension   (see Thurston \cite{Th} and Sullivan\cite{S-Kleinian}.
So, it may  look  like  the dictionary completely breaks down at this
point, though in fact there is a way to rehabilitate it 
(see \S \ref{Sullivan's Dictionary} below) .


\comm{
In a remarkable recent development, Buff and Cheritat have carried out
Douady's program for  constructing  the first example of a quadratic  Julia set of
positive are a \cite{BC}. 
(See also Yampolsky \cite{Ya-posmeas} for an alternative point of view on the
final piece of their argument.)  
There are three types of examples this strategy produced: Cremer, Siegel, and infinitely renormalizable of {\it unbounded } type. 
In this paper we construct examples of a very different nature, namely
infinitely renormalizable of {\it bounded}  type.
}

\subsection{Feigenbaum maps}\label{Feigen maps}

Let us begin with reminding briefly the main concepts of the
complex renormalization theory. (See \S \ref{ql maps} for a precise account.)
A {\it quadratic-like map} is a holomorphic double covering $f:U \to V$ where $U$
and $V$ are quasidisks with $U$ compactly contained in $V$.  The 
{\it filled-in Julia set} of $f$ is the set $K(f)$ of points $z$ with $f^n(z) \in U$ for all
$n \geq 0$; its boundary is the {\it Julia set}  $J(f)$.  The filled-in Julia set
is always a full compact set which is either connected or totally
disconnected, according to whether or not it contains the critical point.

Simplest examples of quadratic-like maps are given by restrictions
of quadratic maps $P_c:  z \mapsto z^2+c$ to suitable neighborhoods of
$K(P_c)$. 
 The precise choice of the restriction is dynamically inessential, which is expressed by saying they
all define the same {\it quadratic-like germ}.
The {\it Mandelbrot set} $\MM$  \note{notation!}
can then be defined as the set of parameters $c \in \C$ for which $K(P_c)$ is
connected. 

The central role of the quadratic family is
made clear by Douady-Hubbard's {\it Straightening Theorem} 
that states that each quadratic-like map with connected Julia set is 
{\it hybrid conjugate} to a unique quadratic map $P_c$, i.e., there exists a
quasiconformal map $h:(\C,K(f)) \to (\C,K(P_c))$ satisfying
$h \circ f=P_c \circ h$ near $K(f)$ and with $\bar \partial h|\,
K(f)=0$ a.e.  We say
that $P_c$ is the {\it straightening} of $f$, and write $c=\chi(f)$.

A quadratic-like map $f:U \to V$ is said to be {\it renormalizable} with
period $\per \geq 2$ if the $\per$-th iterate of $f$ can be restricted to a
quadratic-like map $g :U' \to V'$ 
such that the {\it little Julia sets}
$K_j:= f^j(K(g))$, $0 \leq j \leq \per-1$, are connected and do not cross each
other (meaning that $K_j\sm K_i$ are connected for $i\not=j$). 
 We can always choose $g$ to have the same critical point as $f$,
and such a $ g $ is called the {\it pre-renormalization of period $\per$}
of $f$.  The smallest possible value of $\per$ is called the {\it renormalization
period} of $f$, and the corresponding pre-renormalization,
considered up to
affine conjugacy, is called the {\it renormalization} of $f$ and denoted by
$R(f)$.  The {\it renormalization operator} $f \mapsto R f $ is then well
defined at the level of affine conjugacy classes of quadratic-like germs.

The set of parameter values corresponding to renormalizable quadratic maps
is disconnected.  Its connected components are called (maximal) {\it
Mandelbrot copies}, which can be of two types, 
{\it primitive} or {\it satellite}, 
according to whether they are canonically homeomorphic 
(via the straightening map $c \mapsto \chi(R(P_c))$) 
to the full Mandelbrot set
or to $\MM \setminus \{1/4\}$ 
 (note that $1/4$ is the {\it cusp} of the {\it main cardioid} bounding the
``largest''  component of the interior of $\MM$).
Alternatively, (maximal) satellite copies can be distinguished by the property
that they are  ``attached'' to the main cardioid at 
the ``missing'' cusp.  
They can also be distinguished dynamically: 
For the satellite renormalization
(with the minimal period), all
little Julia sets have a common touching point, while for the primitive
renormalization, they are pairwise disjoint.

The {\it renormalization combinatorics} of a renormalizable quadratic-like
map $f$ is the Mandelbrot copy $\MM'$ containing $\chi(f)$.\footnote {The
renormalization combinatorics
can be alternatively encoded by a finite graph, the
{\it Hubbard tree}, which describes the positioning of the little Julia sets
(of the first pre-renormalization) inside the full Julia set.
It coincides with the Hubbard tree of the superattracting map
$\Bf_{c'}$, $c'\in M'$, whose period is equal to the renormalization
period of $f$.}
The renormalization
period only depends on the renormalization combinatorics, but the converse
is false (except for period two).  There are however only finitely many
combinatorics corresponding to each  period.

Assume now that $f$ is {\it infinitely renormalizable},
 i.e., the renormalizations  $R^j f  $ are well defined for all $j \geq 0$.
We say that $f$ has {\it bounded combinatorics} if the
renormalization periods of the successive renormalizations $R^j f$,
$j \geq 0$,  remain bounded.  The combinatorics is {\it stationary} 
if it is the same for all $R^j f$. 

The ``analytic quality'' of a quadratic-like map $f:U \to V$
is measured by the {\it modulus} of the 
{\it fundamental annulus} $V \setminus \overline {U}$,
denoted by $\mod f$. (The quality is poor if $\mod f $ is small.)
  An infinitely renormalizable map is said to have {\it
a priori bounds} if all of its renormalizations have definite quality, i.e.,
the corresponding moduli are bounded away from zero.  ({\it A priori} bounds
are equivalent to precompactness of the full renormalization orbit
$\{R^j f \}_{j \geq 0}$ in a suitable topology.)  While by no means all
infinitely renormalizable maps have {\it a priori bounds}, many do, and in
particular it is conjectured that bounded combinatorics
implies {\it a priori} bounds (which has indeed been proved whenever
the renormalization combinatorics of all the $R^j f$ are primitive \cite{K}).

Recall that  a {\it Feigenbaum map} is an infinitely renormalizable
quadratic-like map with bounded combinatorics and {\it a priori} bounds. 
%

\begin{thm}
There exists a Feigenbaum quadratic polynomial $P_c$ 
with primitive stationary combinatorics  whose Julia set $J_c$
has positive area.
\end{thm}


Our methods yields in fact an infinite family of primitive
Mandelbrot copies which have the property that all infinitely
renormalizable maps whose
renormalization combinatorics (for all the renormalizations) belong to this
family have Julia sets of positive area.  We recall that
any finite family $\cF$ of primitive Mandelbrot copies with $\#\cF \geq 2$
defines an associated {\it renormalization horseshoe} $\cA$ consisting
of all quadratic-like maps that
belong to the $\omega$-limit of the renormalization operator restricted to
those combinatorics,
 see \cite {AL-horseshoe})  (complemented with \cite{K} )  for a recent account
 of this result . 
 The dynamics of $R|\cA$ is topologically semiconjugate
to the shift on $\cF^\Z$, and the corresponding quadratic parameters in
$\chi(\cA)$ form a Cantor set
naturally labeled by $\cF^\N$.  This Cantor set has bounded geometry
by \cite {FCT}, so we can conclude:

\begin{thm} \label {hdpar}
The set of
Feigenbaum quadratic maps with Julia sets of positive area has positive
Hausdorff dimension in the  parameter space.
\end{thm}

\bignote{should be commented upon in the main body}

\subsection{What do we learn about Julia sets of positive area?}

\sss{Preamble: Area Problem}

The problem of whether 
all  nowhere dense Julia sets  have zero area  
goes back to the classical Fatou's memoirs
who gave first examples of such  Julia sets \cite{F}.%
\footnote{What Fatou showed is that if $|Df(z)|> \deg f$ for all $z\in
  J(f)$, then $J(f)$ is a Cantor set of zero length.}
In 1980-90's, broad classes of Julia sets with zero area were given  
in \cite{DAN,L-area,Sh-ICM,Yar} and \cite{U,PR,GS}. \note{Shen and van Strien?} 
 First examples of a rational maps%
\footnote{For transcendental entire functions,
a class of Julia sets of zero area was described in \cite{EL-DAN},
and examples of  Julia sets with positive area
   appeared in \cite{EL,McM-sin}.}
 (in fact, quadratic polynomials)  with nowhere dense Julia sets with positive area
 have been recently constructed  by Buff and Cheritat \cite {BC}   
in a remarkable development
that successfully brought to completion Douady's program from mid-1990's.
(See also Yampolsky \cite{Ya-posmeas} for an alternative point of view on the
final piece of their argument.)   
An  important technical input to this program was supplied by the recent
breakthrough in the Parabolic Bifurcation Theory 
by Inou and Shishikura \cite{IS}. 

The strategy carried by Buff and Cheritat  depends on a {\it Liouvillean mechanism} 
of fast rational approximation. It produces three type of examples: 
{\it Cremer, Siegel}, and
{\it infinitely renormalizable with unbounded satellite combinatorics}.\footnote {We recall
that a quadratic map with a periodic orbit ${\mathbf \beta}$  with irrationally indifferent
multiplier $e^{2 \pi i \alpha}$, $\alpha \in \R \setminus \Q$, is classified
as {\it Siegel} or {\it Cremer} according to whether it is  locally linearizable
near ${\mathbf \beta}$ or not.}

Feigenbaum Julia sets have a rather different nature, so our work
brings new light on the realm of Julia sets of positive area.

\sss{Parameter visibility}  Julia sets of positive area are supposed to be
{\it visible} objects. However,  sets of parameters produced by
the Liouvillean mechanisms (such as in \cite{BC}) tend to be tiny: 
they probably have  zero Hausdorff dimension.
(This is definitely so  in the Cremer case as the whole set of Cremer parameters
has zero Hausdorff dimension).

By our previous work \cite {AL}, Feigenbaum Julia sets of positive area are
more robust: the existence of a single Feigenbaum Julia
set of positive area inside some renormalization horseshoe
implies that there is a whole ``sub-horseshoe'' of them, restricted to which
the renormalization dynamics is topologically conjugate to an 
subshift of finite type. 
This creates  a  parameter set of positive Hausdorff dimension. 
The construction we use to proof of Theorem \ref {hdpar} is even more
precise, providing us with full renormalization horseshoes and allowing us to obtain an effective estimate:
the set of parameters $c$ such that $P_c$ is a Feigenbaum map of positive
area has Hausdorff dimension at least $1/2$.

We note that it is expected that Lebesgue almost every quadratic map
is hyperbolic, and hence has a Julia set of not only zero Lebesgue measure
but even of Hausdorff dimension less than two.  It is unclear whether
the set  of all complex Feigenbaum parameters
has Hausdorff dimension strictly less than two.%
\footnote{The real analogue of this statement is known  to be true \cite {AM}.}  
At the moment, it is only
known that the Hausdorff dimension of these parameters is at least 1
\cite{L-dim}.  

\sss{Poincar\'e series and Hausdorff dimension}
 The notion of  {\it Poincar\'e series} 
was transferred from the theory of Kleinian groups to Holomorphic Dynamics by 
Sullivan \cite{S-conf meas},
and it became an efficient tool in the study of Hausdorff
dimension of Julia sets.
Previously to this work, in all known cases the Hausdorff dimension of 
rational  Julia sets coincided with the critical exponent
of the Poincar\'e series (see \cite{U,PR,GS} and \cite{AL}).  On the
other hand, it was shown in \cite {AL} that equality must break down in the
case of a Feigenbaum map with periodic
combinatorics and positive Lebesgue measure Julia set.%
\footnote{More recently, such a  phenomenon was also observed in the
  transcendental dynamics \cite{UZ}.}  

The critical exponent does coincide with the {\it hyperbolic dimension} for
all Feigenbaum Julia sets (and indeed for all known  cases of rational maps),
 so our examples display a definite gap between the Hausdorff dimensions of the Julia set
and of its hyperbolic subsets. It is conceivable, however, that 
{\it for Julia sets of zero area, the critical exponent, Hausdorff dimension
and hyperbolic dimension, are all equal}  (without any further
assumptions on the rational map.)


\sss{Positive measure vs non-local connectivity}
  There was a general feeling that these two phenomena are tightly
  linked as the examples constructed by Buff and Cheritat
  are probably all non-locally connected. 
(Note, in particular, that  Cremer Julia sets are never locally
connected). 
On the other hand,   all Feigenbaum Julia sets have well behaved geometry and
  in particular are locally connected, see \cite{HJ, McM1}.
Note that local connectivity make a Julia set {\it topologically
  tame}: it admits an explicit topological model (see \cite{D-model}).
Thus,  our examples show that positive area is compatible with
topological tameness.

Related to this issue is the fact that all Feigenbaum Julia sets constructed
here have primitive combinatorics, while the previously known infinitely
renormalizable examples have satellite combinatorics.  In fact all known
examples of infinitely renormalizable maps with non-locally connected Julia
set have satellite combinatorics.  


\comm{
\sss{Holomorphic Dynamics vs Hyperbolic Geometry}
Feigenbaum Julia sets have clear analogues
in the world of hyperbolic geometry: they
are related to hyperbolic manifolds that asymptotically fiber over the
circle.  As the limit set of such a manifold has zero
area \cite{Th}  (as long as it is nowhere dense), there was anticipation that the same is true for
Feigenbaum maps (see \cite{McM2}, ....) However, the dictionary
breaks down at this point.

\sss{Satellite vs primitive} Infinitely renormalizable maps of
\cite{BC} are all of {\it satellite type} (see \S \ref{}), which is 
more susceptible to wild phenomena. 
In our examples the renormalization is {\it primitive}.
}

\sss{Wild attractors and ergodicity}
  The measure-theoretic dynamics on Feigenbaum Julia sets of positive area
is  well understood.  In particular
it is {\it ergodic} with respect to the Lebesgue measure \cite{P},
and there is a uniquely ergodic Cantor attractor
$A\subset J(f)$ (of Hausdorff dimension
strictly less than two) that attracts almost
all orbits in the Julia set, see \cite{DAN}.  Moreover, almost all orbits are
equidistributed with respect to the canonical measure on $A$.
Such a structure was known to exist in  the real dynamics \cite{BKNS} 
but not in the complex  rational dynamics.%
\footnote{See, however,  \cite{L-exp,R}  for a related  phenomenon 
 in transcendental dynamics.}

By contrast, the description of the measure-theoretical dynamics of the examples
obtained in \cite {BC} is less developed.  While it is known that there exists 
an attractor $A$ properly contained in $J(f)$ such that $\om(
z) \subset  A$
for almost all $z\in J(f)$, see \cite{DAN,IS,Ch}, the
structure of $A$ is not  fully understood. It is also unknown whether the
Lebesgue measure on $J(f)$ is ergodic under the dynamics.

\sss{Sullivan's Dictionary}\label{Sullivan's Dictionary}

A parallel spectacular development in the problem of area and Hausdorff
dimension has happened in the Theory of Kleinian groups.
However, the outcome appeared to be quite different. 
In mid 1990's,  it was proved by Bishop and Jones \cite{BJ} 
that the limit set $\La=\La(\Gamma)$ of  a (finitely
generated) Kleinian group $\Gamma$  has full Hausdorff dimension if
and only if the group is geometrically infinite. As geometrically finite
groups correspond to hyperbolic  or parabolic rational maps,
we see that the answer for Kleinian groups is much simpler. 

As the {\em area} is concerned, it had been the subject of the long-standing {\it Ahlfors
  Area Conjecture} asserting that any limit set $\La(\Gamma)$  has zero area
as long as it is different from the whole sphere. 
Through the work of Thurston \cite{Th}, Bonahon \cite{Bo} and Canary \cite{Ca}, 
this conjecture was reduced to {\it Marden's Tameness Conjecture},
and the latter was recently proved
by Agol \cite{Ag} and Calegary-Gabai \cite{CG}.  
Thus, there are no non-trivial limit sets $\La$ of positive area:
again, the situation is for Kleinian groups is much  more definite
compared with rational maps. 

It does not mean, however, that Sullivan's Dictionary between Kleinian
groups and rational maps completely
breaks down at this point.  Kleinian groups belong
to a special class of {\it reversible} dynamical systems:
the corresponding geodesic flow on the hyperbolic 3-manifold
$M_\Gamma$  admits a nice
involution that conjugates it to the inverse flow. 
The analogous flow for a rational map $f$ lives  on the hyperbolic
3-lamination $\HH_f$ constructed in \cite{LM}. However, this flow is not
reversible,
which  reflects  the {\it unbalanced} property 
(see the next section) of the underlying maps
and bears  responsibility for richer geometric properties of
Julia sets.

\subsection{Basic Trichotomy}
To put our result  into deeper perspective, 
let us briefly recall the basic trichotomy  of  \cite{AL}. 
Consider the following alternative for Feigenbaum maps:

\ssk{\it Lean case:} $\HD(J(f)) < 2$;

\ssk{\it Balanced case:} $\HD(J(f)) =2$ but $\area J(f)=0$;

\ssk {\it Black hole case:} $\area J(f) >0$.

\ssk
In that paper, we showed that if a periodic point of renormalization is
either of  Lean or Black hole type, then this can be verified
``in finite time'', by estimating some geometric quantities associated to some
(not necessarily the first) renormalization of $f$.  Namely,  let us
define two parameters:
\begin{itemize}
\item $\eta_n$ gives the probability
for an orbit starting in the domain of $f$
to enter the domain of the $n$-th pre-renormalization,

\item $\xi_n$ gives the probability that an orbit starting in
the domain  of the $n$-th pre-renormalization will never come back to it.
\end{itemize}
We showed that in the Lean case $\eta_n / \xi_n \to 0$
exponentially, in the Black hole case $\eta_n / \xi_n  \to \infty$
exponentially, and that in the Balanced case $\eta_n / \xi_n$
remains bounded away from zero and infinity.  Moreover, there is an
effective constant $C>1$ (given in terms of some rough
geometric parameters, like $\mod f$)  
such that
if $R^n f=f$ then 
\begin{itemize}
\item $\eta_n / \xi_n >C $ implies the Black hole case,
\item $ \eta_n / \xi_n < C^{-1}  $ implies the
  Lean case.
\end{itemize}

Regarding the Balanced case, Theorem 8.2 of
\cite {AL} asserts that the existence
of both Lean and Black hole Feigenbaum maps inside some renormalization
horseshoe implies that {\it there exist some Balanced Feigenbaum maps} in
this horseshoe, 
but the construction does not yield a renormalization periodic point.  
In fact, in seems unlikely that 
Balanced  maps with periodic combinatorics exist  (the geometric
parameters would be  too fine tuned for it to happen ``by chance''
given that there are only countably many periodic points
of renormalization).\footnote {See also the discussion in \cite {AL} about
a related problem for real maps of the form $x \mapsto |x|^\alpha+c$:
therein one can vary the degree $\alpha$ of the
critical point continuously to fine-tune the parameters, so the
corresponding Balanced case is believed to exist (and a conditional proof is
given, subject to a Renormalization Conjecture), but it is unlikely that the
fine tuned degrees would ever correspond to an integer number (i.e., to a
polynomial).}  

\subsection{Strategy}

As discussed above, \cite {AL} gives a probabilistic criterion for the
Black Hole property to hold for a fixed point of renormalization: it
suffices to check that $\eta_n / \xi_n $ is sufficiently large for
some $n$.  Below we will use this only in the particular case $n=1$: We will
produce a sequence of fixed points of renormalization $f_m:U_m \to V_m$
such that $\inf \eta(m)>0$ while $\lim \xi(m)=0$,
where $\eta(m)=\eta_1(f_m)$ and $\xi(m)=\xi_1(f_m)$.  We will also verify
that the rough initial geometry of the fundamental annuli
$V_m \setminus \overline
U_m$ remains under uniform control.  Since the ``constant to beat'' in the
criterion only depends on such a control, this will show that for $m$
sufficiently large the criterion is satisfied so that
the Julia set of $f_m$ has positive Lebesgue measure.

It is easy to see that if the sequence
$\chi(f_m)$ converges to a parameter $c$ for which $\area K(P_c)=0$, and the
rough initial geometry remains under control, then
$\eta(m) \to 0$.  Given this observation, it is natural to
consider sequences of renormalization combinatorics which approach
a parameter $c$
with either a Siegel disk or a parabolic point.  In our argument,
we will take $c$ to have a Siegel disk of bounded type.  
One still has to select the combinatorics very carefully, 
and a number of natural options we had initially tried  had either displayed 
degeneration of the geometry (for instance, with growing modulus of the fundamental
annulus), or could not be treated in  a definitive way  without 
computer assistance.

We now describe the idea more precisely.
Let us consider a  quadratic polynomial $P_c$ that has a
Siegel disc $S$ with rotation number $\theta= [N,N,\dots]$, $N$ being big
enough. Let $p_m/q_m= [N,\dots, N]$ be the continuous fraction
approximands to $\theta$, and let $P_{c_m}$ be the corresponding
quadratic maps with a parabolic fixed point
with rotation number $p_m/q_m$. 
We perturb $c_m$ within the $(p_m/q_m)$-limb (the connected
component of $\MM \setminus \{c_m\}$ not containing $0$) to a
Misiurewicz map $P_{a_m}$, i.e., one for which the critical orbit is
eventually periodic, but not periodic.  Then we further perturb $a_m$
to a superattracting parameter $b_m$.  This parameter  
is the center of some maximal primitive Mandelbrot copy $\MM_m$.

Let $f_m: U_m\ra V_m$ be the corresponding renormalization  fixed points
with  stationary combinatorics $\MM_m$. To control the dynamics
of these maps in what
follows, we need a good control of the postcritical set after all the
perturbations. 
This has also been crucial in Buff and Cheritat's work \cite {BC}, who proved
using  the Inou-Shishikura renormalization theory \cite{IS} (which currently
is only available for large $N$, hence the choice above), that the
postcritical set of $P_{c_m}$ stays in a small neighborhood of the Siegel
disk $S$.  Our further choice of $a_m$ and $b_m$
is in part designed to keep this
property for the further perturbations: In particular excursions of the
critical orbit away from the Siegel disk must be prevented to avoid
excessive expansion (which would again lead to 
growing fundamental annuli).  
Thus,  the periodic orbit on which the critical point eventually lands must be
taken quite close to the Siegel disk.  The most natural choice would be the
periodic orbit with combinatorial  rotation number $p_m/q_m$ that  arises from the
bifurcation of $P_{c_m}$, but for technical implementation reasons we
actually use some orbit of rotation number $p_{m-\kappa}/q_{m-\kappa}$,
for some big but bounded (as $m\to \infty$)  $\kappa$ (so that the critical point still only goes a
bounded number of levels up in terms of
the cylinder Siegel  renormalization).

We then fine tune the superattracting parameter $b_m$ to get a suitable
control on the initial geometry of the first renormalization.  While
we want the moduli of fundamental annuli to remain bounded,
we'd like  them  to be sufficiently large to obtain control on the actual
renormalization fixed point.  Indeed, there is a ``threshold''
lower bound on the moduli of the fundamental annuli of the first
renormalization of a
Feigenbaum quadratic map with stationary primitive combinatorics,
which, once surpassed, implies uniform control for the
associated renormalization fixed point.  Below this threshold, current
techniques do not give such uniform bounds without further restrictions
(which would in particular not apply when approaching
Siegel parameters).  Thus, we make the critical orbit (after
perturbation) follow closely the periodic orbit for large but bounded 
number of turns around the Siegel disk, picking up the right amount of
expansion from the periodic orbit before drifting apart and closing.

Once the geometry of the first renormalization is controlled,
we construct a {\it safe trapping disk} $D$ that it stays away from the
postcritical set, captures all orbits that escape from the Siegel disk
$S$ to infinity and has the property that 
a definite portion of $D$  lands in the renormalization domain $U$.
Then  a direct Distortion Argument implies that the pullbacks of $U$ 
occupy a definite proportion of $S$,
which implies that the landing probability  $\eta_m$ stays bounded away from $0$.

To control the escaping parameter  $\xi_m$, we make use of
the  {\it Siegel Return Machinery}
that ensures high probability of returns back to the trapping disk,
and hence high probability of  eventual landing in the renormalization
domain $U$. 
(The Return Machinery makes use of the {\it hyperbolic expansion}
outside the postcritical set  \cite{McM1}
and that  was  also used by Buff and Cheritat  \cite {BC}).

In  this construction, there is one free parameter that can be varied
without significant impact on the geometry of the first renormalization,
which is the time the critical point spends 
in the {\it parabolic gate} created when the parabolic map $P_{a_m}$
is perturbed to the Misiurewicz map $P_{a_m}$.  
There is a  uniform control of this perturbation governed by the
limiting transit map (the {\it geometric limit}) .  Varying this time  produces  a sequence of
 Black Hole combinatorics whose Mandelbrot copies  decay
quadratically. 
Alternating these combinatorics  creates a Cantor set
of  Hausdorff dimension  $\geq 1/2$ consisting of Black Hole 
 parameters.

\msk
To carry out the above strategy, we make use of four Renormalization
Theories:

\ssk\nin $\bullet$ 
Renormalization of {\it quadratic-like maps}, including the probabilistic
criterion of \cite{AL}, is discussed in  \S \ref{ql maps} . 

\ssk\nin $\bullet$ 
Renormalization of {\it quasicritical circle maps} is developed in \S \ref{analytic circle maps}
(roughly speaking, ``quasicritical''  means that the map is allowed to
loose analyticity at the critical point, but is assumed to
be quasiregular over there). 

\ssk\nin $\bullet$ 
 {\it Siegel} renormalization theory based upon renormalization of
 quasicritical circle maps is laid down in \S \ref{Siegel maps sec}.

\ssk\nin $\bullet$ 
Finally, is \S \ref{IS sec} we briefly discuss the  {\it parabolic}
renormalization,
and particularly,  the  {\it Inou-Shishikura Theory}. 

\ssk
With these renormalization tools in hands, 
we proceed to  the main construction 
(\S \ref{construction sec}).

\comm{
\begin{rem}
  A very delicate part of \cite{BC} is based on Cheritat's thesis where
   a controlled perturbation of $S$ is performed to create a new
   Siegel  disk $S'$ (which much higher combinatorics) that occupies
   about $1/2$ of $\area S$.   In our
   scenario, this is replaced with a much more straightforward  step:
   ensuring that the pullbacks of
the renormalization domain occupy a definite proportion of $S$, which
is done by a direct distortion estimate.
\end{rem}
}

\subsection{ Basic Terminology and  Notation}  \note{check consistency}
$\N_0= \{ 0,1,\dots\}$, $\N\equiv \N_1= \{ 1,2,\dots\}$,\\ 
and in general,  $\N_\kappa= \{ n\in \N:\ n\geq \kappa\}$;\\
$\bar \N_\kappa  = \N_\kappa\cup \infty$ (with the natural topology);\\ 
$\C^*=\C\sm \{0\}$; \\  
$\D_R(a) =\{z: \ |z-a|< R \}$; $\D_R=\D_R(0)$,  $\D=\D_1$;\\
Notation $\T$ will be used for both the unit circle in $\C$ and its
angular parametrization by  $\R/\Z$;\\
$\area$ refers to the Lebesgue measure;\\
\newcommand{\Comp} {\mathrm{Comp}}
For a set $Z \subset \C$ and a point $z\in Z$,
we let $\Comp_z (Z)$  be the component of 
$ Z$ containing $z$.\\
For  a topological annulus $A\Subset\C$, we let $\di^o A$ and 
$\di^i A$ be its {\it outer} and {\it inner} boundaries.

\ssk\nin
$\Dom f$ is the domain of a map $f$;\\
$\orb z=\orb_f z$ is the forward orbit of a point $z$;\\
$\OO_f$ is the postcritical set of a map $f$, i.e., the closure of the
orbit of its critical point;\\
$\Bf_\theta: z\mapsto e^{2\pi i \theta} z + z^2$, $\theta\in \C/\Z$; \\
$\boldsymbol{\FF} = (\Bf_\theta)_{\theta\in \C}$ is the quadratic family;\\ 
$\MM$ is the Mandelbrot set.

\ssk \nin
By saying that some quantity, e.g. $\eta$,  depending on parameters is {\it definite},
we mean that 
$\eta \geq \eps>0$ where $\eps$ is independent of the parameters
(or rather, it may depend only on  some, explicitly specified,
parameters).
 By saying that a set $K$ is {\it well inside a domain $D\Subset \C$}
we mean that $K\Subset D$  with a definite $\mod (U\sm K)$. 
The  meaning of expressions {\it bounded}, {\it comparable}, etc. is similar.
If we need to specify a constant then we say  ``$\eps$-definite'',
``$C$-comparable ($\asymp$)'', etc. 

\ssk\nin
Given a pointed domain $(D, \beta)$,
we say that  {\it $\beta $ lies  in the middle of  $D$},
or equivalently, that {\it  $D$ has a bounded shape around $\beta$}
  if 
$$
     \max_{\zeta\in \di D} |\beta- \zeta| \leq C\, \min_{\zeta\in \di
       D} |\beta- \zeta|,
$$ 
where $C$ is a constant that may depend only on specified parameters.




\section{Quadratic-like maps} \label{ql maps} 

\subsection{Basic definitions}

A {\it quadratic-like map}  $f: U\ra V$ \cite{DH},
which will also be abbreviated as a {\it q-l map}, 
  is a holomorphic double branched covering between two Jordan disks
$U\Subset V\subset \C$.  
It has a single critical point that we denote $c_0$. The 
annulus $A= U\sm \bar V$ is called the {\it fundamental annulus} of $f$. 
We let $\mod f : = \mod A$. 
The {\it filled Julia set} $K(f)$
is the set of non-escaping points:
$$
   K(f)= \{ z:\,  f^n z\in U, \ n=0,1,2,\dots \}
$$
Its boundary is called the {\it Julia set} $J(f)$. The (filled) Julia set  is either connected or Cantor,
depending on whether the critical point is non-escaping (i.e., $c_0\in K(f)$) or  otherwise. 


Two quadratic-like maps $f: U\ra V$ and $\tl f: \tl U\ra \tl V$ are called {\it hybrid equivalent}
if they are conjugate by a quasiconformal map $h: (V,U) \ra (\tl V, \tl U)$ such that
$\dibar h= 0$ a.e. on $K(f)$.  

A simplest example of a quadratic-like map is provided by a quadratic polynomial $P_c: z\mapsto z^2+c$
restricted to a disk $\D_R$ of sufficiently big radius. The  Douady and Hubbard {\it Straightening Theorem} asserts that 
any quadratic-like map $f$ is hybrid equivalent to some restricted quadratic polynomial $P_c$. 
Moreover, if $J(f)$ is connected then the parameter $c\in \MM$ is unique.

As for quadratic polynomials, two fixed points of a quadratic like maps with connected Julia set
have a different dynamical meaning. One of them, called $\beta$, is 
 the landing point of a proper arc $\gamma\subset U\sm K(f) $  such that $f(\gamma)\supset \gamma$.
It is either repelling or parabolic with multiplier one. The other fixed point, called  $\alpha$, is either 
non-repelling or a {\it  cut-point}  of the Julia set.

\sss{Quadratic-like families}\label{ql families}
A {\it quadratic-like family} $\Bbb F= (F_\la: U_\la\ra V_\la ) $ over a parameter
domain $\La\subset \C$ is a family of quadratic-like maps $F_\la$ 
holomorphically depending on $\la$. The latter means more precisely
that  the set
$$
   \Bbb U =  \bigcup_{\la\in \La} U_\la\ 
$$
is a domain in $\C^2$ and the function  $f_\la(z)$ is holomorphic on
$\Bbb U$.
Let us normalize it so that $c_0=0$ for all $f_\la$.
The associated Mandelbrot set is defined as 
$$
   \MM_{\Bbb F} = \{ \la\in \La: \  J(F_\la) \ \mathrm{is \
   connected} \}.  
$$ 

Let us select a base point $\la_\circ$ and let $U_\circ\equiv
U_{\la_\circ}$ etc.
We say that a quadratic-like family $\BF$ is {\it equipped} if 
there is a holomorphic motion 
$$ 
    h_\la:  \bar V_\circ \sm U_\circ \ra \bar V_\la \sm U_\la
$$ 
of the fundamental annulus $\bar V_\la\sm U_\la $  over the pointed
domain  $(\La, \la_\circ)$ which is {\it equivariant}  on the boundary
of the annulus, i.e., 
$$
        h_\la (f_\circ(z))  = f_\la (h_\la (z)), \quad z\in \di U_\circ. 
$$

An equipped quadratic-like family $\Bbb F$  is called {\it proper} if 
$f_\la(0) \in \di V_\la$ for $\la\in \di \La$ (which assumes implicitly
that the family $f_\la$ is continuous up to  $\di \La$). 

A quadratic-like family  $\Bbb F$ is called {\it unfolded} if the
curve  
$$
  \la\mapsto f_\la(0), \ \la\in \di \La, 
$$
has winding number 1 around $0$. 

\begin{thm}\cite{DH}\label{M-sets of ql families}
For any equipped proper unfolded quadratic-like family $\Bbb F$,
the Mandelbrot set $\MM_{\Bbb F}$ is canonically homeomorphic to the standard 
Mandelbrot set $\MM$.
\end{thm}

The proof can be also found in \cite{L-book}. 

\subsection{Renormalization}

A quadratic-like map  
$f: U\ra V$ is called {\it DH renormalizable} (after Douady and Hubbard) if there is a quadratic-like
restriction 
$$
         Rf\equiv R_{DH} f = f^\per: U'\ra V'
$$ 
 with connected Julia set $K'$ such that
the sets $f^i(K')$, $k=1,\dots, \per-1$,  are either disjoint from $K'$
or else touch it at its $\beta$-fixed point.%
\footnote{See \cite{McM1} for a discussion of this condition.} 
In the former case the renormalization is called {\it primitive}, while in the latter it is called {\it satellite}.     

The map $Rf: U'\ra V'$ is called the {\it pre-renormalization} of $f$.
If it is considered up to rescaling, it is called  the 
{\it  renormalization} of $f$.  

The sets $f^i(K')$, $k=0,\dots, \per-1$, are referred to as  the {\it little (filled) Julia sets}.
Their positions in the big Julia set $K(f)$ determines the renormalization {\it combinatorics}.
The set of parameters $c$ for which the quadratic polynomial $P_c$ is renormalizable
with a given combinatorics forms a {\it little Mandelbrot copy}
$\MM'\subset \MM$. In fact, the family of renormalizations
$R(P_c)$, $c\in \MM'$, with a given combinatorics can be included in
a quadratic-like family $\Bbb F= (f^\per : U_c\ra V_c)$ over
some domain $\La\supset \MM'$ so that $\MM'= \MM_{\Bbb F}$.
A natural base point $c_\circ \in \MM'$ in this family is the superattracting
parameter with period $\per$. 
It is called the {\it center} of $\MM'$.
Any superattracting parameter
in $\MM$  with period $\per>1$ is the center of some Mandelbrot  copy
$\MM'$ like this. 
Moreover, in case of primitive combinatorics the quadratic-like family
$\Bbb F$ is proper and unfolded. 
(See \cite{DH,D-ICM,L-book} for a discussion of all these facts.)

We can encode the renormalization combinatorics by the corresponding
copy $\MM'$ itself.
Equivalently, it  can be  encoded  by the center $c_\circ$ of $\MM'$
or the corresponding  Hubbard tree $H'$. 

A little Mandelbrot copy is called {\it primitive} or {\it satellite} depending on  the type of the corresponding renormalization.
They can be easily distinguished as any satellite copy is attached to
some hyperbolic component of $\inter \MM$
and does not have the cusp at its root point. 

For infinitely renormalizable maps,
notions of {\it stationary/bounded combinatorics},
 {\it a priori} bounds, and {\it Feigenbaum maps}  were defined in the Introduction
(\S \ref{Feigen maps}).  We say that a Feigenbaum map is {\it
  primitive}  if all it renormalizations are such.  

One says that a family $\FF$  of Feigenbaum maps
(e.g.,  for the family of  maps with a given combinatorics)
has {\it beau bounds}
if there exists $\mu>0$ such that for any $f\in \FF$ we have
$$
    \mod R^n f \geq \mu\ \mathrm{for\ all }\ n\geq n(\mod f ). 
$$ 
It was proved by Kahn \cite{K} that {\it primitive Feigenbaum  maps
  have beau bounds}, with $\mu$ depending only on the combinatorial
bound.  In fact,  $\mu$  can be made uniform over  some class of combinatorics \cite{KL}.%
\footnote{ We will not use these  results as the  combinatorics
  we construct do not fall into the class \cite{KL}.
On the other hand, {\it beau bounds} can be easily supplied for our class.}

The {\it renormalization fixed point } $f_*$ is a quadratic-like map
  which is invariant under renormalization: $Rf= f$.
 In terms of the {\it pre}-renormalization, there exists a {\it scaling
   factor}  $\la\in \C\sm \bar \D$ such that 
$$
                Rf (z) =  \la^{-1} f(\la z). 
$$

\begin{thm}\label{renorm fixed point}
  For any stationary bounded combinatorics with a beau bound,
there exists a unique renormalization fixed point   $f_*$ with this
combinatorics.
Moreover,  $\mod f_*\geq \mu$, where $\mu>0$ is the beau bound.   
\end{thm}
                         \note{needed?}

This theorem was originally proved by Sullivan \cite{S}.
Other proofs were given by McMullen \cite{McM2}, 
and recently, by the authors \cite{AL-horseshoe}.


\subsection{Probabilistic criterion for positive area}

Let us now introduce precisely probabilistic parameters $\eta$ and
$\xi $ mentioned in the Introduction. Let $f: U\ra V$ be a Feigenbaum
map with {\it a priori} bound $\mu>0$ (i.e., $\mod R^n f\geq \mu$ for all
$n\in \N$), and let $Rf: U'\ra V'$ be its first pre-renormalization, 
$A'= U'\sm \bar V'$ be the  corresponding fundamental annulus. 

The  {\it landing    parameter} $\eta$ is the probability of  landing
in $U'$. Precisely, 
let $\displaystyle{ \XX= \bigcup_{n\in \N} f^{-n} U' }$ be the set of
  points in $U$ that  eventually land in $U'$.   Then
\begin{equation}\label{landing par}
    \eta  = \frac {\area \XX} {\area U }.
\end{equation}                                                   

The {\it  escaping parameter} $\xi$ is the probability of escaping
from the fundamental annulus $A'$. Precisely, let $\YY$ be the set of
points in $A$ that never return back to $V'$:
$$
      \YY= \{ z\in A': \ f^n z \not\in V'\ \mathrm{for} \  n \geq 1\
      \mbox {(as long as $f^n z$ is well defined)}.  
$$
Then
\begin{equation}\label{escaping par}
         \xi = \frac {\area \YY} {\area A'} .  
\end{equation}

The following result asserts that if the landing
probability is much higher then the escaping one,
then the Julia set has positive area.  

\begin{thm}[Black Hole Criterion  \cite{AL}] \label{black hole}
  There exists $C=C(\mu)$ with the following property.
Let $f$ be a primitive Feigenbaum map with stationary combinatorics and {\it a priori bound $\mu$}.
If $\eta\geq C \xi$ then $\area J(f)>0$. 
\end {thm}

\section{Quasicritical circle maps}\label{analytic circle maps} 

An  {\it (analytic) critical circle map } is an analytic  homeomorphism $f: \T\ra \T$ of the circle  $\T= \R/\Z $
with   a single critical point $c_0$ of cubic type (i.e.,
$f'''(c_0)\not=0)$. It is usually normalized so that $c_0=0$ in the angular coordinate.

To study Siegel disks of non-polynomial maps we need to enlarge this class allowing the map
be only quasiregular at the critical point.

\subsection{Definitions}\label{quasicrit circle maps}  

A {\it quasicritical circle map } is a  homeomorphism $f: \T\ra \T$ of the circle  $\T= \R/\Z $
with the following properties:

\ssk\nin Q1.
$f$ is a real analytic diffeomorphism  outside  a single critical point
   $c_0$  normalized so that $c_0=0$ in the angular coordinate;
we let $c_n= f^n c_0$;

\ssk\nin Q2. Near the critical point, $f$ admits a quasiregular
extension  to $\C$  of cubic type, i.e, it has a local form%
\footnote{We could write $h(z)^\de + c_1 $ with any $\de>1$ as well
  since all the powers are quasisymmetrically equivalent.}  
$h(z)^3 + c_1 $ with a
quasiconformal $h: (\C, c_0)\ra (\C, 0)$ such that $h$ is holomorphic near $z$ whenever
 $f(z)$ lies on the same side of $\T$ as $z$.

\ssk
It follows, in particular, that $f|\, \T$ {\it is quasisymmetric}. 
Moreover, it admits a quasiregular extension to a neighborhood of $\T$,
symmetric with respect to $\T$, that  is holomorphic in the domain
$$
     \Dom^h f  = \{z\in \Dom f:  \mbox { $z$ and $f(z) $ lie on the same side
     of $\T$ } \} \cup \T \sm \{c_0\}. 
$$

\ssk
We will also make extra assumptions about exterior structure of $f$:

\ssk\nin Q3.  $\Dom f$ is a $\T$-symmetric  annulus, and
  $\Dom^h  f\sm \D$ is obtained from the outer annulus  $\Dom f\sm \D$ by
  removing a closed topological triangle
$$
  \TT=\TT_f \subset (\Dom f\sm \Bar \D)  \cup \{ c_0  \}
$$    
with a vertex at $c_0$ and the opposite side 
on   the outer boundary of  $\Dom f$;

\ssk\nin Q4.  
  $f :  \Dom^h f \ra \C$ is an immersion, and $f:  \TT \ra \D\cup \{c_1\} $ is an embedding.

\msk 
Let $\Cir$ stand for  the space of all quasicritical circle maps.
The {\it geometry} of such a map is specified by the  dilatation of the
map $h$ from Q2, and the size of $\Dom f$.
We call $f$ a $(K, \eps)$-{\it quasicritical} if $\Dil h\leq K$ and
$\Dom f $ contains the $(2\eps)$-neighborhood of $[0,1]\subset \C$.
Let $\Cir(\bar N, K, \eps) $ denote the class of
$(K,\eps)$-quasicritical circle maps
of type bounded by $\bar N$.


\subsection{Local properties near the critical point} 

\sss{ John Property} 

Let $\NN(K) $ stand for the class of  $K$-quasiregular {\it normalized}
  maps $F : (\C, \R,0, 1) \ra (\C, \R, 0, 1 )$ 
such that
$F(z)= H (z)^3$, where $H : (\C, \R,0, 1) \ra (\C, \R, 0, 1)$ is a
$K$-qc homeomorphism. 

\begin{lem}\label{John property}
  Let  $F \in \NN(K) $,   and
let $S_\pm= F^{-1} (\C\sm \R_\pm)$, where $ F^{-1}$ is the branch of the
inverse map preserving $\R_\mp$. Then 
$$
     S_+\supset  \{  | \arg z -\pi| \leq \alpha \pi \}, \quad S_-\supset  \{  | \arg z | \leq \alpha \pi \},
$$   
where $\alpha>0 $ depends only on $K$.    \note{switch $\pm$}

\end{lem}

\begin{proof}
We will deal with $S_+$ only, as the argument for $S_-$ is the same. 
The inverse branch  $F^{-1} : \C\sm \R_+ \ra S_+ $ is the composition of $z\mapsto  z^{1/3} $ with
$H^{-1}$, so 
$$
       S_+ = H^{-1}  (T_+), \quad {\mathrm{where}}\   T_+= \{ | \arg z - \pi| < \pi/3\}.
$$  
Since $H^{-1}: (\C,0,1) \ra (\C$,0,1) is a normalized $K$-qc map, it is
$L(K)$-quasisymmetric on the whole plane.  

For any $\zeta\in \R_-$, we have: $\dist (\zeta, \di T_+) = (\sqrt{3}/2)
\, |\zeta|$. 
Take any  $z\in\R_-$ and let $\zeta= H(z)$.
By  definition of  $L_0$-quasisymmetry, we have
 $$
  \frac {\dist(z, \di S_+) }{|z|} <   \frac 1L \cdot  \frac {\dist (\zeta,
    \di T_+)} {|\zeta|} <  \frac {\sqrt{3}} {2L} ,
$$ 
with some $L$ depending only on  $L_0)$ (and on  $\sqrt{3}/2$).
The conclusion follows.
\end{proof}

Any quasicritical circle map $f\in \Cir (K,\eps) $ can be 
non-dynamically normalized 
without changing its dilatation so that it fixes $0$ and $1$,
Namely for any $t\in (0,1/2)$, let 
\begin{equation}\label{normalization}
     F_t: (\C, \R, 0, 1) \ra (\C, \R, 0, 1) , \quad       F_t (x)=
     \frac { f( t x) -c_1 }  { f (t) - c_1 }.
\end{equation}
Then it can be modified outside the $\eps$-neighborhood of $[0,1]$
to turn it into a map of class $\NN(K')$ with some $K'=K'(K,\eps)$. 
Let us call the modified map $\hat F_t$. 
Applying to it the previous lemma, 
we immediately obtain: 

\begin{prop}\label{local sector} 
  For any quasicritical circle map $f\in \CC(K,\eps)$, the domain $\Dom^h f$
  contains  local sectors
$$
    S_+(f)=   \{  | \arg z -\pi| \leq \alpha \pi, \ |z|< \eps  \}
    \quad {\mathrm {and}} \quad  S_- (f)=   \{  | \arg z| \leq \alpha \pi, \ |z|< \eps  \}
$$
with some $\alpha>0 $ depending only on $(K,\eps)$.
\end{prop}

\sss{Normalized Epstein class and scaling limits}


We say that a map $F\in \NN(K)$ belongs to the {\it Normalized Epstein class}
$\EE\NN(K)$ if the inverse maps $F^{-1}|\, \R_\mp$ admits a conformal  extension 
to $\C\sm \R_\pm\ra \C\sm \R_\pm$.

\begin{lem}\label{scaling limits}
Let $f$ be a quasicritical circle map. 
 Then the  family of modified rescalings $\hat F_t$, $t\in (0,1/2)$
 defined after \rm { (\ref{normalization})  }    is precompact in the
 uniform topology on $\hat \C$.
All limit maps as $t\to 0$ belong to the normalized Epstein class $\NN\EE(K)$,
with $K$ depending only on the geometry of $f$. 
\end{lem}

\begin{proof}
The modified rescalings  $\hat F_t$ form a precompact family since 
 they are normalized degree three  uniformly quasiregular maps.
Moreover, the inverse maps  $F_t^{-1}|\, \R_\mp$ admit a conformal
extension to  $(\C\sm \R_\pm)\cap \D_{\de/t}$, with some $\de$
depending on the geometry of $f$.  Hence in the  limit we obtain a map
of Epstein class.
\end{proof}

\sss{Schwarzian derivative}

We will now show that quasicritical circle maps have negative
Schwarzian derivative near the critical point.
Let us begin with  maps of Epstein class:

\begin{lem}\label{neg S in Epstein class}
  Any map $F \in \NN\EE$ has negative Schwarzian derivative on the
  whole punctured line  $\R\sm \{0\}$.   
\end{lem}

\begin{proof}
Let us consider an open interval $I=(a,d) \subset \R\sm \{0\}$ as a Poincar\'e model of the
hyperbolic line.  Given a subinterval $J=(b,c) \Subset I$,
let  
\begin{equation}\label{real hyp metric}
   |J: I| = \log \frac{(c-a)  (d-b) } { (b-a) (d-c) } 
\end{equation}
stand for its hyperbolic length. Condition of negative Schwarzian
derivative for $F$ is equivalent to the  property that $F^{-1} $ is a
hyperbolic contraction:
$$
       | F^{-1} (J) : F^{-1}  (I) | \leq | J: I| 
$$ 
for any pair of  intervals $I$ and $I$ as above.

Let us now consider the slit plane $\C(I):= \C\sm (\R\sm I) $ endowed with its
hyperbolic metric. Then  $I$ is a hyperbolic geodesic in $\C(I)$. 
Let $\D(I)$ be the round disk based upon $I$ as a diameter. 
It is the hyperbolic neighborhood of $I$ in $\C(I)$ of certain radius $r$  
independent of $I$. 

If $F$ belongs to the Epstein class then the inverse map 
$F^{-1} : I\ra I'$  (where $I'= F^{-1} (I)$) extends to a holomorphic map
$  F^{-1} : \C(I) \ra \C(I')$. 
By the Schwarz Lemma, it is a hyperbolic contraction.
 Since $F^{-1} (I) = I'$,   we conclude that  $F^{-1} (\D(I)) \subset
 \D(I')$. Applying the Schwarz Lemma again, we obtain 
that $F^{-1} : \D(I) \ra \D(I')$ is contracting with respect to the hyperbolic metric
in these disks. 
Since  the hyperbolic metrics on $I$ and $I'$ are induced by the
hyperbolic metrics in the corresponding disks, we are done.  
\end{proof}

\begin{rem}
  In fact, in the applications to the distortion bounds, 
the  contracting property for  the cross ratios from 
(\ref{real hyp metric}), rather than the Schwarzian derivative,  
is directly used (see Theorem \ref{Koebe control}).
\end{rem}

\begin{prop}\label{neg S}
  Any quasicritical circle map $f\in \Cir(K,\eps)$ has negative
  Schwarzian derivative in  $\de$-neighborhood of the
  critical point, where $\de= \de(K,\eps)$ depends only on the
  geometry of $f$.
\end{prop}

\begin{proof}
  By Lemma \ref{scaling limits},  the  modified rescalings $\hat F_t$ accumulate on
  a compact set $\KK\subset \NN\EE(K') $  of normalized Epstein maps,
  with $K'=K'(K,\eps)$.  
 By Lemma \ref{neg S in Epstein class}, the latter have negative Schwarzian derivative.
By Proposition \ref{local sector},  the maps $\hat F_t$ are eventually (for $t<t_0(K,\eps)$)
 holomorphic in  definite
sectors  $\{ |\arg z| <\alpha \pi\}  \cap \D$ and $\{ |\arg z-\pi|
<\alpha \pi \}\cap \D$. 
 It follows that $S \hat F_t \to S F\in \KK $ uniformly on $\pm [1/2, 1]$,
and hence the Schwarzian derivatives $SF_t$ are eventually negative on
these  two intervals.
By the scaling properties of the Schwarzian, we have:
$ S\hat F_t (x) = t^2 Sf (tx) $, and hence  $S f <0$ on some punctured
interval $[ -\de, \de]$, with $\de>0$ depending only on the geometry   \note{define}
of $f$.    
\end{proof}

\sss{Power expansion}


Let us consider a map $F\in \NN\EE$ of Normalized Epstein class, and
let 
$
   \Dom^h F= \{z: \  (\Im  z) \cdot (\Im F(z)) >0\}.
$
Recall from Lemma \ref{John property} that it consists of two 
disjoint topological sectors $S_\pm$  with the axes $\R_\mp$ mapped conformally
onto $ \Hyp \sm \R_\pm$ respectively.  Let us slightly shrink these
sectors, namely for $\beta\in (0,1)$, let 
$$
       S_+(\beta)  = \{ z\in S_+:  |\arg F(z)| > \beta \pi\} ,\quad
       S_-(\beta) =   \{ z\in S_-:  |\arg F(z) | < (1-\beta) \pi\}. 
$$


\begin{lem}\label{qr contraction}
Let us consider a  map $F\in \NN\EE(K)$ of Normalized Epstein class,
and let $\beta\in (0,1)$.
Then  
$$
   |F(z) \geq C|z|^{1+\si} \quad {\mathrm{for}} \   
    z\in S_\pm (\beta), \ |z|\geq 1.
$$
 where  $\si >0$ and $C>0$ depend only on $K$ and $\beta>0$. 
\end{lem}

\begin{proof}
Since $S_+$ contains the sector  
$\{ |\arg z -\pi| < \alpha\pi \}$, we have: 
$$
  S_- \subset |\arg z| < (1-\alpha) \, \pi \}.
$$
Hence the inverse branch $F^{-1}: \Hyp \sm \R_- \ra S_-$ can be
decomposed as $ \phi (z)^{1-\alpha} $, where $\phi: (\Hyp, \R, 0, 1)\ra
(\Hyp, \R,  0, 1) $ is a conformal embedding. 
For such a map, we have:
\begin{equation}\label{linear expansion}
 |  \phi(z)| \leq A |z|\quad \mathrm{ as \ long \ as }\  |z|\geq 1,\  |\arg z|
  < \pi(1 - \beta),
\end{equation}
  where $A$  depends only on $\beta>0$. 
Indeed,  the hyperbolic distance (in $\C\sm \R_-$) from  $z$ as above 
  to $1$ is     
$\log |z|+O(1)$ (note that by the scaling invariance, the hyperbolic distance from $z$ to $|z|$ depends
only on $\arg z$) . Since $1$ is fixed under $\phi$, 
the Schwarz Lemma implies (\ref{linear expansion}).
The conclusion for  $F$ on $S_-$ follows.

The argument for $S_+$ is similar, except $-1$ is not the fixed point
any more. But since $F$ is quasiregular, $|\phi(-1)| \asymp 1$,
and the Schwarz lemma   implies the assertion again.
\end{proof}

\subsection{Real geometry}

Due to the above local properties, 
  quasicritical circle maps enjoy the same geometric virtues as usual 
 analytic critical circle maps. The main results formulated below are
 proven in a standard way,  see e.g., 
the monograph by de Melo and van Strien \cite{MvS} for a reference.

\sss{ Koebe Distortion Bounds}

The following statement extends the usual Koebe distortion bounds to
quasicritical circle  maps:

\begin{thm}\label{Koebe control}
  Let $f\in \Cir(K,\eps)$ be a quasicritical circle map.
Let  $J\subset I\subset \R/\Z$ be two nested intervals in $\T$,
with  $I$ open. Assume that for some $n,m\in \N$, the intersection multiplicity of the
intervals $f^{-k} I$, $k=0,1,\dots, n$ is bounded by $m$
and $|f^{-k} I| < \de/2$ with $\de$ from Proposition \ref{neg S}.   \note{needed?}
Then  
$$
       |f^{-k} J : f^{-k} J | \leq C (K,\eps, m) \, | J: I|.
$$
\end{thm}

\begin{proof}
It is obtained by the standard cross-ratio distortion techniques, see
\cite{MvS}. To see  the role of various properties  of $f$, 
let us recall the main ingredients. 

\ssk \nin $\bullet$  {\it Denjoy Distortion control outside the
  $(\de/2)$-neighborhood of $c_0$}.  The distortion bound depends
  on $C^2$-norm of $f$  on $\T$  and 
  on $\displaystyle{\sum_{k\in \LL} |f^{-k} I| }$, where $\LL$ is the
  set of moments $k\leq n$ for which $f^{-k} I\cap (-\de/2, \de/2) =
  \emptyset$. The $C^2$-norm of $f$ depends only on $(K,\eps)$
 by compactness of $\Cir(K,\eps)$ and
the Cauchy control of the derivatives of holomorphic functions.
The total length of the intervals $f^{-k} I$ is bounded $m$.
 
\ssk \nin $\bullet$  
{\it Contraction of the cross-ratio in the punctured  $\de$-neighborhood of $c_0$.}
This is concerned with the moments $k\leq n$ when $f^{-k} I \subset
(-\de, \de)\sm \{0\}$. At these moments the hyperbolic  length  
$|f^{-k} J : f^{-k} I|$  is contracted under $f^{-1}$ by Proposition \ref{neg S}. 

\ssk \nin $\bullet$
{\it Quasisymmetric distortion control at the critical  moments.}
At the moments $k\leq n$ when $f^{-k} I \ni c_0$, we have:
$$
    | f^{-k-1} J : f^{-k-1} I| \leq C(H, L) \cdot | f^{-k} J : f^{-k} I | 
$$
where $L$ is an upper bound for $| f^{-k} J : f^{-k} I | $ and $H=H(K,\eps)$ is
the qs-dilatation of $f$ near $c_0$. Since the number of the critical  moments
is bounded by $m$, their contribution to the total   distortion is  bounded.
\end{proof}

\sss{No wandering intervals}

Recall that an interval $J\subset I$ is called {\it wandering} if
$f^nJ \cap J=\emptyset$ for any $n>0$. 
The above Koebe Distortion Bounds lead to the following generalization
of  Yoccoz's No Wandering Intervals Theorem  \cite{Y1}:

\begin{thm}
  A quasicritical circle map does not have wandering intervals.
\end{thm}

It follows by the classical theory (Poincar\'e's thesis) 
if $f\in \Cir$  does not have periodic points then it is topologically
conjugate to a rigid rotation 
$$
     T_\theta:  x\mapsto x+\theta\ \mod \ 1,
$$
where $\theta\in \R\sm \Q \ \mod\ \Z$ is the {\it rotation number } of $f$.

When we want to specify the rotation number of circle maps under consideration,
we will use notation $\Cir_\theta$ and $\Cir_\theta(K,\eps)$.

\sss{Bounded geometry and dynamical scales}

The further theory largely depends on the Diophantine properties of $\theta $
encoded in its continuous fraction expansion $[N1,\,  N2, \dots]$.  
Let  $p_m/q_m = [N_1,\dots, N_m]$ be the $m$-fold rational approximand to $\theta$. 
The rotation number (and the map $f$ itself) is called of {\it bounded
  type} if the entries of the expansion are bounded by some $\bar N$.
The spaces of circle map with rotation number bounded by $\bar N$
will be denoted $\Cir(\bar N)$,  $\Cir_\theta(\bar N, K, \eps)$, etc.
(depending on how many parameters we need to specify.

The Koebe Distortion Bounds also imply a more general version of the 
H\'erman-Swiatek Theorem \cite{H,Sw}: 

\begin{thm}\label{HS}
  A quasicritical circle map $f\in \Cir(\bar N, K, \eps)$ of bounded type is $H$-quasisymmetrically
  conjugate to the rigid rotation $T_\theta$, with $H= H(\bar N, K, \eps)$.     
\end{thm}

The circle dynamics naturally encodes the continued fraction expansion
of the rotation number, as the denominators $q_n$ are the {\it
  moments of combinatorially closest approaches}%
\footnote{Meaning that these are the closest approaches for the
  corresponding circle rotation $T_\theta$.}
 of the critical orbit $\{c_n\}$ back
to the critical point $c_0$. Let us consider the corresponding
intervals  $I^n=[c_0, c_{q_n}]$ (i.e.,  the combinatorially
shortest intervals bounded by $c_0$ and $c_{q_n}$). The orbits of two consecutive ones,
\begin{equation}\label{circle tilings}
    f^k(I^n), \ k=1,,\dots, q_{n+1}-1\quad  {\mathrm{and}}
     \quad   f^k(I^{n+1})  ,   \ k=1,,\dots, q_n-1 ,
\end{equation}
together with the {\it central} interval $I^n_0: = I^n \cup I^{n+1}$  
form a dynamical tiling  $\II^n$ of $\T$.  Moreover, these tilings are
nested: $\II^{n+1}$ is a refinement of $\II^n$. 

We label the intervals $I^n_k\in \II^n$, $k=1,\dots, q_n+q_{n+1}-2$,  
in an arbitrary way.
Each of these intervals is homeomorphically  mapped onto either
$f^{q_{n+1} } (I^n)$ or
  $f^{q_n} (I^{n+1} ) $ by some iterate of $f$. We call it the 
{\it landing  map }   $L=L_n$ of  level $n$.
On the central interval $I^n_0$, we let $L_n=\id$. 


 In case  of bounded type,
Theorem \ref{HS} ensures that these
tilings have bounded geometry,%
\footnote{This property is also referred to as {\it real a priori bounds}.}
 i.e., the
neighboring tiles are comparable, and hence the consecutive nested
tiles are also comparable. 
 This gives us a notion of {\it
 $n$-th dynamical scale } at any point $z\in \T$ (well defined
up to a constant): it the the size of  any tile $I^n(z)\in \II^n$
containing $z$. 

 More precisely, 
let $C_0= C (\bar N, K, \eps)\geq 2 $ be an upper bound for the ratios of
any two neighboring  and any two  consecutive nested dynamical tiles. 
We say that a point $\zeta\in \C$ lies in $n$-th dynamical scale
around $z\in \T$ if
\begin{equation}\label{dyn scales}
           C_0^{-1} | I^n_k |  \leq   |\zeta - z| \leq C_0 | I^n_k | 
\end{equation}
for the  dynamical tile $I^n_k$ of depth $n$ containing $z$. 
Any point $\zeta\in \D_2$  lies in some dynamical scale around any
$z\in T$, and the number of such scales is bounded in terms of 
$(\bar N, K, \eps)$.

\subsection{Renormalization $R_\cp$ of circle pairs}
\label{renorm of circle pairs}
 
A quasicritical circle map can be represented as a discontinuous map of the fundamental interval $[c_1-1 , c_1 ] $,
which motivates the following definition:  a {\it  quasicritical circle pair}
$F=(\phi_+,\phi_+)$     
is a pair of real analytic homeomorphisms
\begin{equation}\label{circle pair}
     \phi_- : [\beta_- , 0 ) \ra  [b, \beta_+) , \quad \phi_+ :
     ( 0, \beta_+] \ra  ( \beta_-, b]  
\end{equation}
 with  some $\beta_-  \leq 0  \leq \beta _+$,  $\beta_+ - \beta_- = 1$,     
Moreover, $c_0=0$ is the only critical point of the $\phi_\pm$
and this point is of quasicubic type in the sense of property Q2 from 
\S \ref{quasicrit circle maps}. Properties Q3 and Q4 are also easily translated to this
setting.


\msk
{\it Renormalization} $R_\cp$ of circle pairs is defined as
follows.  
In the degenerate case  $\beta_-=0$ or $\beta_+=0$  
(so that  the critical point is fixed under $\phi_+$ or $\phi_-$) 
$F$ is non-renormalizable.
In the non-degenerate case,   assume for definiteness that $ b \in (\beta_-, 0] $
(otherwise , one should change the roles of $\beta_-$ and $\beta_+$).
If  $\phi_-^N(\beta_-) \leq 0$ for all $N\in \N$
 (equivalently,  there is a fixed point in $(\beta_-, 0)$) 
then $F$ in  still non-renormalizable.%
\footnote{In other words, maps with zero rotation number are non-renormalizable.}
Otherwise, let $N\geq 1 $ be the biggest integer such that 
$$
   \beta_-' := \phi_-^{N}(\beta_- ) \leq 0, \quad  \beta'_+= \beta_+,
$$
and let 
$$
  \phi_-'|\, [\beta_-' , 0] = \phi_-,        \quad      \phi_+'|\, [0, \beta_+']  = \phi_-^{N}\circ \phi_+.
$$
Rescaling the interval $[\beta_-', \beta_+']$ to the unit size by an
orientation preserving%
\footnote{Under the usual convention, the rescaling is orientation
  reversing. However,  in  further applications to Siegel maps, this
  would lead to some inconvenience.}
 linear map, we obtain $R_\cp F$. 

To see how the renormalization acts on the rotation numbers,
let us consider the linear case (corresponding to the pure rotation). 
In this case, 
a convenient normalization of $F$ is to let 
$\max (|\beta_-|, \beta_+)=1$ 
leaving only one parameter 
$\beta = \min (|\beta_-|, \beta_+) \in [0,1]$
(related to the rotation number $\theta$ of $f$ by  $\theta = \beta_+/(1+\beta)$). 
Then $N$ is the biggest integer such that $N\beta \leq  1$, so $N$ is the
integer part of
$1/\beta$.  Under the renormalization, we obtain 
$$
  \beta' = \frac{1-N\beta} {\beta} = \frac 1\beta \ \mod \Z,
$$ 
which is the Gauss map applied to $\beta$. In this way, 
 the continued fraction expansion of $\beta$ (and hence $\theta$)  is directly 
related to the renormalization dynamics.%
\footnote{It is worth noting that in this renormalization scheme,
the cases $\beta_+=1$ and $\beta_-=-1$ alternate.}


\sss{Epstein class}

We say that a quasicritical circle pair $F= (\Phi_\pm)$ 
 belongs to {\it Epstein class} $\EE $  if
 the inverse  maps $\Phi_\pm^{-1}$ admit a conformal extension to 
the whole  upper  half-plane $\{ \Im z > 0\}$. 
Letting $I_-= [\beta_-, c_0]$, $I_+ = [c_0, \beta_+]$, we see  by
symmetry that the  maps $\Phi_\pm^{-1}$  admit a conformal extension to the plane slit
along two rays.  $\C_\pm = \C\sm ( \R \sm \Phi_\pm (I_\pm) ) $.
For $\de>0$, we let  $\C_\pm(\de)$ be a similar domain 
where  the interval 
 $\Phi_\pm (I_\pm) $ is scaled by factor  $(1+\de)$. 

\comm{*******************
We say that a quasicritical circle map $f: \R/\Z \ra \R/\Z$  belongs
to {\it Epstein class} $\EE$  if the inverse branches of it lift to $\R$, 
$$
    f^{-1}: (c_1,c_1+ 1)\ra (0,1) \quad  {\mathrm {and}}\quad  f^{-1}:  (c_1, c_1-1) \ra
    (-1, 0), 
$$
admit    conformal extensions to the whole upper and lower half-planes.

(Here and  in what follows, we notationally identify  a circle map with 
its lift.) 

Let us extend this definition to a  {\it quasicircle pair}
$F=(\phi_\pm)$, where $\phi_\pm:  I_\pm \ra \R$
(\ref{circle pair})  in a way that secures an appropriate {\it
  space} around the intervals $I_\pm$.

Let us consider a quasicritical circle pair $F= (\phi_\pm)$ with
$$I_-= [\beta_-,0] , \quad I_+=[0,\beta_+].$$  
We say that $F= (\phi_\pm)$ 
 belongs to {\it Epstein class} $\EE(\de) $  if

\ssk\nin $\bullet$
  The four intervals $I_\pm$, $\phi_\pm(I_\pm)$ are $\de^{-1}$-comparable.
 
\ssk\nin $\bullet$
  $\phi_\pm=  \psi_\pm\circ h_\pm$, where $h_\pm\in \Cir(\de^{-1},\de) $,
 while  the inverse  maps $\psi_\pm^{-1}$ admit a conformal extension to 
the slit plane 
$\C(\beta_--\de, \beta_+ + \de)$
 (this notation was introduced  in  the proof of 
Lemma \ref{neg S in   Epstein class}).  
***************}

\ssk

Let us consider another  quasicritical circle map  $f = (\phi_\pm)$
and  write its renormalizations as 
as
$$
     R_\cp^m f= (\phi_{m,\pm} ) = (\psi_{m,\pm } \circ \phi_\pm) 
$$
(splitting off the first iterate of $f$).
We say that  they converge (along a subsequence)  to $F\in \EE $ as above
if the latter can be represented as 
$$
\mbox{
$F=(\Phi_\pm)  = (\Psi_\pm\circ \phi_\pm) , $ 
}
$$
and   
there exists $\de>0$ such that  for any domain $\Om_\pm $ compactly
contained in the slit plane 
$   \C _\pm(\de) $, 
the inverse maps $(\psi_{m,\pm} ) ^{-1}$
are  eventually  defined on  $\Om_\pm$ and uniformly converge to
$\Psi_\pm^{-1}$ on it (along the subsequence in question).

The real {\it a priori} bounds imply, in the standard way,  precompactness of the
renormalizations, with all limits in the Epstein class,  see \cite{dFdM}: 

\begin{prop}\label{Epstein limits}
  For a quasicritical circle map $f\in \Cir(\bar N, K,\eps)$ of
  bounded type, 
any  sequence of the renormalizations $R_\cp^n f$  admits
a subsequence
converging to a quasicritical circle pair of Epstein class $\EE (\de)$
with $\de=\de(\bar N, K,\eps)$. 
\end{prop}

\subsection{Complex bounds and butterfly}

\sss{Holomorphic circle pairs (butterfly)}

Let us now complexify the above notions. A {\it holomorphic circle
  pair}  or a {\em butterfly map }
\begin{equation}\label{hol circle pair} 
    F=  (\phi_-, \phi_+) : (\hat X_- ,  \hat X_+ ) \ra  \hat Y
\end{equation}
is a holomorphic extension of a real  circle pair  $(\phi_- , \phi_+ ): (I_-, I_+)\ra \R$ with the following properties:

\ssk \nin $\bullet$ $\hat X_\pm\supset \inter I_\pm $ are disjoint
$\R$-symmetric  Jordan disks whose closures  touch only at $0$;
we let $X_\pm= \hat X_\pm\cap \{ \Im z >0\} $; 

\ssk \nin $\bullet$ $\hat Y$ is an $\R$-symmetric  topological disk compactly containing the $X_\pm$;
we let $Y= \hat Y\cap \{ \Im z>0 \}$; 

\ssk \nin $\bullet$ Each   $\phi_\pm $ maps the corresponding  $X_\pm$
univalently onto  $Y$;


\ssk \nin $\bullet$ The maps $\phi_\pm$ admit a quasiregular extension
to a neighborhood of $c_0$ with local degree $3$.


The configuration of domains $ X_+\cup X_-$ sitting inside $Y$ is called a {\it butterfly}.

Let us mark in $\hat Y$ the critical point $c_0=0$, and in $\hat X_\pm$ the critical
value $c_\mp = \phi_\mp (0)$. 
We say that a butterfly has a {\it $\kappa$-bounded shape} if each of  the marked domains involved
can be mapped onto  the unit disk $(\D,0)$  by a global $\R$-symmetric $\kappa$-qc
map.

\ssk
Let $\mod F= \min (\mod (Y\sm X_\pm)) $.

\sss{Complex bounds}
We are ready to state a quasicritical version of 
de Faria-Yampolsky  complex bounds \cite{dF,Ya-bounds} : 

\begin{thm}  \label{bounds for comm pairs}
 Let  $f\in \Cir (\bar N, K,\eps) $ be a quasicritical circle map of bounded type.
 Then there exists an $\underline{l}$ depending only on $(\bar N, K,\eps)$ 
 such that for all  $m\geq \underline{l}-1$
the renormalizations $R_\cp^m f $ can  be represented as 
a butterfly $  X^m_ - \cup  X_+^m \ra   Y^m$ of bounded shape  such
that 
$\hat Y^{\underline {l}-1 }\Supset \hat Y^{\underline{l}}\Supset\dots$
and
$$
      \mod (\hat Y^m\sm \hat Y^{m+1} ) \geq \mu, \quad \mod R_\cp^m \, f \geq \mu >0.
$$ 
Moreover, the boundary $\di X^m_+ $ near $c_0$ is a wedge of
angle $\pi/3$ obtained by taking $f$-preimage of  $[c_1, c_1 + \de_m^3)$
with $\de_m \asymp \diam Y^m$
(and similarly for $X_-$, using the other half-neighborhood of $c_1$).

All  geometric constants and bounds  depend only on $(\bar N, K, \eps)$. 
\end{thm}


\begin{proof}
The proof is  the same as in the analytic case 
(at the last moment making use of Lemma \ref{qr contraction}).
We will remind main steps  following the strategy of \cite{LY,Ya-bounds}. 

Due to Proposition \ref{Epstein limits},
it is sufficient to prove the bounds for 
circle pairs $F$ of Epstein class.

Let us consider an interval $I=I^n\in \II^n $ attached to the critical
point, and let $q=q_{n+1}$, $J = f^q (I)$.
Then $F^q |\, I$ can be decomposed as $\psi \circ F$ where
$\psi^{-1}: J\ra f(I) $ admits a conformal extension to the slit plane 
$ \C( J) $.   Here is the Key Estimate:
for  any $z$ outside $\T$, we have: 
\begin{equation}\label{key estimate}
      \frac{\dist (\psi^{-1}  (z)) , | f(I)| } { |f(I)| } \leq A \left(
            \frac{\dist (z, I )}  {|I|}\right) +B. 
\end{equation}
The proof uses only the real bounds and the Schwarz lemma for holomorphic
maps between slit planes.
 As  both these ingredients are available for our
class (as we always apply only holomorphic inverse branches of $F$), 
the Key Estimate is valid in this generality.

At the last moment we apply the inverse branch of the
cubic quasiregular map $F$ near its critical point. By  Lemma \ref{qr contraction},   it  is highly contracting in big
(rel $|I|$) scales, which implies (\ref{key estimate}).
%
%

%

Take now a big hyperbolic neighborhood $\De$ of  $L_\de(F))$  in the slit plane $\C(L_\de(F))$
and pull it back by $f^q$. The Key Estimate easily implies that the
pullback will be trapped well inside $\De$. This produces a 
butterfly with a definite modulus $\mu$. 

Slightly shrinking $\hat Y^l$ (using the space in between $\hat Y^l$ and
the $\hat X^l_\pm$)  and taking its pullbacks under $R^l f$, we obtain a butterfly with a
bounded shape.
\end{proof}

\sss{Expansion}

For $z\in \C\sm \bar \D$ near $c_0$, 
we will use notation $\ang z $ for the smallest angle between $z$ and $\R$
in the $\C/\R$-model with  $c_0 = 0$.
Together with the Schwarz Lemma, the above complex bounds imply:

\begin{cor}\label{expansion for circle maps}
Under the circumstance of Theorem \ref{bounds for comm pairs},
the butterfly renormalizations $f_m:= R_\cp^m f$ are expanding in the hyperbolic
metric of $Y^m$. Moreover, 
$$
     \| D f_m (z) \|_\hyp \geq \rho >1 
$$ 
with $\rho$ depending only on $(\bar N, K, \eps)$ and a lower bound
on $\ang z$ 
\end{cor}

\begin{proof}
Let $z\in X^m_+$, for definiteness. 
  The hyperbolic  expanding factor is equal to the inverse of  $\| Di (z) \|_\hyp$, where
  $i: X^m_+\ra Y$ is the embedding, and the norm is measured from the
  hyperbolic metric of $X^m_+$ to the one of $Y^m$. This norm  is
  bounded in terms of the upper bound on $\dist_\hyp(z, \di X^m_+)$
  measured in $Y^m$, which in turn, is controlled by  the relative
  Euclidean distance, 
$\dist(z, \di X^m_+)/ \dist(z, Y^m)$. 
Finally,  complex bounds  
(and in particular, the  wedge property of the butterfly)
imply that the latter is  bounded in terms of
the lower bound on $\ang z$. 
\end{proof}

\begin{cor}\label{expansion for circle maps-2}
   Let $f\in \Cir (\bar N, K,\eps)$ be a quasicritical circle map.
Then there exist $a>0$ and  $\rho >1$ depending on 
$ (\bar N, K,\eps) $  only such that 
if $z\in Y^m\cap \Dom^h f^n$ while  $f^n z\in Y^{m-k}$ for some $n\in \N$,  $0< k< m$
(with $m-k> ul $), 
then $$
    \| Df^n (z) \|_\hyp \geq a \rho^k,
$$ 
where the norm is measured in the hyperbolic metric of  $\C\sm \bar \D$.
\end{cor}

\begin{proof}
  On its way from $Y^m$ to $Y^{m-k}$, the orbit of $z$ must
land in the middle of 
$\asymp k$ domains $X^{m-i}_+ \cup X^{m-i}_-$, $0\leq i \leq k$. 
By Corollary \ref{expansion for circle maps}, the return map $R^{m-1}f $
is definitely expanding at these moments, in the hyperbolic metric of
$Y^{m-1}$.  By the complex  bounds,  this metric restricted to $Y^m$ 
 is boundedly equivalent to the         \note{OK?}
hyperbolic metric on $\C\sm \bar \D$. The conclusion follows. 
\end{proof}

\sss{Compactness} 
Let us normalize a complex pair $F: \hat X_+\cup \hat X_-\ra \hat Y$ 
so that $|\hat Y\cap \R| =1$ and introduce the following topology on the space of
normalized pairs.  A sequence 
$F_n: \hat X_+^n \cup \hat X_-^n\ra \hat Y^n$ converges
to a pair $F: \hat X_+\cup \hat X_-\ra \hat Y$ if the domains $\hat Y^n$ Carath\'eodory
converge to $\hat Y$ and the inverse branches 
$(F_n)^{-1} :  \hat Y^n_\pm\ra \hat X^n_\pm$ 
converge to the corresponding branches of $F^{-1}$ uniformly
on  compact subsets of $\hat Y_\pm$.

The {\it geometry} of a complex commuting pair is controlled by three
parameters:  $\mu$ (a lower bound on the modulus), $\kappa$
(a bound on the shape of the butterfly), and $B$, a bound on the
geometry of the  intervals $\hat X_\pm\cap \R$ inside $\hat Y\cap \R$. The latter is
defined as the best dilatation of a quasisymmetric map $(\hat Y\cap \R , 0) \ra
([-1,1],0)$ that moves the boundary points of the intervals in question
to some standard configuration.  Let $\PP(\mu, \kappa, B )$ stand for
the space of complex pairs with geometry controlled by the specified parameters.

\begin{prop}\label{compactness}
   The space $\PP(\mu, \kappa, B )$ is compact. For any $f\in
   \Cir(\bar N, K,\eps)$, the renormalizations
   $R^m f: \hat  X^m_+\cup \hat X^m_-\ra \hat Y^m$ of a quasicritical circle  map of
   bounded type eventually (for $m\geq m_0(\bar N, K,\eps)$) belong to some space
   $\PP(\mu, \kappa, B )$,
   with all the parameters  depending only on $\bar N$   and $K$.   
\end{prop}

\begin{proof}
  The first assertion follows from the standard compactness properties
  of the Carath\'eodory topology. The second one is the content of
  real and complex {\it a priori} bounds.  
\end{proof}

\subsection{Periodic points $\alpha^l$, collars $A^l$,
  and trapping disks $\Disc^l$}  

\sss{Periodic points $\alpha^l$}

Let us start collecting consequences of the complex bounds.

\begin{prop}\label{periodic pts}
  For any  $ l \geq \underline{l}-1$,    a quasicritical circle%
\footnote{In this statement we will use the unit circle model for
  $\T$ rather than $\R/\Z$.} 
 map of
  bounded type has a  repelling periodic point
$\alpha^l \in  X_-^l  \cup X_+^l $
 of period $q_l $.
Moreover, 

\ssk \nin ${\mathrm{ (i)}}$  $\dist(\alpha^l, \T)$ is comparable to the dynamical
depth at $c_0$ at scale $l$; 

\ssk \nin ${\mathrm {(ii)}}$
the multiplier of $\alpha^l $ is  bounded and bounded away
from 1 in absolute value.
\end{prop}

\begin{proof}
    Each restriction
$
       R^l f: X^l_\pm  \ra Y^l
$ 
is  a conformal map from
    a smaller domain onto a bigger one. By the Wolff-Denjoy Theorem
    (applied to the inverse map) it has a fixed point in the closure
    $\bar X^l_\pm $. However, it does not have fixed points on
    the boundary since $f$ does not have  periodic points on $\R$,
    while the image of $\di X^l_\pm\sm \R$ under $R^l f$ 
    (equal to $\di Y^l_\pm\sm \R $)  is disjoint from itself.  
   So, there is a fixed point $\alpha_\pm^l \in X^l_\pm $.  

  Assertions (i) and (ii) follow from compactness (Proposition \ref{compactness}).

Finally one of the points $\alpha^l_\pm$ has period $q_l$. 
\end{proof}

\sss{Collar Lemma and trapping disks $D^l$} \label{geom tilings} 

For all sufficiently big $l $,  
complex {\it a priori}  bounds allow us to construct nice collars 
$A^l $   around $\bar \D$ and nice trapping disks $D^l$ that
capture all orbits that  escape beyond the corresponding collars.

We say that a point $z\in \C\sm \bar\D$ {\it lies on depth $l$},
$d(z)=l$,  if
$$
        C_0^{-1}  |I^l (\zeta) | \leq  \dist (z, \T) \leq C_0 |I^l (\zeta)|,
$$
where  $\zeta$ is the closest to $z$  point of $\T$, 
and $C_0=C_0(\bar N, K, \eps)$ is the constant from (\ref{dyn scales}).
Of course, any point can lie on several depths (so $d(z)$ is    \note{OK?} 
multivalued),  but this number is bounded in term of $(\bar N, K, \eps)$. 

\begin{lem}\label{collar lemma}
For any quasicritical  circle map 
$f\in \Cir(\bar N, K, \eps)$ and any $l\geq \underline{l}-1$,  
there exists a  pair of smooth annuli (``collars'')%
\footnote{We prepare a pair of collars for each $l$ to make the
  statements robust under perturbations.}
$A_0^l \Subset A^l $ surrounding  $\D$ in $\Dom f\sm \bar \D$,
and a smooth quasidisk $D^l\ni \alpha^l $ in $ Y^l$
 with  the following properties:

\ssk\nin \rm {(A1)} 
Any boundary point $z\in \di^o A_0^l\cup \di^o A^l$  of these collars lies on  depth $d(z)$
with  
$$
|d(z)-l| \leq   \bar\iota= \bar \iota(\bar N, K, \eps);
$$
Moreover,   for any $z\in \di^o A_0^l,$ $\dist (z, \di^o A^l)\asymp \dist (z,
  \T)$,  and similarly for the inner boundaries $\di^i A_0^l$ and
  $\di^i A^l$;

\ssk \nin \rm {(A2)}
It is impossible to``jump over the collar'':
$$
\mbox{ If $z\in \Comp_0(\C\sm A_0^l) \sm \bar \D$ while  
$f( z)\not \in \Comp_0(\C\sm A_0^l )$ then $f(z)\in A_0^l$; } 
$$ 

\ssk \nin {\rm {(D1)}}
The disk $D^l$ has a  bounded shape around  $\alpha^l $;
it  has also the hyperbolic diameter of order 1 in  $Y^l\sm \bar \D$ and in $\C\sm \bar \D$; 
  
\ssk \nin {\rm {(D2)}} 
A definite portion of $D^l$ is contained in $f^{-1} (\D)\sm  \D $;
moreover, 
$$\mbox{
there is a point $\beta\in f^{-1} (\T )\sm  \bar\D $
that lies in the middle of $\Disc^l$; 
}$$

\ssk\nin \rm { (D3) } If  $z \in A^l $ 
 then there exists a moment  $k  < q_{l+1}$ such that $f^k  z$ lies in the
middle of  $D^l$.

\ssk\nin \rm { (D4) } 
 There exists 
$\underline{\iota}= \underline{\iota}(\bar N, \mu, K)$ 
such that 
for any 
$ \iota > \underline{\iota}$  and $l> \underline{l}+\iota$,
 we have  under the circumstances of  \rm { (D3) }:
$$
   f^i z  \not \in  D_1^{l-\iota} , \quad i=0,1,\dots, k,
$$ 
where $ D^{l-\iota}_1 \Subset \Om \sm \bar S $ is a disk containing 
$D^{l-\iota}  $ with a definite $\mod (D^{l-\iota}_1\sm D^{l-\iota})$; 
in particular, 
 $\Disc^l \cap \Disc^{l-\iota}_1=\emptyset$;  

\ssk \nin \rm { (D5) } Moreover,  under the above circumstances, 
$$
     f^i z  \in \Comp_0(A^{l-\iota})   , \quad i=0,1,\dots, k,
$$
and $A^{l-\iota} \Subset \Comp_0 (A^{l-2\iota} ) $. 

\ssk
All the bounds and constants depend only on $(\bar N, K,\eps)$. 

\end{lem}

\begin{proof}
Let us consider the circle pairs
renormalization $R_\cp^l f : X^l_-\cup X^l_+ \ra Y^l$.
For $Y^l$ we will also use notation $Y^l_0$.

  Any dynamical tile $I^l_k \in \II^l $  
(\ref{circle tilings})  
is compactly contained in the  topological disk $Y^l_k$ obtained by pulling
$Y^l$ back  by 
the conformal landing map, the complex extension of  
$L_l :   I^l_k \ra I^l_0 $. 
Complex {\it a priori} bounds   imply that $I^l_k $ is contained well inside 
$Y^l_k $.
 Hence each  $Y^l_k $  contains a round disk $\De_\eps (I^l_k)$ based on the
$(1+\eps)$-scaled interval $I^l_k$,
where  $\eps>0$ depends only on  {\it a priori}  bounds.
The union of these disks is an annulus $\DD^l\supset \D$ whose boundary 
lies on dynamical depth  $l$. 
Moreover, for $k\not=0$,  these disks lie well inside $\Dom^h f$,
since $Y^l_k\subset \Dom^h f$. 

Obviously, there is $\bar \iota= \bar \iota (\bar N, K, \eps)$ such that
for any $\underline{\iota}$,  we can select  collars 
$A^l\Subset A^l_1 \Subset \DD^l\sm \bar \D$
with the following properties:

\ssk
(i) They satisfy property (A1); 

\ssk
(ii)  Every point $z\in A^l_1$ lies in the middle of some half-disk 
    $\De_\eps (I^l_k) \sm \bar \D$;


\ssk
(iii) Every $z\in \di^i  A^l\cup\di^i  A^l_1$ lies on depth $d(z)$ with 
$0< d(z)- (l+ \underline{\iota})  < \bar i $.

Since $f$ is quasiregular,
there is  $\bar \iota = \bar \iota (\bar N, K, \eps)$ such that
$$
       d(f(z)  \geq d(z)  - \bar \iota, \ 
     z\in  \Dom^h f
$$
Together with (iii), this implies that if $\underline {\iota}$ is selected sufficiently big
($\underline{\iota}> 2\bar \iota$),
then property (A2) is satisfied as well: no point can jump over the
collar $A_0^l$.

Let us view the topological half-disk $Y^l \sm \bar \D$ as the 
hyperbolic plane, and let $D^l=  D^l(R)$ be the 
 hyperbolic disk of radius $R$ in $Y^l$ centered at $\alpha^l$.
By the Koebe Distortion Theorem, 
these disks  satisfy property (D1) with constants
depending on $R$ (or better to say,  on  an upper bound for $R$) .  

For $R$ big enough (depending only on ($\bar N$, $K$, $\eps$)) 
they also satisfy  (D2).  Indeed, since $f$ is quasiregular,
any sufficiently small disk $\D(c_0, r)$ contains a comparable disk
$\D(\zeta, ar) \subset f^{-1}(\D)\sm \D$. 
 Since the  domains $Y^l$ have a bounded shape around $c_0$, 
while the disks $D^l(R)$ closely approximate $Y^l\sm \bar \D$
(uniformly in $l$),  we conclude that for $R$ big enough, 
$$\mbox{
$D^l(R) \supset \D(\zeta, ar/2)$ 
and $\area D^l(R)\asymp \area \D(\zeta, ar/2)$, 
}$$
which yields the first part of (D2).

The second part of (D2) follows from Proposition \ref{local sector}
that implies that there is a point $\zeta\in f^{-1}(\T)$ lying in the
middle of $Y^l$. For $R$ big enough, it lies in the middle of $D^l(R)$
as well. 

If $z\in A^l$ then by Property (ii),  $z$  lies in the middle of
some  half-disk $Y^l_k \sm \bar \D$. By the Koebe Distortion Theorem, 
 under the landing map $L_l: Y^l_k  \ra Y^l $,
 it lands in the middle of $Y^l  \sm \bar \D$.  
Hence for  $R$ big enough
 $L_l( z)$ lies in the middle of  $ D^l(R)$ as well,
which establishes property (D3).

Since the whole orbit $\{ f^i z\}_{i=0}^k$ lies on depth $\geq l - O(1)$,
it  is separated from $D^{ l -\iota} $ and from 
$A^{l-\iota}$, as long as  $\iota$ is sufficiently big. 
Similarly, since $A^{ l-\iota} $ lies on depth $l-\iota$, it  is
separated from  $A^{l -2\iota}$ for $\iota$ big enough.
These remarks prove (D4) and (D5). 
\end{proof}

We say that the trapping disk $D=D^l$ is {\it centered} at $\alpha^l$,
or that $\depth D= l$.  

\bignote{Old collars and trapping disks and def of $\Dom^h f^n$ are hidden here}

\comm{******
\begin{lem}\label{collar lemma}
There exists $R=R(\bar N, K, \eps)$ such that for any quasicritical  circle map 
$f\in \Cir(\bar N, K, \eps)$ and any $l\geq \underline{l}-1$,  the collars 
$A^l \equiv A^l(R)$ have the following properties:

\ssk\nin \rm {(A0)}  Each collar $A^l$ separates
 $0$ from $\infty$,
and $A^l  \Supset A^{ l+1}  $;

\ssk\nin \rm {(A1)}  Distance from any   point $\zeta\in \di A^m $ to $\T$ is
comparable  
(with a constant depending only on  $(\bar N, K, \eps)$).
 with the $l$-th dynamical scale at $\zeta' \in \T$, where
$\zeta'$ is the point on the circle closest to $\zeta$; 

\ssk \nin \rm {(A2)}
$f(A^l \sm Y^l )\subset A^l $, and  
 there exists $\iota_0 =\iota_0 (\bar N, K, \eps)$ such that
 $f(Y^{l+\iota_0} ) \subset \hull A^l$;

\ssk\nin \rm { (A3) } If  $z \in \Dom^h f\cap \Comp_0(\C^*\sm A^l)$ 
while  $f^n z\in \Comp_\infty(\C^*\sm A^l)$ for some $n\in \N$, then
there exists a moment  $m \in (0,n)$ such that $f^mz$ lies in the
middle of  $Y^l(R)$.

\end{lem}

\sss{Trapping disk $\Disc^l$}\label{trapping disk for circle}
\msk
 For $n\in \N\cup \{\infty\}$, we  let   
$$
     \Dom^h f^n  = \{z: \ f^k z\in \Dom^h f, \ k=0,1,\dots, n-1 \}. 
$$

Given a pointed domain $(D, \beta)$,
we say that  {\it $\beta $ lies  in the middle of  $D$},
or equivalently, that {\it  $D$ has a bounded shape around $\beta$}
  if 
$$
     \max_{\zeta\in \di D} |\beta- \zeta| \leq C\, \min_{\zeta\in \di
       D} |\beta- \zeta|,
$$ 
where $C$ is a constant that may depend only on specified parameters.


\begin{prop}\label{disc D}
Let $f\in \Cir (\bar N, K,\eps)$ 
and  let $l\geq \underline{l} - 1 $. 
  Then  there exists a quasidisk $\Disc^l \ni \alpha^l$ 
with bounded shape around  $\alpha^l $ and a natural number
$\iota=\iota(\bar N, K, \eps) \in \N$ with the following properties: 

\ssk \nin {\rm {(i)}}
$D^l$ has hyperbolic diameter of order 1 in  $A^{ l -1 }\sm \bar
A^{l+\iota-1}$ and in $\C\sm \bar \D$; 
in particular, $\Disc^l \cap \Disc^{l+\iota}=\emptyset$;

\ssk \nin {\rm {(ii)}} 
A definite portion of $D^l$ is contained in $f^{-1} (\D)\sm  \D $;
moreover, there is a point $\beta\in f^{-1} (\T )\sm  \bar\D $
that lies in the middle of $\Disc^l$; 

\ssk \nin {\rm {(iii)}}
 any point $z\in A^l \cap \Dom^h f^n$ whose orbit escapes $A^l$
    at moment $n+1$ 
 must pass through $\Disc^l$ at the  moment $n$;

\ssk \nin {\rm {(iv)}}
 $\Disc^l $ can be univalently and with bounded distortion
pulled back along the orbit  $\{f^k z\}_{k=0}^n$ of the above point,
and this pullback is contained in $A^{l-1}$.

\ssk
All the bounds and constants depend only on $(\bar N, K,\eps)$. 
\end{prop}

\begin{proof}
For an even $\iota\in 2\N$,
   let us consider a set 
$Q^l_\iota   = Y^l\sm ( A^{l+\iota/2}\cup \bar    \D)$.
Complex {\it a priori} bounds (Theorem \ref{bounds for comm pairs}) and   
Property (A1) imply that  for $\iota$ sufficiently big
(depending only on $\bar N, K, \eps)$), we have
$$
  \diam Q^l_\iota  \asymp  \dist (Q^l_\iota, A^{l+\iota-1}) \asymp
  \dist(Q^l_\iota, \D),
$$ 
(with constants depending on $\iota$ but independent of $l$).
%
%
Hence $Q^l_\iota$ can be inscribed into a quasidisk 
$\Disc^l= \Disc^l_\iota   \subset Y^{l-1}  \sm A^{l+\iota-1}$ 
with the same property.
This implies property (i) for this disk.

Since distance from the periodic point  $\alpha^l$ to $f^{-1}(\D) \sm \D$ is
comparable to its distance to $c_0$, the disk $D^l$ can be selected so
that property (ii) is satisfied. 

Now, if  $z\in A^l$  is an escaping point as in (iii), 
then by property (A2),  $f^nz\in \Disc^l $,
provided $\iota/2> \iota_0$.

 Let us further enlarge $\Disc^l$ to a similar disk 
$\De^l \subset  Y^{l-1}\sm \bar \D$
 with a  definite space in between. By property (A3), 
the pullback of $ \De^l$ 
 along the orbit of $z$ stays inside $A^{l-1} $. Consequently, this
 pullback is obtained  by applying  well defined inverse branches of
 $f^k$,  and (iv) follows from the Koebe Distortion Theorem.
\end{proof}
************}

  

\subsection{Cylinder circle renormalization}

\sss{Real definition}
There is a different approach to the circle renormalization that
avoids using circle pairs. 
The quotient of $\R$ by the lift of $f$ is a circle $\T'$, and the first return map to the fundamental interval $[0, f(0)] $  
descends to a critical circle map of $\T'$.  Identifying $\T'$ with
$\T$ by means of a an orientation preserving analytic diffeomorphism
we obtain the renormalization $R_\cyl f$  of $f$ (defined up to an orientation preserving analytic conjugacy). 
 The rotation number of  $R_\cyl f$   is equal to  $-1/\theta \ \mod \Z $. 

This leads to the modified Gauss map  $G_* : \theta\mapsto -1/\theta\
\mod \Z$ 
accompanied by the modified  continued fraction expansion 
$$
   \theta = \frac 1{ N_1 - \frac 1{N2 - \dots} }  \equiv    
[N_1,\,   N_2,\dots]_* , \quad N \geq 2.
$$ 
We will use the same notation for the rational approximands in this
expansion,  $p_m/q_m = [N_1,\dots, N_m]_*$.
Of course,  notion of ``bounded type'' is 
independent of which expansion we use.

The rotation numbers 
$\theta_N= [N, \, N, \, N, \dots ]_*$  
  with equal entries%
\footnote{Note that $\theta_2=1$.} 
  $N\geq 3$
are called of {\it stationary type} (with respect to the modified expansion) 
The most familiar  of these is the
golden mean  $\theta_3 = (3-\sqrt{5})/2$.  

\sss{Complexification}\label{complex cyl renorm}
Let us start with a topological lemma:

\begin{lem}\label{nice curve}
  For a butterfly map $F$ (\ref{hol circle pair}), \note{term}
there exists an arc $\gamma$ connecting the fixed point  $\alpha\in
X_+$   to $\beta_+$  in such a way that $\alpha$ is the only 
common point of $\gamma$ and $f (\gamma)$.
Moreover, the triangle bounded by $\gamma$, $f(\gamma) $ and 
the arc of $[\phi_+ (\beta_+] \in \R$ is $\kappa$-qc equivalent 
(by a global map $\hat \C\ra \hat\C$) to the half-strip
\begin{equation}\label{half-strip}
 \{ z: \ \Im z\geq 0, \  0\leq \Re z \leq 1 \} \cup\{\infty\}
\subset \hat\C, 
\end{equation}
with $\kappa$ depending only on the qc geometry of the pair of domains
$(Y,X_+)$.   \note{comment}
\end{lem}

\begin{proof}
Let $\phi\equiv \phi_+$, $X\equiv X_+$,  $\beta_+\equiv \beta$,  $J := [\phi(\beta), \beta]$.
The pullback $J'  : = \phi^{-1}(J)$ is a  smooth subarc of $\di X_+$ 
touching $J$  at $\beta$ with angle $\pi/3$. Hence $J\cup J'$ is a
quasiarc.  Pulling  it further, we obtain a s sequence of
smooth arcs $J^n := \phi^{-n}(J)\subset \bar X$, $n=0,1,\dots$, 
one touching  the previous at angle $\pi/3$ and shrinking at a
geometric rate. 
Their union  $\bigcup J^n$ is a quasiarc
converging to the fixed point $\alpha$.  Adding $\alpha$ to it,
we obtain a closed quasiarc  $\Gamma=[\alpha, \phi(\beta)] $  such that 
$f(\Gamma) = [\alpha, \phi(\beta)]$ is a longer quasiarc.   Moreover,
the dilatation of $\Gamma$ depends only on  the geometry of  the pair
$(Y,X)$
(by compactness of the corresponding maps $\phi$).

The map  $\phi$ on $X$ can be globally linearized
by a $\kappa$-qc homeomorphism  $\psi: (\C, X) \ra (\C, \psi(X))$
which is conformal on $X$,
$\psi ( \phi(z))  =  \la \psi(z)$, $z\in X$,
with $\kappa$ depending only on the geometry of $(Y,X)$.  It can be further conjugate to the
doubling map $T: z\mapsto 2z$ by a qc homeomorphism $h: \C \ra \C$
that straightens the quasiarc $\Gamma$ to the unit  interval $[0,1]$.
In this model, we can let $\tl \gamma \equiv  h((\psi(\gamma))$ be a segment of a circle passing
through $0$ and $1$ sufficiently close to $\R$ so that it fits to the
domain $h(\psi(X))$.  Moreover, the triangle bounded by $\tl \gamma$, 
$2\cdot \tl \gamma$ and $[1,2]$ is $\kappa$-qc equivalent to the half-strip
(\ref{half-strip}),
with $\kappa$ depending only on the geometry of the pair $(Y,X)$. 
\end{proof}

For $m$ sufficiently big, the cylinder  renormalizations $R_\cyl^m f$ 
we have  described above can be complexified
as follows, see Yampolsky  \cite{Ya-cylinder}. 
Let us consider a periodic point $\alpha^m$, $m\geq \underline{l}$,  from Corollary \ref{periodic pts}. 
Then there is a $\T$-symmetric arc $\gamma_m$  connecting $\alpha^m$ to the symmetric
point%
\footnote{Here we describe it in terms of the unit circle $\T$ in $\C$.}  
 $1/\bar\alpha^m$ in such a way that
$f^{q_m} (\gamma_m)$ does not intersect $\gamma_m$. 
Let us consider
the fundamental  region $\De^m= \De^m(f)$ bounded by  these two  arcs.

\begin{lem}\label{definite geometry of De: circle}
   Let $f\in \Cir(\bar N, K,\eps)$. 
  Then the regions $\De^m$ are $\kappa$-qc equivalent to 
the strip $0\leq \Re z \leq 1$,
with $\kappa$  depending only on $(\bar N, K,\eps)$. 
\end{lem}

Let us now
identify the boundary components of $\De^m$ by means of  $f^{ q_m}$. 
We obtain a cylinder $\Cyl^m $ which is conformally equivalent to the 
standard bi-infinite cylinder $\C/\Z$.   
The first return map to $\De^m$ descends to a holomorphic map on
$\Cyl^m$ near the circle, and then can be transferred to $\exp (\C/\Z ,
\R/\Z)= (\C^*, \T)$. 
This is the {\it cylinder renormalization}  of a holomorphic circle
map  (well defined up to affine conjugacy). 


\bignote{Necklace is hidden here}

\comm{************
\sss{Necklace}
The following statement follows directly  from the definitions,
{\it a priori} bounds,  
and compactness (Proposition \ref{compactness}):

\begin{lem}\label{Delta}
Let  $f\in \Cir_\theta(\bar N, K, \eps)$.
Then for any $m\geq m_0(\bar N, K, \eps)\geq \underline{l}$,  the fundamental  quadrilateral $\De^m$ 
can be selected so that:


\ssk \nin $ \mathrm{ (i)}$ $\pi_m |\, \inter \De^m$ is a univalent map with bounded
distortion onto an annulus $A^m=A^m(R^mf)$ around $0$ slit along a
$\T$-symmetric  quasiarc $J^m$ which is disjoint from $\De^m (R^m f)$;

\ssk \nin $ \mathrm { (ii)}$   
  $\mod ((A^m \sm J^m) \sm \De^m (R^m f) )\geq \mu(\bar N, K, \eps)$; 

\ssk \nin $ \mathrm { (iii)}$ The pullbacks $\De_{-k}^m$, $k=0,1, \dots,
q_m,$ of $\De^m $ along the circle form a closed chain 
consecutively attached one to another;
moreover, for $k<  q_m$, the interiors of the above pullbacks, $\inter \De_{-k}^m$, 
except for the first and the last
ones, are all disjoint;

\ssk \nin $  \mathrm { (iv)}$ All the maps $f^k: \De_{-k} \ra \De_0$,
  with $k<  q_m$,   are univalent with bounded distortion;
  the map $f: \De_{- q_m}\ra \De_{- q_m+1}$ is a quasiregular triple
  branched covering;
 hence the sizes of the $\De_{-k}^m $ are comparable with the $m$th
 scale.  

\end{lem}

The union $\BDe^m : = \cup \De_k^m$ will be  called  the {\it
  necklace} of $f$ of depth $m$. 
We extend $\pi_m$ to he necklace (except for the last piece)  by means of the dynamical
periodicity,  $\pi (fz) = \pi(z)$.
*****************}

\comm{******
\sss{Pseudo-coverings}
Consider two plane discs  $D$ and $ V \ni 0$, and let $V^*=V\sm \{0\}$.  
Let us say that  a map $\pi:  (D , z) \ra  (V^*, \zeta)$ is a {\it
  pseudo-covering of degree} $\geq d$ if any path $\gamma $ in $V^* $
that begins at $\zeta$ and winds at most $d$ times around $0$,
i.e.,
$$
      \frac 1{2\pi}  \left |\int_\gamma d\arg\zeta \right|  \leq d
$$
 lifts to a path in $D$ that begins at $z$. 

Let  $M_m$ be the M\"obius map $(\hat \C, \D )  \ra (\hat \C, \D) $ sending $\zeta_m$ to
$\infty$, and let  $\Phi_m  = M_m^{-1} \circ \exp$  
  be the universal covering map  $ \C\ra   \hat \C\sm \{\zeta_m,  1/\bar\zeta_m\} $. 

\begin{lem}\label{control of chages of variable: circle}
  Let $f\in \Cir(\bar N, \mu, H)$. There exists $r>1$ and 
  $\underline{l}$ depending only on $(\bar N, \mu, H)$ such that
  the change of variable $\pi_m$ admits a lift%
\footnote{i.e., $\hat \pi_m= \pi_m\circ \Phi_m$.}
 $\hat \pi_m$ (that begins, say, at
  some $\Phi_m$-preimage  of $c_0$) to  $\Phi_m^{-1} (A^{m+\underline{l}} )$ which is a
  pseudo-covering  of degree at least $\Bq_m$ over the annulus $\{  1/r <
  |\zeta|< r \} $.  
\end{lem}
************************}

\comm{
*******************

In turn, this yields control of the changes of variables $\pi_m$,
which is very important in applications of the renormalization techniques:

\begin{lem}\label{control of chages of variable: circle -- old}
  Let $f$ be  a critical circle map of bounded type. 
   For any point $z\in \T$ and $m\geq l $,
 there is  a quasidisk  $Q^m= Q^m_f(z)$  centered at $z$ with a bounded shape  whose size has order
 of the dynamical scale at depth $m$  and such that the change of variable
 $\pi_m$ maps it with bounded distortion on a quasidisk covering
a definite arc of the circle.
Moreover,  there exists an $N$ such that 
the maps $\pi_N|\, Q^N $ are are expanding, and for any $m\geq 2N$ we have
\begin{equation}\label{Markov}
    \pi_N  (Q^m_f (z))  = Q^{m-N}_{R^N f}   (\pi_N(z)) \Supset Q^m_{R^N f}   (\pi_N(z)). 
\end{equation}
(All constants and bounds depend only on
the rotation number and the geometry of $f$.)
\end{lem}

\begin{proof}
    The changes of variable can be represented as follows:
$$
      \pi_m= \exp\circ \phi_m \circ \log \circ M_m.
$$
Here $M_m$ is the M\"obius map $\Hyp\ra \D$ sending $\alpha^m$ to
$\infty$, so  $\Pi_m:= \log\circ M_m  (\De_m)$ is a topological
strip that has a locally  bounded geometry. The map $\phi_m$
straightens this strip, equivariantly on the boundary, to the genuine vertical strip of width 1.
 It has a locally bounded distortion.

Let us   consider a ball $D_m := \D(\pi_m(z), 1/2)$. The above decomposition
implies that there exists a well defined branch $\pi_m^{-1}$ on $D_m$
that sends $\pi_m(z)$ back to $z$ and  has a bounded distortion.  
Hence the disk $Q^m := \pi_m^{-1}(D_m)$ is a quasidisk with a bounded shape  whose size has order
 of the dynamical scale at depth $m$.

In particular, it follows that the maps $\pi_N : Q_N\ra D_N$  are
uniformly expanding for $N$ big enough (with derivative of order
$s_m(z)^{-1}$, where $s_N (z )$ is the dynamical scale at $z$).  
Since $s_{m-N} / s_m\to  \infty $  as $N\to \infty$ uniformly in $m$
(and also: $m-N\geq N$),
we obtain inclusion (\ref{Markov}) for $N$ big enough.  
\end{proof}


*****************}

\subsection{Quasiconformal conjugacy}

\begin{thm}\label{qc conjugacy on circle}
  Two quasicritical circle maps,
$$
\mbox{  $f: \Dom^h f \ra Y$  and  $\tl f:   \Dom^h {\tl  f} \ra \tl
  Y$,  of class $\Cir(\bar N, K, \eps)$ }
$$
with the same rotation number  are $L$-qc conjugate, with 
$L= L(\bar N, K, \eps)$.  
\end{thm}

\begin{proof}
   It is an application of  {\it Sullivan's Pullback Argument},
   see \cite{MvS}. 
By Theorem~\ref{HS}, there is a quasiconformal map $h: \C\ra \C$
conjugating $f$ and $\tl f$ on the unit circle (with dilatation
depending only on $\bar N$). Using the complex bounds
(Theorem \ref{bounds for comm pairs}) this map can be adjusted so that it is equivariant
on the boundary of the butterfly, with dilatation depending only on $(K,\eps)$.

We can now start lifting the map $h$ under the dynamics to make it equivariant
on bigger and bigger parts of $\Om^h_f$. Since $f$ is  conformal on
$\Om^h_f$, these lifts preserve the dilatation of $h$. By
compactness of the space of normalized $L$-qc maps, we can pass to a
limit and produce the desired conjugacy.  
\end{proof}

\section{Siegel maps and their perturbations} \label{Siegel maps sec}

\comm{*********************
\subsection{Analytic critical circle maps}\label{analytic circle maps}

An  {\it analytic critical circle map } is an analytic  homeomorphism $f: \T\ra \T$ of the circle  $\T= \R/\Z $
with   a single critical point $c_0$ of cubic type (i.e.,
$f'''(c_0)\not=0)$. It is usually normalized so that $c_0=0$ in the angular coordinate.

By the classical theory supplemented with Yoccoz's No Wandering
Intervals Theorem  \cite{Y1}, 
if this map does not have periodic points then it is topologically conjugate to a rotation $T_\theta:  x\mapsto x+\theta\ \mod \ 1$,
where $\theta\in \R\sm \Q \ \mod\ \Z$ is the {\it rotation number } of $f$. The further theory largely depends on the Diophantine properties of $\theta $
encoded in its continuous fraction expansion $[N1,\,  N2,  \dots]$.  
Let  $\num_m/\Bq_m = [N_1,\dots, N_m]$ be the $m$-fold rational approximand to $\theta$. 
The rotation number (and the map $f$ itself) is called of {\it bounded type} if the entries of the expansion are bounded by some $\bar N$.

\begin{thm}[\cite{H,Sw}] 
  A critical circle map of bounded type is quasi-symmetrically conjugate to the rotation $T_\theta$.     
\end{thm}

The circle dynamics naturally encodes the continued fraction expansion
of the rotation number, as the denominators $q_n$ are the {\it
  moments of closest approaches} of the critical orbit $\{c_n\}$ back
to the critical point $c_0$. Let us consider the corresponding
intervals  $I^n=[c_0, c_{q_n}]$. The orbits of two consecutive ones,
\begin{equation}\label{circle tilings}
     I_k^n := f^k(I^n), \ k=0,1,,\dots, q_{n+1}-1\quad  {\mathrm{and}}
     \quad   I_k^{n+1}  ,   \ k=0,1,,\dots, q_n-1
\end{equation}
form a dynamical tiling  $\II^n$ of $\T$.  Moreover, these tilings are
nested: $\II^{n+1}$ is a refinement of $\II^n$. 

We let $I^n(z)$ be the interval of $\II^n$ containing $z\in \T$
(neglecting a slight ambiguity at the postcritical points).
Each of these intervals is diffeomorphically  mapped onto either
$f^{q_{n+1} } (I^n)$ or
  $f^{q_n} (I^{n+1} ) $ by some iterate of $f$. We call it the  ``landing
  map'' $L$. 

\note{as defined, the landing map is not the identity in $I^n \cup
  I^{n+1}$, but the first return map}

 In case  of bounded type,
Theorem \ref{HS} ensures that these
tilings have bounded geometry, i.e., the
neighboring tiles are comparable (and hence the consecutive nested
tiles are also comparable). This gives us a notion of {\it
  dynamical scale at depth $n$} at any point $z\in \T$ (well defined
up to a constant): it the the size of  any tile $I^n(z)\in \II^n$ containing $z$.  
 
*************************}

\subsection{Douady-Ghys surgery}

\sss {Blaschke model for Siegel polynomials}\label{Blaschke model}


Let us consider a quadratic polynomial 
\begin{equation}\label{P_theta}
   \Bf_\theta: z\mapsto e^{2\pi  i \theta} z + z^2, \quad \theta \in \R/\Z.
\end{equation}
When the rotation number $\theta$ has bounded type, it is linearizable
near the origin, and thus has a Siegel disk $\BS\equiv \BS_{P_\theta} \equiv \BS_\theta$.  
Here we will  briefly describe  the  Blaschke model  for this
quadratic map due to Douady and Ghys. 
It is based on a surgery that turns an appropriate
Blaschke product into $\Bf_\theta$.
 
Consider a family of Blaschke products
$$
   B_\alpha (z)= e^{2\pi i \alpha} z^2 \, \frac {z-3} {1-3z}  .
$$ 
It induces a family of critical circle maps on the unit circle $\T$.
 Adjusting the parameter $\alpha$ one can make the rotation number  of $B_\la$
assume an arbitrary value, so it can be made equal to the rotation
number $\theta$ from (\ref{P_theta}). 

Assume $\theta$  is of bounded type. 
Then by Theorem \ref{HS}, $B_\la: \T \ra \T$ is quasi-symmetrically conjugate to the pure rotation $T_\theta$. 
We can use this conjugacy to glue the Blaschke product on $\C\sm \D$
to the rotation of $\D$. 
This produces a degree two  quasiregular map $F$ of a quasiconformal sphere.
Moreover,  $F$  preserves the conformal structure  obtained by
spreading around the standard structure on the disk $\D$. By the
Measurable Riemann Mapping Theorem,  $F$ is quasiconformally conjugate
to some quadratic polynomial $ z\mapsto \la z + z^2 $. Since this
quadratic polynomial  has an invariant
Siegel disk with rotation number $\theta$, it coincides with $\Bf_\theta$.

\subsection {Expansion}\label{expansion sec}

Let us endow the complement $\C\sm \bar \BS $  of a Siegel disk
$ \BS = \BS_\theta$  of bounded type 
with the hyperbolic metric $\|\cdot\|_\hyp$. A standard application of
the Schwarz Lemma shows that
 the map $\Bf = \Bf_\theta$ is expanding in this metric,
$$
    \| D\Bf  (z)\|_\hyp >1, \quad   {\mathrm{if }} \ 
   z, \Bf (z) \in \C\sm  \bar\BS.
$$ 
Indeed, the map $\Bf : \C\sm f^{-1} (\bar \BS) \ra \C\sm \bar \BS$ is a
covering and hence a hyperbolic  isometry.
By the Schwarz Lemma, the embedding
\begin{equation}\label{i}
    i : \C\sm  \Bf^{-1} (\bar \BS) \ra \C\sm \bar \BS
\end{equation}
is a hyperbolic contraction.  Hence $ \Bf \circ i^{-1} : \C\sm \bar \BS \ra
\C\sm \bar \BS$ is  expanding on its domain of definition (i.e., on
$\C\sm \Bf^{-1}(\bar \BS)$). 

Using the Blaschke model,
 McMullen showed  that the expansion is uniform near the critical point:

\begin{lem}[\cite{McM1}]\label{expansion}
  Let $\Bf= \Bf_\theta$ be a Siegel quadratic polynomial of type bounded
  by $\bar N$,
and let $C>0$. 
Then there exists $\rho= \rho(\bar N,C) > 1$ such that 
$$
    \| D\Bf  (z)\|_\hyp > \rho  \quad   {\mathrm{if }} \ z, \Bf (z) \in
    \C\sm \bar \BS, \ {\mathrm{and}} \  |z-c_0|\leq C\, \dist(z, \BS),
$$
where the $\dist$ stands for the Euclidean one. 

\end{lem}

\begin{proof}
From the above argument  we see that $\| D\Bf (z)\|_\hyp = \| Di^{-1} \|_\hyp  $,
where $i$ is embedding (\ref{i}). The latter is bounded in terms of 
the hyperbolic distance from $z$ to $P^{-1} \bar S$ (in $\C\sm \bar
S$).  For the  Blaschke model,  this hyperbolic distance is bounded in terms of $C$.
The Blaschke model is $K$-qc equivalent to $\Bf $ where $K$ is bounded in
terms of $N$. The conclusion follows. 
\end{proof}


Let us now consider  a perturbation
$ \tl\Bf = \Bf_{\tl\theta}$
(not necessarily with real $\tl\theta$)
 of the Siegel polynomial $\Bf =\Bf_\theta $.
%
Let $ \tl\BO$ be the postcritical set of $\tl \Bf $.
Endow its complement $\C\sm \tl\BO$ with the hyperbolic metric
$\|\cdot\|_{\tl \hyp }  $.
Then the map $\tl \Bf $ is expanding with respect to this metric
(for the same reason as the  Siegel map $\Bf$). 
In fact,  it is also uniformly expanding near
 the critical point: 

\begin{lem}\label{perturbative expansion}
Let the type of $\theta$ be bounded by $\bar N$, and
let $C >0$.
Then there exists $\rho= \rho(\bar N,C) > 1$ 
such that for any compact set 
$K\Subset \C\sm  \bar\BS$ there exists $\de>0$ 
with the following property. 
Let $ |\tl\theta - \theta| < \de$, and assume  $\tl\BO$ is contained in the
$\de$-neighborhood of the Siegel disk $\BS$.  Then
for any point $z\in K\sm \tl\Bf^{-1} ( \bar \BO) $
such that 
\begin{equation}\label{z-position}
  |z -  \Bc_0|\leq C\, \dist(z, \bar\BS), 
\end{equation}
we have:
$$
    \| D   \tl\Bf (z)\|_{\tl\hyp}   \geq \rho.
$$
\end{lem}

\begin{proof}
  As the proof of Lemma \ref{expansion}  shows,
the expansion factor $\rho$ is bounded from below in terms of the
hyperbolic distance from $z$ to $ \tl\Bf^{-1}( \tl\BO )$ 
in $\C\sm \tl \BO $.

Let $U=U_\de$ be the $\de$-neighborhood of $\bar \BS$.
For $\de$ small enough,  $\bar U$ is disjoint from $K$.
Then the hyperbolic metrics on $\C\sm \bar \BS$
and on  $\C\sm \bar U$ restricted to $K$ are comparable
(and in fact, close for $\de$ small). 

By assumption,  the postcritical set
$\tl\BO$ is  contained in $U$.
By the Schwarz Lemma, the hyperbolic metric  $\|\cdot \|_{\tl\hyp} $ on
$\C\sm \tl\BO$ restricted to $K$ is  bounded by the
hyperbolic metric on $\C\sm\bar U$.  
Altogether, for $\de$ sufficiently  small we conclude: 
$$
   \|\cdot \|_{\tl \hyp}   \leq C_1 \,  \| \cdot \|_\hyp \quad
   {\mathrm{on}} \ K,
$$
with the constant  $C_1$ depending only on $N$
(in fact, $C_1$ can be taken arbitrary close to 1 for $\de$ small).

Since the dynamics of $\Bf$ on $\di \BS$ is minimal,
the set $\tl\BO$ makes  an $\eps$-net for $\di \BS$
provided $\de$ is small enough. Hence
 $ \tl\Bf^{-1} ( \tl\BO ) $ makes an $ O(\eps) $-net for 
$\Bf^{-1} (\BS)$. 
As we know (see the proof of Lemma \ref{expansion}), 
 condition  (\ref{z-position}) implies that 
the hyperbolic distance from $z$ to $ \Bf^{-1} (\BS)$ in  $\C\sm \bar \BS $ is
bounded. 
It follows that the hyperbolic  distance from $z$ to $\tl\Bf^{-1} (\tl\BO)$ in 
$\C\sm  \tl\BO $ is bounded as well. 
\end{proof}

\comm{****
Let us now consider  a sequence 
of  perturbations $ \Bf _m= \Bf_{\theta_m}$ of the Siegel polynomial
$\Bf =\Bf_\theta $ (not necessarily with real $\theta_m$). 
%
Let $ \BO_m$ be the postcritical set of $\Bf_m$.
Endow its complement $\C\sm \BO_m$ with the hyperbolic metric
$\|\cdot\|_m $.
Then the map $ \Bf_m$ is expanding with respect to this metric
(for the same reason as the  Siegel map $\Bf$). 
In fact,  it is also uniformly expanding near
 the critical point: 

\begin{lem}\label{perturbative expansion}
Let the type of $\theta$ be bounded by $\bar N$, and
let $C >0$. 
Then there exist $\rho= \rho(\bar N,C) > 1$ 
 with the following property. 
If
\begin{equation}\label{limsup O}
 \limsup \BO_m \subset  \bar\BS\equiv \bar S_\Bf.
\end{equation}
then for all sufficiently big $m $, we have
$$
    \| D   \Bf_m  (z)\|_m  > \rho  \quad   {\mathrm{if }} \  
z\in \C\sm   \bar \BS,\   
\Bf_m(z) \in    \C\sm  \BO_m 
$$
 and 
\begin{equation}\label{z-position}
  |z -  c_0|\leq C\, \dist(z, \bar\BS).
\end{equation}
\end{lem}

\begin{proof}
  As the proof of Lemma \ref{expansion}  shows,
the expansion factor $\rho$ is bounded from below in terms of the
hyperbolic distance from $z$ to $ \Bf_m^{-1}( \BO_m )$ in $\C\sm\ \BO_m $.

Let us take
a compact set $K\Subset \C\sm\bar S$ containing $z$
 and a  disjoint small disk neighborhood $U$ of $\bar \BS$.
 Then the hyperbolic metrics on $\C\sm \bar \BS$
and on  $\C\sm \bar U$ restricted to $K$ are comparable
(and in fact, close). 
Also, by assumption (\ref{limsup O}), the postcritical sets 
$\BO_m$ are eventually contained in $U$.
By the Schwarz Lemma, the hyperbolic metrics  $\|\cdot \|_m $ on
$\C\sm \BO_m$ restricted to $K$ are bounded by the
hyperbolic metric on $\C\sm U$.  
Altogether, we conclude for $m$ sufficiently  big:
$$
   \|\cdot \|_m  \leq C_1 \,  \| \cdot \|_\hyp \quad
   {\mathrm{on}} \ K,
$$
with the constant  $C_1$ depending only on $N$
(in fact, $C_1$ can be taken arbitrary close to 1 for $m$ big).

Also, since the dynamics of $\Bf$ on $\di \BS$ is minimal,
the set $\BO_m$ makes  an $\eps$-net for $\di S$
for $m$ big enough. Hence
 $ \Bf_m^{-1} ( \BO_m ) $ makes an $ O(\eps) $-net for 
$\Bf^{-1} (\BS)$. 
As we know (see the proof of Lemma \ref{expansion}), 
 condition  (\ref{z-position}) implies that 
the hyperbolic distance from $z$ to $ \Bf^{-1} (\BS)$ in  $\C\sm \bar \BS $ is
bounded. 
It follows that the hyperbolic  distance from $z$ to $\Bf_m^{-1} (\BO_m)$ in 
$\C\sm  \BO_m $ is bounded as well. 
\end{proof}
*****}

\subsection{Siegel maps}
\sss{Definition}\label{def of Siegel maps}
A {\it Siegel map} $f: (\Om , 0)\ra (\C,0)$,
i.e., a holomorphic map on a Jordan disk $\Om \equiv \Om_f = \Dom f$ with the following properties:

\ssk \nin S1.
 $f$  has a Siegel disk $S=S_f$ (centered at  0)  which is  a {\it quasidisk
   compactly contained in $\Om $}.  

\ssk \nin S2. 
  $f$ has a non-degenerate  critical point $c_0\in \di S$;
we let $c_n= f^n c_0$;

\ssk \nin S3. 
The domain  $\Om^h_f= \{ z\in \Om\sm \bar S: \; f z \in \Om \sm \bar S
\} $
is obtained from the annulus  $\Om \sm \bar S$ by
  removing a  topological triangle
$$
 \TT= \TT_f : =  (\Om \sm \bar \D)  \cup \{ c_0  \}
$$    
with a vertex at $c_0$ and the opposite side 
on   the  boundary of  $\Om$;

\ssk\nin S4.  
  $f :  \Om^h_f \ra \C$ is an immersion, and $f:  \TT \ra S \cup \{c_1\} $ is an embedding. 

\ssk We let $\Dom^h f = \Om^h_f \cup \bar S$.

\begin{rem}
  Note that Siegel maps are holomorphic by definition, 
so in this case  superscript $``$h$"$ is taken only by analogy with  the circle case. 
\end{rem}

Given $\bar N\in \N$ and $\mu>0$, 
let $\Siegel(\bar N, \mu, K) $ stand for the space of Siegel maps $f :  \Om
\ra \C$ of type bounded by $\bar N$ and such that  $\mod (\Om \sm
S_f)\geq \mu$ and $\di S_f$ is a $K$-quasicircle.  
(If irrelevant, some of these parameters can be skipped in the notation.)

We will later use notation
$\Siegel_\theta (\mu, K)\equiv \Siegel_N(\mu, K)$ for
 the class  of  Siegel maps $f\in \Siegel(\mu, K)$
with stationary rotation number 
 $\theta= \theta_N$ and such that
 $\mod (\Om \sm S_f)\geq \mu$.

\sss{Circle model for Siegel maps}
 By performing the Douady-Ghys surgery on
an arbitrary {\it analytic} critical circle map $g $ of
bounded type (not only on the Blaschke map), 
we can produce  plenty of Siegel maps. 
However, to produce all of them, we need to allow  {\it quasicritical}
circle maps.

\begin{prop}\label{Siegel map surgery}
  Any  Siegel map  $f: (\Om , 0)\ra (\C,0)$ of bounded type can be obtained by performing
  the Douady-Ghys surgery on a quasicritical circle map. 
\end{prop}  
                   
\begin{proof} 
Let $\psi_+: \C\sm  S\ra \C\sm  \D$ be the 
uniformization  of the complement of $S$ normalized so that $\psi_+(c_0)=1$.
Since $S$ is a quasidisk, it extends to a global quasiconformal map
$
    \psi_+:  (\C,  S) \ra  (\C,  \D).
$
Then 
$$
   g := \psi_+\circ f\circ \psi_+^{-1}:  (\C, \D) \ra (\C, \D)
$$ 
is a global quasiregular map in a neighborhood of $\bar \D$
which is a holomorphic immersion on $\psi_+(\Om^h)$. 
Applying the Schwarz Reflection Principle,
we obtain a quasiregular map $g$ near $\T$ that restricts to a homeomorphism $\T\ra \T$. 
Moreover, it 
is a holomorphic immersion  on  $\Dom^h g$, and hence  is real analytic on
$\T\sm \{1\}$. 
At the critical point $c_0=1$,  it has  local degree $3$. 
Moreover, properties (S3) and (S4) of $f$ readily translate to properties (Q3)
and (Q4) of $g$.   
Thus, $g$ is a quasicritical circle map.  

On the other hand, the uniformization $\psi_-:  \bar S\ra \bar \D$
conjugates $f$ to the rotation $T_\theta$ (and extends to a global qc
map).  Hence  $f$ is
the quasiconformal welding  between $g$ and $T_\theta$.
\end{proof}

\subsection {Circle $\leadsto$ Siegel transfer}

By means of the Douady-Ghys surgery,
we can transfer the objects defined above for quasicritical  circle maps
to their Siegel counterparts.
Somewhat abusing notation, 
we will usually keep the same notation for the transferred objects. 

\sss{Dynamical scales}                                
For any $f\in \Siegel ( \bar N, \mu, K)$, 
we can transfer  the circle  dynamical tilings
(\ref{circle tilings}) to the boundary of the Siegel disk $S$. Since the surgery is
quasisymmetric,
these {\it Siegel dynamical tilings} $\II^m$ have bounded geometry as well
(depending only on ($\bar N, \mu, K$)), 
which gives  us for any $z\in \di S$  a notion of the 
{\it  dynamical scales} near $z$.%
\footnote{  with the constant $C_0$ from (\ref{dyn scales}) replaced
with an analogous constant $C_0= C_0(\bar N, \mu, K)$ controlling   the geometry
of the tilings for Siegel maps.} 


\sss{Siegel butterfly renormalization}

Since any Siegel map $f$ of bounded type is conjugate on 
the boundary of  $ S$ to a quasicritical circle map,  
we can immediately define the {\it Siegel pairs} renormalizations $R_\Sp f$ 
on  $\di S$.  The complexification of  this notion, a {\it Siegel
  butterfly} 
\begin{equation}\label{Siegel cp}
      R_\Sp^m :  X_+^m \cup X_-^m \ra Y^m,
\end{equation}  
corresponds, via the surgery,  to the  external  part of the circle  butterfly.  
Theorem~\ref{bounds for comm pairs} implies:

\begin{thm}  \label{bounds for Siegel pairs}
 Let  $f\in \Siegel (\bar N, \mu, K) $ be a Siegel map of bounded type.
 Then there exists an $\underline{l}$ depending only on 
$(\bar N, \mu, H)$ 
 such that for all  $m\geq \underline{l}-1$
the renormalizations $R_\Sp^m f $ on $\di S$ can  be extended to
Siegel butterflies
$$
      R^m_\Sp f:    X^m_-\cup X_+^m \ra Y^m
$$
 with 
$Y^{\underline {l}-1 }\Supset Y^{\underline{l}}\Supset\dots$ such that  
the $Y^m$ are quasicircles of bounded shape and 
$$
  \dist (\di Y^m \sm \di S , \,  Y^{m+1}  ) \asymp 
  \dist (\di Y^m \sm \di S, \,  X^m_\pm )   \asymp \diam Y^m. 
$$

All constants and bounds depend on $(\bar N, \mu, K ) $ only.
\end{thm}


As in the circle case, these {\it a priori} bounds lead to external  expansion:

\begin{cor}\label{expansion for Siegel maps}
Under the circumstance of Theorem \ref{bounds for Siegel pairs},
the renormalizations $f_m:= R_\Sp^m f$ are expanding in the hyperbolic
metric of $Y^m$. Moreover, 
$$
     \| D f_m (z) \|_\hyp \geq \rho >1 
$$ 
with $\rho$ depending only on $(\bar N, K, \eps)$ and a lower bound
on  
 $ \dist( z, \bar S )  / \dist (z, c_0) $. 
\end{cor}

\begin{cor}\label{expansion for Siegel maps-2}
   Let $f\in \Cir (\bar N, \mu,  K)$ be a  Siegel map.
Then there exist $a>0$ and  $\rho >1$ depending only on 
$ (\bar N, \mu, K) $  such that 
if $z\in Y^m\cap \Dom^h f^n$ and  $f^n z\in Y^{m-k}$ for some $n\in \N$,  $0< k< m$
(with $m-k > \ul $), 
then $$
    \| Df^n (z) \|_\hyp \geq a \rho^k,
$$ 
where the norm is measured in the hyperbolic metric of  $\C\sm \bar S$.
\end{cor}

\sss{Periodic points $\alpha^l$} 

Proposition \ref{periodic pts} implies: 

\begin{cor}\label{Siegel periodic pts}
 For any Siegel map  $f\in \Siegel (\bar N, \mu, K)$,
there exists $\underline{l}= \underline {l} (\bar N, \mu)  $
such that for any $ l \geq \underline{l}-1 $,  $f$  has a repelling  periodic point $\alpha_l$ of
period $q_l$    in the $l$-th dynamical scale  near
the critical point $c_0$. 
\end{cor}

\begin{rem}\label{parabolic explosion}
  If $\Bf=\Bf_\theta $ is a Siegel quadratic polynomial with  rotation
  number of bounded type, then  the periodic
  point $\alpha^l $ was born in the {\it parabolic explosion} from the
  parabolic approximand $\Bf_{p_\kappa/\ q_\kappa}$. It can be
  characterized as the landing point of a ray with rotation number
  $ p_\kappa/ q_\kappa$.  
\end{rem}

\sss{External collars of $A^l$ and trapping disks $D^l$}

Let us now  transfer, by means of the surgery,  the collars 
and trapping disks from the circle plane to the
Siegel plane. It is a direct consequence of Lemma \ref{collar lemma}
and quasisymmetry of quasiconformal maps. 

\begin{prop}\label{disk D for Siegel}
For any Siegel map  
$f\in \Siegel(\bar N, \mu, K)$ and any $l\geq \underline{l}-1$,  
there exists a  pair of smooth annuli (collars)
$A_0^l \Subset A^l $ surrounding  the Siegel disk  $S=S_f$ in $\Dom f\sm \bar S$,
and a smooth quasidisk $D^l \Subset  \Dom f \sm \bar S$ containing $\alpha^l$
 with  the following properties:

\ssk\nin \rm {(A1)}
  For any $z\in \di^o A_0^l,$ 
$\dist (z, \di^o A^l )\asymp \dist (z,  \di^o S)$, 
 and similarly for the inner boundaries $\di^i A_0^l$ and $\di^i A^l$; 


\ssk \nin \rm {(A2)}
It is impossible to``jump over the collar'':
$$
\mbox{ If $z\in \Comp_0(\C\sm A_0^l) $ while  
$f( z)\not \in \Comp_0(\C\sm A_0^l )$ then $f(z)\in A_0^l$; } 
$$ 

\ssk \nin {\rm {(D1)}}
The disk $D^l$ has a  bounded shape around  $\alpha^l $ and
it  has  the hyperbolic diameter of order 1  in $\C\sm \bar S$; 
  
\ssk \nin {\rm {(D2)}} 
A definite portion of $D^l$ is contained in $f^{-1} (S)\sm  S $;
moreover, 
$$\mbox{
there is a point $\beta\in f^{-1} (\di S )\sm  \bar S$
that lies in the middle of $\Disc^l$; 
}$$

\ssk\nin \rm { (D3) } If  $z \in A^l $ 
 then there exists a moment  $k  < q_{l+1}$ such that $f^k  z$ lies in the
middle of  $D^l$;

\ssk\nin \rm { (D4) } 
 There exists 
$\underline{\iota}= \underline{\iota}(\bar N, \mu, K)$ 
such that 
for any 
$ \iota > \underline{\iota}$  and $l> \underline{l}+\iota$,
 we have under the circumstances of  \rm { (D3) }:
$$
   f^i z\not \in D_1^{l-\iota} , \quad i=0,1,\dots, k,
$$ 
where $ D^{l-\iota}_1 \Subset \Dom f \sm \bar S $ is a disk containing 
$D^{l-\iota}  $ with a definite $\mod (D^{l-\iota}_1\sm D^{l-\iota})$; 
in particular, 
 $\Disc^l \cap \Disc^{l-\iota}_1=\emptyset$. 

\ssk \nin \rm { (D5) } Moreover,  under the above circumstances, 
$$
     f^i z  \in \Comp_0(\C\sm A^{l-\iota})   , \quad i=0,1,\dots, k,
$$
and $A^{l-\iota} \Subset \Comp_0 \C\sm (A^{l-2\iota} ) $. 

\ssk
All  bounds and constants depend only on $(\bar N, \mu, K)$. 
\end{prop}

\bignote{Old collars and trapping disks are here}
\comm{**************

We let 
$$
\hull_0\,  A^m = A^m \cup \bar S.
$$
These sets form a basis of neighborhoods of $S$. 
Properties (A1)--(A3) transfer to the following properties of these neighborhoods:

\ssk\nin (T0)  
For $l\geq \underline{l}-1$, we have $\hull_0\,  A^l \Supset A^{l + 1}$ ;

\ssk\nin (T1)  Distance from any   point $\zeta \in \di (\hull_0\,  A^m) $ to $\di S$ is
comparable (with a constant depending only on $\bar N, \mu, K$)
 with the $m$th dynamical scale at $z\in \di S$, 
where $z\in \di S$ is the closest point to $\zeta$; 

\ssk \nin (T2)
$f(A^m\sm Y^m)\subset\hull_0\,  A^m$, and  
 there exists $\iota$ (depending only on $\bar N, \mu, K$)
such that  $f(Y^{m+\iota } ) \subset \hull_0\, A^m$ for 
$m  \geq \underline{l}-1 $.


\ssk \nin (T3)
 $f^{-1} (A^m)\subset \hull_0\,  A^m$.


\sss{Trapping disk $D$}

Finally, we can also transfer the trapping disk $D^l$ around
$\alpha^l$ (see \S \ref{trapping disk for circle})
to the Siegel plane: 

\begin{prop}\label{disk D for Siegel}
Let us consider a Siegel map $f$ of class $\Siegel(\bar N, \mu, K)$,
and let $l\geq \underline{l}-1$.
 Then   there exists a quasidisk $\Disc^l \subset Y^l $ 
 with bounded shape   around $\alpha^l $ 
and a natural number $\iota=\iota(\bar N, \mu, K)$
with the following properties:

\ssk\nin {\rm {(i)}}
 $D^l$  has hyperbolic diameter of order 1 in 
 $ A^{ l -1 }\sm \bar A^{l+\iota-1}$ and in $\C\sm \bar S$; 
in particular, $\Disc^l \cap \Disc^{l+\iota}=\emptyset$;


\ssk\nin {\rm {(ii)}}
A definite portion of $D^l$ is contained in $f^{-1} (S )\sm S $;
moreover, there is a point $\beta\in f^{-1} (\di S)\sm  \bar S$
that lies in the middle of $\Disc^l$;  

\ssk\nin {\rm {(iii)}}
 Any point $z\in  A^l $ whose orbit escapes $ A^l$
    at  moment $n+1$ 
 must pass through $D^l$ at the  moment $n$;

\ssk\nin {\rm {(iv)}}
$D^l$ can be univalently and with bounded distortion pulled
 back along the orbit $\{f^k z\}_{k=0}^n$ of the above point, 
and this pullback is contained in  $ A^{l-1}$.  

\ssk 
All  the bounds and constants  depend only on $(\bar N, \mu, K)$.
\end{prop}

***********************}

\bignote{Siegel renormalization fixed point is hidden here}

\comm{****
\subsection{Fixed point for $R_\Sp$ and geometry of a Siegel disk}

\sss{Fixed point} 
Using the Blaschke model and {\it a priori} bounds,
McMullen \cite{McM3}   proved existence of a fixed point 
for the Siegel pairs renormalization of stationary type:

\begin{thm}\cite{McM3} \label{Siegel fixed point}
   Let  $\theta=\theta_N$ be  
a stationary rotation number,
and let $\Bf_\theta: z\mapsto e^{2\pi \theta} z + z^2$ be the
corresponding Siegel quadratic polynomial. 
Then the renormalizations $R_\Sp^m (\Bf_\theta) $ converge exponentially
fast to an invariant  Siegel pair $F_\#$.  
\end{thm}

************}

\bignote{McMullen's Thm is hidden here} 

\comm{******

\subsection{Geometry of a Siegel disk}

The following result plays an important role in the work of Buff and Cheritat \cite{BC},
and will be as important for us.  

\begin{thm}[\cite{McM3}]\label{density pts}
Let $\Bf_\theta$  be a  Siegel quadratic polynomial of stationary type.
 Then for any $z\in \di S_\theta $, 
$$
    \area (K(\Bf_\theta) \cap \D_\eps(z)  ) \geq (1-\de) \,  \pi \eps^2,
$$ 
where $\de= O(\eps^\alpha)$ for some  $\alpha > 0$. 
\end{thm}
**********}

\subsection{Siegel cylinder renormalization}

\sss{Definition}
Using the circle model, 
we can extend Yampolsky's construction of 
the {\it cylinder renormalization}   $R_\Sieg$ \cite{Ya-posmeas} to all Siegel
maps  $f\in \Sieg_\theta$ of bounded type.
Let $g$  be the quasicritical circle map corresponding to $f$ through
the surgery.  
Let us transfer the arc used for the $m$th cylinder renormalization of
 $g$  (see  \S \ref{complex cyl renorm})
 to an arc $\de_m$ connecting the periodic point $\alpha^m$ of $f$ from
Corollary \ref{Siegel periodic pts} to the boundary of $S_f$. 
By continuing along the internal ray of $S_f$,
 extend $\de_m$  to an arc  $\gamma_m$ connecting $\alpha^m$  to the
 Siegel fixed point $0$. Then $ f^{q_m} (\gamma_m)$ does not intersect $\gamma_m$,
and these two arcs bound a {\it fundamental crescent}  $\CC^m$ for
$f^{ q_m}$.
Now we can proceed with the construction as in the circle case:
identifying the boundary arcs of $\CC^m$,
we  produce a map of the standard cylinder $\C/\Z$ whose upper end 
corresponds to the Siegel fixed point.  To recover this point  back,
let us map $\C/\Z$ onto $\C^*$ by means of $e^{iz}$. We obtain a Siegel
map with rotation number $-1/\theta\ \mod \ 1$.

The following statement is a  Siegel counterpart of 
Lemma \ref{definite geometry of De: circle} 
that  follow from the latter by surgery.  

\begin{lem}\label{definite geometry of De: Siegel}
  Let $f$ be a Siegel  map of  class $\Siegel (\bar N, \mu, K)$.
  For any $m\geq \underline{l} -1$,  
  the fundamental crescent $\CC^m$ is $\kappa$-qc
  equivalent to the quadrilateral composed by attaching 
the half-strip (\ref{half-strip})  (corresponding to  $\CC^m \sm S$)
  to a triangle with angle $2\pi/q$ at $0$  (corresponding to
 $\CC^m  \cap \bar S$). The dilatation  $\kappa$
  depends only on $(\bar N, \mu, K)$.
\end{lem}

As in the circle case, 
let $\pi_m= \pi_m^f $ stand for the  change of variable projecting the
original dynamical plane to the renormalized one:
it starts in the fundamental crescent 
$\CC^m $ and then is spread  around by means of the dynamics.

\bignote{Siegel necklace is hidden here}

\comm{**********
\sss{Siegel necklace}
The following statement is a  Siegel counterpart of 
Lemma~\ref{Delta} that follows from the latter by the surgery. 

\begin{lem}\label{Delta-Siegel}
Let  $f\in \Siegel_\theta(\bar N, \mu, K)$.
Then for any $m\geq m_0(\bar N, \mu, H)\geq \underline{l}$,  there is
a choice of the crescent $\CC^m$ and a
pentagon $\De^m\Subset \Om\sm \{0\} $ (obtained by cutting off a  wedge   of  $\CC^m$ attached to $0$) 
with the following properties:


\ssk \nin $ \mathrm{ (i)}$ $\pi_m |\, \inter \De^m$ is a univalent map with bounded
distortion onto an annulus $A^m=A^m(R_\Sieg^mf)$ around $0$ slit along a
  quasiarc $J^m$ which is disjoint from $\De^m (R_\Sieg^m f)$;

\ssk \nin $ \mathrm { (ii)}$   
  $\mod ((A^m \sm J^m) \sm \De^m (R_\Sieg^m f) )\geq \nu (\bar N, \mu, H)$; 

\ssk \nin $ \mathrm { (iii)}$ The pullbacks $\De_{-k}^m$, $k=0,1, \dots,
\Bq_m,$ of $\De^m $ along the circle form a closed chain 
consecutively attached one to another;
moreover, for $k< \Bq_m$, the interiors of the above pullbacks, $\inter \De_{-k}^m$, 
except for the first and the last
ones, are all disjoint;

\ssk \nin $  \mathrm { (iv)}$ All the maps $f^k: \De_{-k} \ra \De_0$,
  with $k< \Bq_m$,   are univalent with bounded distortion;
 the last map $f: \De^m_{-\Bq_m} \ra \De^m_{-\Bq_m+1}$ is a double branched covering;

\ssk \nin $  \mathrm { (v)}$
the sizes of the $\De_{-k}^m $ are comparable with the $m$th
 scale.  

\end{lem}

The union $\BDe^m : = \cup \De_k^m$ will be  called  the {\it
Siegel  necklace} of $f$ of depth $m$. 
As in the circle case, we extend $\pi=\pi_m$ to he necklace (except for the last piece)  by means of the dynamical
periodicity,  $\pi (fz) = \pi(z)$.

We let $\Om^m$ be the {\it filled necklace} $\BDe^m$ obtained by
adding to $\BDe^m$ the complementary component containing $0$.
Obviously, $\Om^m\subset S$. 

\begin{cor}\label{Om-domain}
  The domains $\Om^m$ shrink  to the Siegel disk $S_f$.  
\end{cor}

***********************}

\comm{***************
Let $\Phi_m$ be the universal covering $(\C, 1) \ra (\hat\C^*  \sm\{\zeta_m\}, c) $.  \note{correct normalization ? }

\begin{lem}\label{changes of variable}
  Let $( f: (\Om_f , 0)\ra (\C, 0) ) \in \Siegel(\bar N, \mu)$. There exists a nest of  disks
 $$
      \Om_f =\Om_0 \Supset \Om_1\Supset \dots  \Supset S_f, \quad
      \Om^*_m: = \Om_m\sm \{0\},
$$
such that 

\ssk \nin $\bullet$
  For any $z\in \di \Om_m$, $\dist(z, S_f) $ is comparable with the
  dynamical scale at depth $m$ at the closest point $u\in \di S_f$;

\ssk\nin $\bullet$
The change of variable $\pi_m$ admits a lift  $\hat \pi_m$ 
  to  $\Phi_m^{-1} (\Om_m^*)$ which is a
  pseudo-covering  of degree at least $\Bq_m$ over $\Om_{R_\Sieg^m f }^*$.
\end{lem}

*******************}

\bignote{S/u manifolds for Siegel renormalization are hidden here}


\comm{*****
\sss{Hybrid classes and stable manifold}

Let us say that two Siegel maps, $f$ and $\tl f$,  are {\it $L$-hybrid conjugate}
if there exists an equivariant $L$-qc map $\Dom^h_f\ra \Dom^h_{\tl f}$
which is conformal on the Siegel disk $S_f$.  

By means of the Douady-Ghys surgery,
Theorem \ref{qc conjugacy on circle}  can be immediately transferred to
the Siegel setting:

\begin{thm}\label{qc conjugacy for Siegel}  
   Two Siegel maps $f, \tl f\in \Siegel(\bar N, \mu, K)$   with the same rotation number  are
   hybrid conjugate.
\end{thm}

As in the usual quadratic-like   theory,
the hybrid classes become the renormalization stable manifolds:

\begin{thm}\label{exp convergence}
   Let $\theta=\theta_N $ be a  stationary rotation number.
Then the Siegel cylinder renormalization operator $R_\Sieg: \Siegel_\theta\ra
\Siegel_\theta$ has  a unique fixed point  $f_\# $ and  for any Siegel map
$f\in\Siegel_\theta$, the renormalizations $R_\Sieg^n f$ converge to
$f_\#$ exponentially fast. 
\end{thm}

\begin{proof}
  The fixed point $f_\#$ can be obtained by taking the cylinder
  renormalization of the butterfly fixed point (denoted in the same way)
of Theorem \ref{Siegel fixed point}.  
Exponential convergence to $f_\#$ can be proved 
by means of  McMullen's towers theory
in the same way as  Theorem \ref{Siegel fixed point}.
Namely, Theorem \ref{density pts} implies that    
$R_\Sieg^n  f$ and $R_\Sieg^n (\Bf_\theta)$ are $L_n$-qc conjugate with
the exponentially decaying dilatation $L_n$.  Hence the orbit $\{R_\Sieg^n
f\}$  exponentially converges  to the same fixed point ($f_\#$) as $\{ R_\Sieg^n
(\Bf_\theta)\}$. 
\end{proof}

%
%

\sss{Unstable manifold}

Given a Siegel map $f_\circ\in \Siegel_N$,
let $\BB_\circ =\BB(f_\circ)$ be the Banach space of bounded
holomorphic functions on $\Dom f$. 
The fundamental crescent $\CC_\circ$ of $f_\circ$ gives an origin  to
holomorphically moving crescent $\CC_f$ over a neighborhood of
$f_\circ$.  This allows one to extend  analytically  the  renormalization operator
$R_\Sieg$ to a neighborhood of $f_\circ$.
In particularly, we can consider $R_\Sieg$ near the renormalization fixed point
$f_\#$.

\begin{thm}[\cite{Ya-cylinder}]\label{hyp fixed point}
Let $\theta=\theta_N$ be a stationary rotation number.
  Assume all the holomorphic maps $f$ with $f(0)=0$ and 
$f'(0)=e^{2\pi  i \theta}$ that are sufficiently close to the renormalization fixed
point $f_\#$ on $\Dom f_\#$  belong to some class $\Siegel_\theta(K,
\eps)$.  Then $f_\#$ is a hyperbolic fixed for $R_\Sieg$.  
 \end{thm}

\begin{proof}
   The conditions  $f(0)=0$, $f'(0)=e^{2\pi  i \theta}$ specify
   a codimension-one subspace  in the space $\BB$ of bounded holomorphic functions
   on $\Dom f_\#$ with a fixed point at $0$.  By our assumption, there
   is a ball in this subspace filled with maps of class $\Sieg_\theta
   (K, \eps)$. By Theorem \ref{exp convergence}, this class is
   contained in the stable manifold of $f_\#$. Hence the codimension
   of the latter is at most 1. 

The multiplier $\la= f'(0)$ is a linear functional of $\BB$ that gives
a transverse coordinate to $\Sieg_\theta(K, \eps)$,
and so is the rotation number $\theta= (2\pi i)^{-1} \log \la$. 

The cylinder  renormalization $R_\Sieg$ analytically extends to a
small neighborhood of $f_\#$ 
(using a holomorphically moving  fundamental crescent).  
It becomes an operator fibered over the $\alpha$-axis,
with the Gauss map $\alpha\mapsto \, -1/\alpha\, \mod 1$ acting as the
quotient. Since the latter is expanding,  
the hyperbolicity follows. 
\end{proof}

****************}

\comm{************
Given a real analytic family $\FF=\{f_\theta\}$ of  maps
transversally passing through
a Siegel map $f_\circ \in \Siegel_N$ of stationary type, 
we define $R_\Sieg \FF$ by applying $R_\Sieg$ to the interval of
renormalizable $f_\theta$'s (i.e., such that $\theta\in [ 1/N, 1/(N-1)$) 
around $\theta_\circ$.  In fact, we can extend $R_\Sieg$ to a small
complex neighborhood of the renormalization interval 
(within the compexification of  $\FF$). This
neighborhood is not canonically defined, but we fix some and refer to
these maps as ``renormalizable'' by $R_\Sieg$. This yields a notion of 
$n$-times renormalizable maps as well.   Hyperbolicity of the
renormalization fixed point implies that the renormalizations
$R^n_\Sieg (\FF)$ form a precompact family.
In particular, we obtain:
************}


\bignote{Perturbed Siegel maps are hidden here}

\comm{****

\subsection{Perturbed Siegel maps} 

\sss{Perturbed necklace}

Let us begin with a  straightforward perturbation of Lemma \ref{Delta-Siegel}:

\begin{lem}\label{perturbed Delta-Siegel}
Let  $f_\circ \in \Siegel_\theta(\bar N, \mu, K)$.
Then for any $m\geq m_0(\bar N, \mu, K)\geq \underline{l}$,  there is
a Banach $\de$-neighborhood $\UU=\UU_\de $ of $f_\circ$  in the Banach space $\BB_\circ$ 
such that any $f\in \UU$ has 
a crescent $\CC^m= \CC^m(f)$ and a
pentagon $\De^m=\De^m(f) $
with the following properties:

\ssk \nin $ \mathrm{ (0)}$ 
$\CC^m(f)$ and $\De^m(f)$ move holomorphically with $f\in \UU$;
the crescents determine the renormalization change of variable $\pi_m=\pi^f_m$;

\ssk \nin $ \mathrm{ (i)}$ 
 The change of variable $\pi_m |\, \inter \De^m$ is a univalent map with bounded
distortion onto an annulus $A^m=A^m(R_\Sieg^m f)$ around $0$ slit along a
  quasiarc $J^m= J^m_f $ which is disjoint from $\De^m (R_\Sieg^m f)$;

\ssk \nin $ \mathrm { (ii)}$   
  $\mod ((A^m \sm J^m) \sm \De^m (R_\Sieg^m f) )\geq \nu (\bar N, \mu,
  H, \de)$; 

\ssk \nin $ \mathrm { (iii)}$ The pullbacks $\De_{-k}^m = \De_{-k}^m(f)$, $k=0,1, \dots,
\Bq_m,$ of $\De^m $ along the critical orbit form a closed chain 
consecutively attached one to another;
moreover, for $k< \Bq_m$, the interiors of the above pullbacks, $\inter \De_{-k}^m$, 
except for the first and the last
ones, are all disjoint;

\ssk \nin $  \mathrm { (iv)}$ All the maps $f^k: \De_{-k} \ra \De_0$,
  with $k< \Bq_m$,   are univalent with bounded distortion;
 the last map $f: \De^m_{-\Bq_m} \ra \De^m_{-\Bq_m+1}$ is a double branched covering;

\ssk \nin $  \mathrm { (v)}$
the sizes of the $\De_{-k}^m(f) $ are $(1+O(\de))$-comparable with those
of $\De_{-k}^m(f_\circ)$.  

\ssk \nin $  \mathrm { (vi)}$ There exist $\rho_m \to \infty$ such that all the maps $\pi_m |\, \De_m$
are $\rho_m$-expanding.
\end{lem}

The following modification of the previous lemma, uniform in $m$, 
is crucial. 

\begin{prop}\label {perturbed Delta-Siegel II}
Assume that the renormalization fixed point
 $f_\# \in \Siegel_\theta(\bar N, \mu, K)$ is hyperbolic (see Lemma
 \ref{hyp fixed point}).
Let  $f_\circ \in \Siegel_\theta(\bar N, \mu, K)$
(e.g., $f_\circ$ is a Siegel quadratic map), and let
$\FF=\{f_\la\}$ be a holomorphic family in $\BB_\circ$  through $f_\circ=f_{\la_0}$ transverse to
$\Sieg_\theta$ (e.g., $\FF$ is the quadratic family).
Then there exists  $\de>0$  such that
 for any $m\geq m_0(\bar N, \mu, K, \FF)\geq \underline{l}$,  
 any map  $f_\la$ with 
\begin{equation}\label{perturbed multiplier}
  |G^n (\theta(\la )) - G^m (\theta(\la_\circ )) |< \de, \quad
  {\mathrm{where}} \ \theta(\la)= f_\la'(0),\ G(\theta) = -1/\theta\
  \mod 1
\end{equation} 
has   polygons $\De^n_{-k} =\De^n_{-k}(\la) $
satisfying properties of Lemma \ref{perturbed Delta-Siegel}
and such that 
$$
   \diam \De^n_{-k} (\la) =O(\rho ^{-n })
$$
for some $\rho=\rho (\bar N) > 1$.
Moreover,  for any $\theta$ with $|G^n (\theta)-
G^n (\theta(\la_\circ)|< \de$ there is a map $f_\la\in \FF$   satisfying (\ref{perturbed multiplier}).  
\end{prop}

\begin{proof}
Since the Gauss map $G$ is expanding, we have:
$$
    |  \theta(G^j (\theta(\la)) - \theta(G^j ( \theta(\la_\circ) | <
    \de, \quad j =0,1,\dots, n.
$$
Since by Corollary \ref{Siegel renorm of famiies},
the renormalizations $R^m (\FF)$ form a precompact family of  families,
we obtain:
$$
   \dist (R_\Sieg^j f_\la, R_\Sieg^j f_\circ) < \eps =\eps(\FF, \de), \quad
   j =0,1,\dots n.  
$$
Hence we can select some $m$ for which
Lemma \ref{perturbed Delta-Siegel} is  applicable for
the renormalization  $R^m$ of any family $R^j (\FF)$, $j =0,1\dots,
n-m$.
The Telescoping Argument produces for us the desired  polygons
$\De^n_{-k}$. 

The last assertion follows from the  dynamical $\la$-lemma 
applied to $R_\Sieg$ near the hyperbolic fixed point $f_\#$. 
\end{proof}

Similarly to the unperturbed case,
the union $\BDe^n : = \bigcup_k \De_k^n $ will be  called  the {\it
perturbed Siegel  necklace} of $f$ of depth $n$.
We also let $\Om^n = \Om^(\la) $ be the {\it filled perturbed necklace} $\BDe^m$ obtained by
adding to $\BDe^n $ the complementary component containing $0$.

\begin{cor}\label{Om-domain}
Under the circumstances of the previous lemma, 
 the domains $\Om^n (\la) $ are contained in the $O(\rho^{-n})$-neighborhood
 of the Siegel disk $S_\circ$.  
\end{cor}

************}

\section{Inou-Shishikura class}\label{IS sec}

\subsection{Parabolic Renormalization}

Here we will briefly outline the Parabolic Renormalization Theory
that provides us with a good  control of bifurcations of parabolic
maps. It was laid down in the work by   Douady and Sentenac 
(see \cite{DH:Orsay,D-discont}), Lavaurs \cite{La},  
and Shishikura \cite{Sh}, 
which can be consulted for details.

\sss{Parabolic Piuseaut germs and their  transit maps}
\label{parabolic  sec} 

For $q\in \N$ and a small neighborhood $U$ of $0$,  
let $\GG_0(U)$ be the space of 
parabolic germs near $0$ given by Piuseaut series  
\begin{equation}\label{parabolic germ}
          f: z\mapsto z+z^2  + \sum_{k\in \N} a_k  z^{2+k/q},
\end{equation}
(continuous up to the boundary). By definition, 
it is isomorphic to the space of holomorphic germs
$$
       \hat f: \zeta \mapsto \zeta^q+\zeta^{2q}+\sum_{k\in \N} a_k \zeta^{2q+k} 
$$
on the neighborhood $\hat U$, the full preimage of $U$ under
the power change of variable $z = \zeta^q$. 
The latter space is  endowed  with uniform topology,
which is inherited by $\GG_0(U)$.   

Let us consider the
principal branch of $f$ (which is real on $\R_+$) in the slit plane
$\C\sm  ( i \R_-)$.  
It is endowed with the following structure: 

\ssk\nin (C1)
An  {\it attracting petal}  $\PP^a \equiv \PP^a (f)$,
which is an open 
piecewise smooth Jordan  disk with the following properties:

\ssk\nin $\bullet$
$\PP^a $ is $\R$-symmetric and $\PP^a \cap \R=(-\de, 0 )$ for some $\de>0$;

\ssk\nin $\bullet$
$\PP^a $ touches the origin at a certain angle $\alpha$ which can
be selected arbitrary in the range $(0, \pi)$. To be definite, we let
$\alpha=\pi/2$;

\ssk\nin $\bullet$  
$f$ univalently maps $\PP^a $ into itself, 
$
 f (\di \PP^a )\cap \di \PP^a =\{0\} ,
$
and $f^{n} (z)\to 0$ as $n\to +\infty$ uniformly on $\PP^a $.


\ssk 
Along with the attracting petal, there is a {\it repelling petal}
$\PP^r \equiv \PP^r (f)$ containing an interval $(0, \de)$ with some
$\de>0$  that can be defined as the attracting petal for $f^{-1}$. 

\ssk\nin (C2)
{\it The horn  map $H\equiv H_f : \PP^r \ra \PP^a$}.
  For any angle $\theta>0$,  
  there exist $\eps>0$ and $n\in \N$ such that for any point $z\in
  \PP^r$ with  $|z|<\eps$ and  $\arg z>\theta$  
(where $\arg z$ is the principal value of the argument)
we have $f^n z\in  \PP^a$. Moreover, $\eps$ and $n$ can be selected 
the same for all maps $\tl f\in \GG_0(U)$ near $f$.

\ssk\nin (C3)
The  attracting and repelling  {\it Fatou coordinates}%
\footnote{To make sure that the Fatou coordinates are isomorphisms
  onto the corresponding half-planes requires a special choice of the
  petals, which will also be assumed in what follows.} 
$$
    \phi^a\equiv \phi_f^a  : \PP^a\ra \{ \Re z> 0\}, \quad   
     \phi^r\equiv \phi_f^r: \PP^r \ra \{ \Re z < 0 \}.
$$
that conformally conjugate $f$ and $f^{-1}$ to the translations
$z\mapsto z+1$ and $z\mapsto z-1$ respectively.
The Fatou coordinates are defined up to translation, so they are
uniquely determined  by normalization that
specifies which points $c^{a/r} \equiv c^{a/r}_f \in \PP^{a/r}(f) $  correspond to 
$\pm 1\in \C$.
Moreover, if the base points $c^{a/r}_f$ depend holomorphically on $f$ 
then so do the normalized  Fatou coordinates.

\ssk\nin (C4)
 An  {\it  attracting fundamental crescent} $\CC^a  \equiv \CC^a (f)$.
It is the strip $\{ 1/2 \leq  \Re z \leq 3/2 \}$  properly embedded into the
attracting petal $\PP^a$ 
such that $\di \CC^a \cap \di \PP^a = \{0\}$
and $f  (\CC^a ) \cap \CC^a $ is  a  boundary component
of $\CC^a $.   To be definite, we will use the following choice:
$$
     \CC^a\equiv \CC^a(f)  = \{ z\in \PP^a: \   3/4 \leq  \Re \phi^a (z) \leq 7/4\} .
$$ 
Since the Fatou coordinate depends holomorphically on $f$,
the crescent $\CC^a(f)$ moves holomorphically with $f$. 

Similarly, one can define the {\it repelling fundamental crescent}
$$
\CC^r\equiv \CC^r (f) =  \{ z\in \PP^r: \   -1/4 \leq  \Re \phi^r (z) \leq -5/4\} .  
$$

\ssk\nin (C5)        
The  {\it \'Ecalle-Voronin cylinders} $\Cyl^{a/r}\equiv
\Cyl^{a/r}(f)$, which are
 the quotients of the petals $\PP^{a/r}$ by the dynamics.
They  can be 
obtained by identifying  the boundary components of the corresponding
fundamental crescents $\CC^{a/r }$ by means of $z \sim  f(z)$.
The normalized Fatou coordinates induce  isomorphisms of the pointed
cylinders $(\Cyl^{a/r}, c^{a/r} ) $ to the standard cylinder $(\C/\Z, 0) $,
and in what follows, we will freely identify the cylinders
 with the standard model.   


\ssk\nin (C6)
A  complex one parameter family of {\it transit } isomorphisms
\begin{equation}\label{EV iso}
          I_\la : \Cyl^a\isom \C/ \Z \ra \C/\Z\isom \Cyl^r, \quad z\mapsto z+\la,\quad \la \in \C/\Z.              
\end{equation}
For any    $\la\in \D_{1/4}$, the isomorphism
$I_\la$ lifts to  translation 
$$
  \{ 3/4\leq \Re z \leq 7/2\} \ra \{ \Re z <0\},\quad z\mapsto z-2+\la,
$$
which induces, by means of the Fatou coordinates  $\phi^{a/r}_f$, 
an  embedding 
\begin{equation} \label{lift I}
   I_{f,\la}  : \CC^a(f) \ra \PP^r(f).
\end{equation}
Holomorphic dependence of the Fatou coordinates on $f$ implies that these
embeddings  depend nicely on the parameters:

\begin{lem}\label{continuous dependence of I}
Assume the base points   $c^{a/r}_f\in \PP^{a/r}(f) $ are selected
holomorphically in $f$ over some neighborhood $\UU_0 \subset \GG_0(U)$.
Then  the  family of transit maps (\ref{lift I})  
depends holomorphically on $(f,\la)\in \UU_0 \times \D_{1/4}$.
\end{lem}

The horn map $H\equiv H_f$ from (C2) also descends to the cylinders, 
and we will keep the same notation, $H: \Cyl^r \ra \Cyl^a$, 
for the quotient. 

\ssk\nin (C7)  {\it Parabolic renormalization $R_\parab f$.} 
Composing the transit maps with the   horn  map, 
we obtain  a one-parameter family of   return maps
\begin{equation}\label{returns to C}
      I_\la\circ  H_f  :   \C/\Z \ra \C/\Z
\end{equation}
defined near the ends of the repelling cylinder $\Cyl^r\isom \C/\Z$. 
By means of%
\footnote{This special normalization of the exponential map is chosen 
   to make it consistent with the  one used by Inou and  Shishikura, see
  below.}  
$$\Exp: \C/\Z \ra \C^*, \quad \Exp(z) = -(4/27 ) e^{-2\pi i z}, $$  we can
identify  the cylinder  $\C/\Z$ with $\C^*$ so
that its upper end corresponds to  $0$ and the boundary of the 
fundamental crescents $\CC^{a/r}$  correspond to the ray $i\R_-$.
Then family of  return maps (\ref{returns to C})
becomes  a one-parameter family
$g_{ f, \la}$ of conformal germs near $0$.  

Moreover, there is a unique choice of the transit  parameter $\la$
that makes the map  $g_{f, \la}$ parabolic, with  multiplier $1$ 
at $0$. 
This map $g_{f,\la} $ is called   the 
{\it parabolic renormalization } $R_\parab f$    of $f$.

\sss{Transit maps for perturbations and their geometric limits}
   Let us now consider the space $\GG(U)$  of Piuseaut germs
(continuous  up to the boundary)   
\begin{equation}\label{perturbed parabolic map}
    f:    z\mapsto e^{2\pi i \gamma} (z + z^2) + 
               \sum_{k\in \N} a_k  z^{2+k/q} 
\end{equation}   
on $U$. 
We will refer to $\gamma\in \C/\Z$ as the {\it complex   rotation number} of $0$.


Let $\UU_0\subset \GG_0 (U) $ be a  neighborhood of a parabolic map $f_0$ 
Let us consider a neighborhood $\UU$ in $\GG(U)$
consisting of maps $f= e^{2\pi i \gamma} \tl f$, where $\tl f\in
\UU_0$ and  $|\arg \gamma | < \pi/4$.  

If $\UU $ is sufficiently small then
any map $f\in\UU\sm \UU_0$ 
has a second fixed point $\beta=\beta_f$ near $0$,
and there exist  crescent-shaped domains 
bounded by   (closed)  arcs
$\om^{a/r} = \om^{a/r}_f $ connecting $0$ to $\beta$ and 
their respective images $f^{\pm 1} (\om^{a/r})$. Moreover, all four
arcs, $\om^{a/r}$ and $f^{\pm a} (\om^{a/r} ) $ are pairwise disjoint
(except for the endpoints).
The domain $\PP=\PP (f)$  bounded by the arcs
 $\om^a$ and $\om^r$ will be referred as 
the {\it  petal } for $f$.%
\footnote{ What happens is  that the attracting and repelling petals
of a parabolic map ``merge'' under perturbation to form $\PP$.}

As in the parabolic case, the perturbed map can be linearized on its
petal. The linearizing coordinate
$$
    \phi \equiv   \phi_f : \PP  \ra \CC, \quad \phi(f z) = \phi(z) +1,
    \ z\in \PP\cap f^{-1} (\PP)
$$
is called the {\it Fatou-Douady coordinate} (or {\it perturbed Fatou
  coordinate}).  It is  defined uniquely up to translation, so it can
be normalized by prescribing a point $c^a\in \PP$ corresponding to $1$,
or a point $c^r$ corresponding to $-1$.  The normalized Douady
coordinate depends holomorphically on $f\in \UU\sm \UU_0$.

The petal $\PP$ can be selected so that $\phi(\PP)$ is a vertical
strip $\{  A < \Re z < B \} $ with  big $B-A$, 
and in what follows we will assume such a choice.  
Let us define the {\it attracting and repelling fundamental crescents} as
$$
   \CC^a\equiv \CC^a(f) = \{  A+3/4 \leq \Re \phi(z) \leq A+7/4\}, 
$$    
$$ 
  \CC^a\equiv \CC^a(f) = \{  B-5/4 \leq \Re \phi(z) \leq B-1/4 \}.
$$
If a point $c_f^a$ is selected in $\CC(f)$ holomorphic in
$f\in \UU $ (including  parabolic maps $f \in \UU_0$), 
then the linearizing coordinate
$\phi^a_f$ depends   holomorphically,  and hence continuously,  on $f\in \UU$.
Thus, 
if $f_n\to f$ then for any compact set $K\subset \PP^a (f)$,
the $\phi_{f_n} $ are eventually well defined on $K$, and
$\phi_{f_n}\to \phi_f$ uniformly on $K$.    

A similar discussion applies to the repelling fundamental crescents. 

The quotients of the petals $\PP^{a/r}$ 
 by  the dynamics  provide us with  a pair of 
{\it Douady cylinders}  $\Cyl^{a/r} =\Cyl^{a/r}(f)$.
They can be  obtained by  
identifying the boundary arcs of the crescents $\CC^{a/r}$ 
by  means of $z\sim f(z) $.
As in the purely parabolic case, the Fatou-Douady coordinate $\phi$
induces an isomorphism between the cylinders $\Cyl^{a/r}$ and the
standard cylinder $\C/\Z$,  and we we will freely identify the
cylinders with the standard model. 

Let us consider the transit map  $ T \equiv T_f : \CC^a \ra \CC^r$,
i.e., $Tz = f^j z$ where $f^k z\in \PP$, $k=0,1,\dots, j$, 
and $f^j z\in \CC^r$. It is usually discontinuous, but it induces a
conformal isomorphism between the cylinders,%
\footnote{
Notice an essential  difference with the parabolic case:  in that
case, there is a one-parameter family of isomorphisms between the
cylinders, all on equal footing, while  in the perturbed case, 
 (\ref{perturbed iso})  is a {\it preferred} isomorphism induced by the dynamics. }
\begin{equation}\label{perturbed iso}
   I_f  : \Cyl^a  \isom \C/\Z  \ra  \C/\Z \isom \Cyl^r, 
   \quad z\mapsto z+\la, \quad \la=\la(f)\in \C/\Z.
\end{equation}



\begin{thm}\label{parabolic bifurcation}
Assume the base points   $c^{a/r}_f\in \CC^{a/r} (f) $ are selected
holomorphically in $f$ over some neighborhood $\UU \subset \GG(U)$.  
Let  $(\La_f, c^r_f)  $ be the lift of $(\D_{1/8},0)\subset 
(\C/\Z,0)$ to $\PP(f)$ (by means of the Fatou-Douady coordinate). 
Then for every sufficiently big $j $, 
there exist a holomorphic embedding
$$
   \Phi_j :  \UU_0\times \bar \D_{1/8} \ra \UU,\quad (\tl f, \la) \mapsto  
      e^{2\pi i \gamma_j } \tl f,
$$
where $\gamma_{j, \tl f} : \bar \D_{1/8} \ra \C$ is a conformal embedding 
such that: 

\ssk\nin $\bullet$
$f= \Phi_j (\tl f, \la)$ for some 
$(\tl f, \la) \in \UU_0\times \bar\D_{1/8} $ iff
$$
\mbox{
$f^k(c^a)\in \PP$, $k=0,1,\dots, j$, $f^j (c^a)\in \La_f$, and
  $\la(f)  =\la $.
}$$ 

\ssk \nin $\bullet$ 
For $(\tl f, \la) \in \UU_0\times \bar\D_{1/8}$  and $\eps>0$,
let  $ f=\Phi_j(\tl f, \la)$ 
and 
$$ \CC^a_\eps (f)= \{ z:  1/2-\eps < \Re\phi^a_f (z) <  3/2 +\eps \}, $$ 
Then the  transit maps $f^j : \CC_\eps^a(f) \ra \PP(f)$ 
converge to the  parabolic  transit map
$I_{ \tl f, \la}: \CC^a (\tl f) \ra \PP(\tl f) $  uniformly on  compact subsets of $\CC^a (\tl f)$, 
and  uniformly over the tube $\UU_0\times \bar \D_{1/8}$.%
\footnote{Under these circumstances, the pair $(\tl f_\infty ,I_\la)$ is called the
{\it geometric limit} of the sequence $\{f_j \} $. } 
%

\ssk \nin $\bullet$  $\diam \Im \gamma_{j, \tl f}  \asymp j^{-2 }$.

\end{thm}

The images $Q^j= \Im \Phi_j$ will be called {\it parabolic tubes}. 
They are endowed with the {\it horizontal foliation} whose leaves
$\LL^j(\la)\isom \UU_0$, $\la\in \bar\D_{1/8} $, correspond to the same transit
parameter $\la\in \bar\D_{1/8}$.

\ssk
The  horn map  from (C2)  is robust 
under a perturbation  $f= e^{2\pi i \gamma } \tl f$
(\ref{perturbed parabolic map}).
The perturbed  horn  map 
$H\equiv H_f: \PP \ra \PP $ is  defined  for $z\in \PP$  with
$|z|< \eps$ and $0< \theta < \arg z < \pi/2$. 
It induces the cylinder horn map $\Cyl^r \ra \Cyl^a$  near the upper %
\footnote{The assumption that $arg \gamma> \theta$  breaks the symmetry
  between the ends  as it ensures that the
  points within a compact set of $\CC^a_f$ escape through the upper
  end of $\CC_f^r$. }  
 end of the
Douady cylinders. We will use the same notation $H\equiv H_f$ for this map. 

Composing it with the transit map $I_f: \Cyl^a\ra \Cyl^r$, 
we obtain   the return map 
$ I_f\circ H_f   : \Cyl_f^r \ra \Cyl_f^r$ near the upper end of the cylinders. 
Viewed in the
$\exp$-coordinate, 
it becomes a germ $g_f: (\C,  0)\ra (\C, 0)$. Its  rotation
number is given by the (modified) complex  Gauss map 
$G_*(\gamma)  =  -1/\gamma\ \mod \Z$.
If $G_*(\gamma)$ is small then this return map is close to the
parabolic renormalization of $\tl f$.
It  is  called the {\it almost parabolic renormalization}  
 of $f$. We will keep the same notation $R_\parab $ for this operator.


\sss{Case of  rotation number $p/q$} 
Let us now consider a holomorphic parabolic germ
\begin{equation}\label{p/q}
  f: \zeta\mapsto e^{2\pi i p/q} \zeta + \zeta^2+ \dots      
\end{equation}
with  rotation number  $p/q$.   Assume it is non-degenerate, 
i.e., it has $q$ petals (rather than a multiple of $q$ petals).
Then the  $q$-th iterate $f^q$ has a form
$$
     f^q: \zeta\mapsto \zeta+ a_{q+1} \zeta^{q+1} +\dots, \quad \mathrm{with}\ a_{q+1}\not=0.
$$
Performing a power change of variable $z=c \zeta^q$, we bring $f^q$ to
Piuseaut form (\ref{parabolic germ}).

Let us now perturb the parabolic map $f$  to
\begin{equation}\label{p/q perturbed}
  f_\eps: \zeta\mapsto e^{2\pi i( p/q+\eps) } \zeta + \zeta^2+ \dots  .  
\end{equation}
The $q$-th iterate $f_\eps^q$ has non-vanishing terms $a_kz^k$ with
$1< k< q+1$, but these terms can be killed by a conformal change of
variable. Performing further a power change of variable  $z=c \zeta^q$,
we bring $f_\eps$ to Piuseaut form (\ref{perturbed parabolic map}).
As all the above coordinate changes depend holomorphically on $f$,
this   allows us to apply the above theory to the space of germs
(\ref{p/q}). \note{enough?}

\subsection{Inou-Shishikura class}\label{IS class sec} 

Inou and Shishikura \cite{IS} have constructed a class $\IS_0$  of maps
with the following properties:

\ssk\nin (P1)
  Any map $f\in \IS_0$ is holomorphic on some quasidisk $\Om_f$ 
    containing $0$,  and  has a form $P_0\circ \phi^{-1}$ where
    $P_0$ is the restriction of $z\mapsto z(1+z)^2$ to some domain $\Om_0$,
    and $\phi: \Om_0\ra \C $ is an appropriately normalized 
    univalent map that admit a global qc extension to $\C$;

\ssk\nin  (P2) 
  $0$ is the parabolic fixed point of any $f\in \IS_0$;

\ssk\nin (P3) 
   Any $f\in \IS_0$ has a single  quadratic critical point $c_0=c_0(f)$;
moreover, the orbit of $c_0$ does not escape $\Om_f$, and 
 $f^n(c_0)\to 0$ as $n\to \infty$;

\ssk\nin (P4)
  The class  is endowed with the {\it Bers-Teichm\"uller  topology} and complex
  structure inherited from the space of Schwarzian derivatives $S\phi$;
  they make it  isomorphic to the Universal  Teichm\"uller Space;

\ssk\nin (P5)
   The class is also endowed with  {\it weak}  topology induced by  the
   compact-open topology on the space of univalent functions 
    $\phi  : \Om_0\ra \C$; the weak completion $\overline{ \IS}_0 $ is compact; 

\ssk\nin (P6) 
  The parabolic renormalization  $R$ acts from $\overline{ \IS}_0 $ to $
  \IS_0$; its restriction to $\IS_0$  is a compact holomorphic operator;

\ssk\nin (P7) 
 The parabolic renormalization of the quadratic map $z\mapsto z+z^2$
 has a restriction in $\IS$.


 
\ssk
For $\theta \in \R/ \Z$, 
define the class $\IS_\theta $ as $e^{2\pi i\theta } \cdot \IS_0$, and
let 
$
 \displaystyle{  \IS = \bigcup_\theta \IS_\theta.}.
$ 
(Notation  ${\overline\IS}_\theta$ and $\overline\IS$ has a similar meaning.) 
Property (P6) is robust under perturbation: 

\begin{thm} [\cite{IS} ] \label{IS thm}
If $\theta$ is sufficiently small then the almost parabolic
  renormalization $R_\parab$ induces an operator $R_\IS :
  \overline{IS}_\theta  \ra \IS_{-1/\theta} $ that restricts to 
a compact holomorphic  operator 
$R_\IS: \IS_\theta \ra \IS_{-1/\theta } $.
\end{thm}

We will call this operator $R_\IS$ (and in this section we will often
abbreviate it, without saying,  to $R$).

\begin{cor}\label{underline N}
  There exists $\underline{N}$ such that if $\theta=[N_1,  N_2,\dots
  N_m,\dots ]_*$   with $N_i >  \underline{N}$, $i=1\dots, m$,  then
  any map $\bar f\in {\overline\IS}_\theta$ is
  $m$ times renormalizable under $R_\IS$. Hence it is infinitely
  renormalizable  if $\theta$ is irrational.
\end{cor}


We say that a rotation number  $\theta\in \R/\Z$  (rational or
irrational) 
has {\it high type} if  all $N_i > \underline{N}$ with $\uN$ as above. 
Let $\IS({\underline{N}})$ stand for the union of  the spaces $\IS_\theta$
over all $\theta$ of high type.
For $\theta=[N,N,\dots]_*$ of high stationary  type
($N>\underline{N}$)   
we will also use notation $\IS_N\equiv \IS_\theta$. 
Similar notation will be used for the weak completion $\overline{\IS}$.

\subsection{Postcritical set}

  Inou and Shishikura have deduced from the above results

\begin{prop} [\cite{IS}]\label{postcrit set for IS}
    For any map $f\in \overline{\IS} ( {\underline{N}})$,  the critical
    point is non-escaping (i.e., $ f^n (c_0)\in  \Om_f$, $n=0,1,\dots$)
 and  stays away from the boundary of $\Dom f$.
 Thus, the postcritical set $\OO_f$ is compactly contained in 
$\Om_f$ (uniformly over $\overline{IS}$).  
In the parabolic case we have:  $f^n (c_0)\to 0$ as $n\to \infty$. 
In general, $\orb c_0$ is non-periodic.
\end{prop}    

\begin  {proof}    ({\it Sketch.}) 
  The mere fact that the IS renormalization $R f$
is well defined implies that
  the first $N_1$ iterates of the critical point stay in 
$\Om_f$ (where $N_1$ is the first entry of the rotation number). 
Existence of all the renormalizations imply that the whole
critical  orbit stays in $\Om_f$. 
Uniform bounds on the postcritical
set follow from compactness of $\overline{\IS}$.  

In the parabolic case, the map is finitely renormalizable and its last renormalization
falls to the class $\IS_0$. Property (P3) implies that  $f^n (c_0)\to
0$ as $n\to \infty$.  In the  irrational case, 
$f$ is infinitely renormalizable and all the renormalizations $R^m f$ 
are small perturbations of  parabolic maps of class $\IS_0$. 
Hence $R^m f (c_0)\not=c_0$.
On the other hand, if $c_0$ was periodic, then it would be the fixed
point for some renormalization.
\end{proof}

\subsection{Renormalization Telescope}\label{telescope}

In this section we will collect some technical results,
essentially contained in the work of Buff \& Cheritat \cite{BC} 
and Cheraghi \cite{Ch}. 

Given a map $f\in  \overline{\IS}_\theta$ and a  topological sector $\Sec$ centered at $0$, 
a {\em principal branch of the first return map to $\Sec$} is
an iterate  $f^l  : V\ra \Sec$, where $V$ is a relatively open subset of $\Sec$ with
$0\in \di V$ such that for any $z\in V$, $f^l(z)$ is the first return
of $\orb z$ to $\Sec$.
     
The following statement provides us with a convenient  domain of
definition for the renormalization change of variable:

\begin{lem}[\cite{Ch}, \S 2]\label{Davoud}
  For any map $f\in  \overline{\IS}_\theta$ with $\theta=[N_1,N_2,\dots]_*$ sufficiently small, 
there exists a smooth sector $\Sec= \Sec_f$  attached 
  to the fixed point $0$ with the following properties:

\ssk\nin {\rm (0)}  
  It has angle $\theta$ at $0$;

\ssk\nin {\rm (i)}  There exists  a bounded $s=s_f$ such that $f^s (\Sec)$
is a sector containing the critical value $c_1$ of $f$. In an appropriate Fatou
coordinate,%
\footnote{This coordinate is normalized so that  
         the critical value is placed at $1$.}
 the latter sector becomes the half-strip
\begin{equation}\label{Fatou strip}
      \{    3/4 \leq \Re  z \leq 7/4, \ \Im z \geq -2 \}. 
\end{equation} \note{choice is different from Davoud, OK?}

\ssk\nin {\rm (ii)} There exists a well defined 
 change of variable  $\pi=\pi_f: \Sec \ra \C$ which is univalent on $\Sec$ 
and  $\sim z^{1/\theta} $ as $z\to 0$ (uniformly over the class). Moreover, 
  $\pi (\Sec )\supset \Sec_{Rf} $, and  the boundary of $\pi(\Sec)$
touches the boundary of $\Sec_{Rf}$ at a single point, the fixed point $0$.

\ssk \nin {\rm (iii)}  
  The change of variable is equivariant:
it conjugates two  principal branches of  first return map to
$\Sec$ and $Rf$ on its full domain.%
%
\footnote{There is a precise formula for  the return times in
  terms of the arithmetic of $\theta$, see Lemma 2.2 in \cite{Ch}.}

\ssk\nin {\rm{ (iv)} }
For some $k$ independent of $f$, the union
$$
  \Om_f^1 =  \bigcup_{n = 0}^{  N_1+s-k }   f^n (\Sec) 
$$
is a neighborhood of $0$ compactly containing $\{ c_n \}_{n = 0}^{ N_1+s-k } $.

\ssk \nin {\rm (v)}  The sectors $\Sec_f$ depend continuously on 
$f\in \overline\IS (\underline{N} ) $. 
\end{lem}

For $t\geq 2$, 
 let $\De=\De_f(t)$ be the subset of the sector $\Sec_f$ corresponding to the
box
$$
     \{    3/4 \leq \Re  z \leq 7/4, \   -2 \leq  \Im z  \leq t \} 
$$
in the Fatou coordinate (compare (\ref{Fatou strip})).

\begin{lem}\label{Delta}
Under the circumstances of Lemma \ref{Davoud},
  for $t$ sufficiently big, the image $\pi_f (\De_f(t) )$ compactly contains
  $\De_{Rf}(t)$, with a definite space in between.
Moreover,  the domain $\De_f(t)$ depends continuously on $f$.  
\end{lem}

\begin{proof}
The last statement follows from item (v)  of  Lemma \ref{Davoud}
and continuous dependence of the Fatou coordinate 
of $f$.
 Together with the weak compactness of the Inou-Shishikura  class
$\overline{\IS}$
and item (ii) of the lemma,  this  implies that the change of variable $\pi_f$  
on $S_f$ is uniformly  comparable with $z\mapsto z^{1/\theta}$. 
This map is attracting near $0$, so the ``bottom'' of $\De_f$ 
(corresponding to $\{\Im z= t\}$ in the Fatou coordinate) 
goes even closer to $0$.  Together with item (ii) of the Lemma, this
implies that $\pi_f(\De_f(t))$ compactly contains $\De_{Rf}(t)$.
Using weak compactness of $\overline \IS$   once again, we conclude that there is a definite
space in  between.
\end{proof}

From now on, $t$ will be fixed, and will not appear in notation. 

If $f$ is $m$ times IS-renormalizable then we can compose the 
 above changes of variable,  to obtain a map
$$
  \pi^m_f = \pi_{R^{m-1} f} \circ \dots  \circ \pi_f,  
$$ 
well defined and univalent on a sector  $\Sec_f^m$ attached to $0$.
Spreading these sectors around by the iterates of $f$,  
we  obtain a neighborhood of $0$: 
\begin{equation}\label{Om-domains}
   \Om_f^m = \bigcup_{n=0}^{r_m}  f^n (\Sec_f^m), 
\end{equation}
where  $r_m$ is an appropriate time expressed in terms of the arithmetic
of $\theta$, 
and $f^n|\, \Sec^m_f$ is at most 2-to-1 for $n\leq r_m$
(note that these maps are not branched coverings over their images). 
Moreover, the iterate $f^{s_m-1} |\, \Sec^m_f$ 
(whose image $\Sec^m(c_0) \equiv \Sec^m_f(c_0)$  contains the critical point $c_0$) is univalent. We let
\begin{equation} \label {Pi_m def}
       \Pi_m\equiv \Pi^m_f =  \pi_m \circ f^{-(s_m-1) }  : \Sec^m(c_0) \ra \C , 
\end{equation}
where $ f^{-(s_m-1)} | \,  \Sec^m(c_0)$ 
 is the branch of the inverse map
 with image $\Sec^m$.

 The domain $\Om^m_f$ can be inductively obtained from
$\Om^{m-1}_{Rf}$
 by lifting the latter by an appropriate inverse
branch of $\pi_f$, and then applying $O(N_1)$ number of iterates of $f$
to ``close up the gaps''. 
 (See \S 2.2 of \cite{Ch}) for a detailed description). 

Property (P3) and Lemma 
\ref{Davoud} (iv) imply:

\begin{lem}\label{O is trapped}
  Let $f$ be an $m$ times IS-renormalizable map such that 
  $R^m f$ is a  parabolic map with multiplier 1.
Then the postcritical set $\OO_f$  is trapped inside  $\Om^m_f$. 
\end{lem}

Let us  also consider the lifts $\De^m_f$ of the  domains 
$\De_{R^m   f}$ under $\pi^m_f$. 
We let
$$
  \Neck^m_f = \bigcup_{n=0}^{r_m} f^n (\De^m_f),
$$ 
where the times $r_m$ are the same as in (\ref{Om-domains}).
Moreover,  $f^{s_m-1} $ maps $\De^m$  univalently and with
bounded distortion \note{precise ref?} 
onto its image $\De^m(c_0)\equiv \De^m(c_0)$ containing the
critical point $c_0$.  Thus,  change of variable $\Pi_m$
(\ref{Pi_m def})  restricted to $\De^m(c_0)$,
\begin{equation}\label{restricted Pi}
    \Pi_m : \De^m(c_0) \ra \C,
\end{equation}
is a univalent map with bounded distortion.
Notice also that by compactness of $\overline{\IS}$ and continuous
dependence of $\De^m_g$ on $g=R^m f$, the image of the restricted map $\Pi_m$ contains a
definite neighborhood of the critical point. 

Like the $\Om^m_f$, 
the sets $\Neck^m_f$ can be inductively constructed by lifting and spreading.
We call these sets {\it necklaces}. 

Lemma \ref{Delta} implies: 

\begin{cor}\label{Delta-m}
  The image $\pi^m_f(\De^{m-1}_f)$ compactly contains
  $\De_{R^mf}$, with a definite space in between.
There exist $\rho=\rho(\bar N) > 1$  such that
 $\diam \De^m_f = O(\rho^{-m} )$.
Moreover, for each $m$, the domain $\De^m_f$ depend continuously on $f$.  
\end{cor}
 
\begin{cor}\label{recurrence}
 Let $f\in \overline{\IS}_\theta$ be a map of IS class with irrational
 rotation number. Then the critical point is recurrent.
\end{cor} 

\begin{proof}
  Indeed,  the critical point returns to all the domains $\De^m f$,
and these domains shrink.
\end{proof}

\subsection{Siegel disks}

The next statement shows that maps $f\in \IS_\theta$ 
with  $\theta$ of high bounded type are Siegel maps:

\begin{prop}[\cite{Ya-posmeas}]\label{compact S}
    Let  $f\in \overline\IS_\theta$, where $\theta$ is a rotation number of
    high type bounded by some $\bar N$.
 Then $f$ is a Siegel map:  its
    Siegel disk $S_f$  is a quasidisk compactly contained in
 $\Om_f$,  and $\di S_f\ni c_0$. 
Moreover, 
 $f|\, \di S_f$ is quasisymmetrically conjugate to
   $\Bf_\theta|\, \di \BS_\theta$.
\end{prop}

\begin{proof}
By replacing $f$ with its IS renormalization $R f \in \IS$, 
we can assume that $f\in \IS$ (see Property (P6)). 

By \S \ref{Blaschke model},
    we know that the assertion is valid for the quadratic map 
$\Bf_\theta$ and hence for its renormalization  
$\Bg:= R (\Bf_\theta)  \in \IS_\theta$.  
Since $\IS_\theta$ is isomorphic to  the Universal
    Teichm\"uller Space,  
 any other map    $f\in \IS_\theta$ can be connected to $\Bg $ by a 
    holomorphic Beltrami path $f_\la$, $\la \in \D$.

Let $c_0(\la)$ be the critical point of $f_\la$, and let
$c_n(\la)=f^n(c_0(\la))$, $n\in \N$. 
By Proposition \ref{postcrit set for IS},  all points $c_n(\la)$ are
well defined, and then, they depend holomorphically on on $\la$. 
Moreover,  they do not collide: $c_n(\la)\not= c_m(\la)$ for
$n\not=m$ (by Proposition \ref{postcrit set for IS} and Corollary
\ref{recurrence}). 
Hence, they form a holomorphic motion over $\D$. 

By the $\la$-lemma, this motion extends to the 
postcritical set $\BO$ of $\Bg$,  and provides us with a family of quasisymmetric
homeomorphisms  $h_\la : \OO\ra \OO_\la$, $\la\in \D$,  where  
$ \OO_\la$ is the postcritical set for $f_\la$.
It follows that $\OO_\la$  
is a quasicircle for any $\la\in
  \D$, in particular, for the original map $f$.

Let  $D$ be a quasidisk bounded by $\OO_f$.  Then the family of iterates
$f^n$ is normal on $D$, so $D\subset S_f$. On the other hand, as 
the Siegel disk $S_f$
does not contain preimages of $c_0$,  which are dense in $\di D = \OO_f$,
 $S_f$  is contained in $D$.  
%
\end{proof}

\subsection{IS Renormalization fixed point}

Now the whole theory of Siegel maps developed in \S \ref{Siegel maps sec}
(external tilings, periodic points, trapping disks,  renormalization
fixed points, etc.)
is applicable to any class $\IS_N$, $N> \underline{N}$.


\begin{thm}[\cite{IS}]\label{hyp fixed point}
Let $\theta=\theta_N$ be  
a stationary rotation number of high type.
  Then the IS renormalization   $R $ 
has a unique hyperbolic fixed point  $f_\infty\in \IS_N$ .
The unstable manifold $\WW^u (f_\infty)$ is a complex curve that can be parametrized
by the complex rotation number ranging over a neighborhood of $[0,\theta]$.
Moreover, $R^n f \to f_\infty$ exponentially fast for any Siegel
map $f \in \II_N$.
 \note{space? size  of the unstable manifold?}
 \end{thm}

\begin{cor}\label{Siegel renorm of famiies}
Under the circumstances of the above lemma,
let us consider a holomorphic family $\FF\subset \BB_\circ$ 
passing through a Siegel map
$f_\circ \in \IS_N$ transversally to $\IS_N$. 
Then the sequence of the IS renormalizations $R^n (\FF) $,
$n=0,1,\dots$, is precompact, and in fact, it converges to the unstable manifold
$\WW^u(f_\infty)$.  
\end{cor}


\subsection{Perturbations of Siegel maps}

The above  control of one renormalization,
together with existence of the hyperbolic renormalization fixed point,
provides us with a good control of perturbations of  
Siegel disks of stationary type:


\begin{lem}\label {perturbed Delta-Siegel II}
Let  $f_\circ $ be a Siegel map of Inou-Shishikura class with stationary
rotation number $\theta_\circ=[N,N,\dots]_*$, 
and let $\FF=\{f_\la\}$ be a holomorphic family through $f_\circ=f_{\la_0}$ transverse to
$\IS_N $.
Then for any rotation number $\theta\in \R/\Z$,  \note{OK?}
there exists a map $f_\la\in \FF$ such that the  
renormalization $R^m f_\la$ 
with the same combinatorics as $R^m f_\circ$  is well
defined and has rotation number $\theta$.
Moreover,  
 the set $\Om^m_f$ is contained in $O(\rho^{-m})$-neighborhood
of $S_\circ$, where $\rho=\rho(N)>1$. 
\end{lem}

\begin{proof}
Existence of $f=f_\la$ follows from the Renormalization Theorem \ref{hyp fixed point}.
Moreover,  the renormalizations of $f_\la$ shadow those of $f_\circ$: 
\begin{equation}\label{shadowing-2}
  \dist (R^n f , R^n f_\circ) \leq C |\theta-\theta_0| \,
  \rho_0^{-{ (m-n) }}, \ n=0,1,\dots, m,
\end{equation}
where $\rho_0=\rho_0 (N) > 1$.

Let us now apply the lifting and spreading procedure to control the
 necklaces, and hence the $\Om^m$-domains. 
Assume we have already constructed a necklace $\NN^{m-n}_{R^n f}$
which is confined to a $\de$-neighborhood of $S_{R^n f_\circ} $ 
By Corollary \ref{Delta-m},  under sufficiently many further lifts, it will shrink
by a big  factor. Spreading this pullback around by a bounded
number of iterates of $R^{m-n-k} f $ , the necklace can be pulled
father away from
$S_\circ$ by  exponentially small (in $m-n)$) distance, see (\ref{shadowing-2}). 
These two  mechanisms imply  the desired.
\end{proof}


Together with Lemma \ref{O is trapped}, this leads us to the
following important conclusion: 

\begin{cor} [\cite{BC}]\label{confinement of PP}
Under the circumstances of Lemma \ref{perturbed Delta-Siegel II},
assume the map $R^m f_\la$ is parabolic with multiplier 1.   
Then the postcritical set $\OO_\la $ of $f_\la$  is contained in  the 
$O(\rho^{-m})$-neighborhood of 
 the Siegel disk $S_\circ$. 
\end{cor}

\comm{****
Finally, it is convenient to build a renormalization telescope near the
critical point:

\begin{lem}\label{De near c_0}
There exists an $m$ such that 
  for any map $f\in \overline\IS$ which is $m$ times renormalizable,  
there exists a definite neighborhood 
$\Ups_f \subset  f^{s-1} (\De_f)$ of the critical point  $c_0$,  
and an equivariant  change of variable
  $ \Pi_f$ defined on $ \Ups_f $ such that $\Pi_f (\Ups_f)
  \Supset \Ups_{R f} $ for some $m\in \N$, 
where
\begin{equation}\label{Pi_m}
\Pi_f^m= \Pi_{ R^{m-1}  f } \circ\dots \circ \Pi_f
\end{equation}

\end{lem}

\begin{proof}
Let $s_m$ be the moment such that $f^{s_m-1} (\De^m_f)\ni c_0$.
At this moment $f^{s_m-1} |\, \De^m_f $ is a univalent map with
bounded distortion onto its image $D^m$ . \note{OK?}
Hence  $\Pi^m := \pi_m\circ f^{-s_m+1} $  univalently maps $D^m$ onto
the domain $\Om_{R^m f} $ slit along the ray $i\R_-$.  
Hence its image 
contains well inside a definite neighborhood $\Ups_{Rf}$ of $c_0$.  

Let $\Ups^m_f :=  (\Pi^m))^{-1} (\Ups_f)$. 
By the Koebe Distortion Theorem, 
$\Pi^m : \Ups^m_f \ra \Ups^m$ has a bounded distortion.

Let us consider the
domain $f^{s-1} (\De_f)\ni c_0$ and the change of variable 
$\Pi_f:= \pi_f\circ f^{-s+1}$ on it, where $f^{-s+1}$ is the inverse branch 
$ f^{s-1} (\De_f ) \ra \De_f $.
 
Iterating, we obtain changes of variables  
$\Pi^m_f  : \Ups^m \ra \Ups_{R^m f} $ (\ref{de near c_0}) 
equal to $\pi_m \circ f^{-s_m=1} $ for the inverse branch
of the map $f^{s_m-1} \De^m_f$,...

By Corollary \ref{Delta-m},  map $\Pi^m_f$ (\ref{Pi_m})  is
 expanding for $m$ big enough,  
so the pullback of $\Ups_{R^m f }$ under $\Pi^m_f$ is compactly contained
in $\Ups_f$, with a definite space in between.  

\end{proof}
******************}

\comm{***********

The above statement can be refined as follows:

\begin{prop}\label {perturbed Delta-Siegel II}
Let  $f_\circ \in \Siegel_\theta(\bar N, \mu, H)$
(e.g., $f_\circ$ is a Siegel quadratic map), and let
$\FF=\{f_\la\}$ be a holomorphic family in $\BB_\circ$  through $f_\circ=f_{\la_0}$ transverse to
$\Sieg_\theta$ (e.g., $\FF$ is the quadratic family).
Then there exists  $\de>0$  such that
 for any $m\geq m_0(\bar N, \mu, H, \FF)\geq \underline{l}$,  
 there exists a  map  $f_\la$ with 
\begin{equation}\label{perturbed multiplier}
  |G^m (\theta(\la )) - G^m (\theta(\la_\circ )) |< \de, \quad
  {\mathrm{where}} \ \theta(\la)= f_\la'(0),\ G(\theta) = -1/\theta\
  \mod 1.
\end{equation} 
Moreover, this map has   domain $\De^n_s (\la) $
satisfying properties of Lemma \ref{perturbed Delta-Siegel}.
\end{prop}

\begin{proof}
Existence of $f_\la$ follows from the hyperbolicity of the
renormalization fixed point $f_\#$.
Existence of the domains $\De^n_s (\la) $
follow from Lemma \ref{perturbed Delta-Siegel}.
%
%
\end{proof}

\sss{Perturbed Siegel postcritical set}
The following  important result  gives control of   the postcritical set 
 under perturbations of a Siegel disks:

\begin{prop}[\cite{BC}]\label{confinement of PP}
   Let  $f_\circ \in \IS_\circ= \IS_{\theta_\circ} $ with $\theta_\circ $ of high bounded type,
and let $\FF=\{f_\theta \}$ be a real analytic  family through $f_\circ=f_{\theta_\circ}$
transverse to $\IS_\circ$.
For any $\eps>0$, if  $\theta$ is a rotation number of high type 
which is sufficiently close to $\theta_\circ$ then its postcritical
set $\OO_\theta= \OO_{f_\theta} $ is contained in  the 
filled necklace $\hat \BDe^m$. 
\end{prop}

************************}

\comm{*****

\begin{proof}
Take a large $n\in \N$.
If $\theta$ is sufficiently close to $\theta_\circ$ (depending on the
family $\FF$) then the map $f_\theta$ is $n$ times renormalizable
under $R_\Sieg$.   By Proposition \ref{postcrit set for IS},
the postcritical set of $R^n f_\circ$ is contained in $\Dom f_\circ$.
Hence the postcritical set of $f$ is contained in the domain $\Om^n $
defined before Corollary \ref{Om-domain}. Application of that Corollary
concludes the proof. 
\end{proof}

************}

\bignote{Parabolic flower is hidden here}

\comm{*********

\sss{ Parabolic flower}

Applying the previous discussion to a parabolic perturbation of a
Siegel map, we obtain:

\begin{prop}\label{parabolic flower} 
Under the circumstances of Proposition \ref{confinement of PP},  
assume that $\theta=\theta_\kappa$ is a parabolic approximand $r_\kappa/q_\kappa$ of
$\theta_\circ$.  Then $f_\theta$ 
has  an invariant  flower $\Flower= \Flower_\theta$ with $q=q_\kappa$ petals
$
     \Petal_i,  \ i=0,1,\dots, q-1,
$ 
with the following properties:   \note{picture}

\ssk\nin ${\mathrm {(i)}}$
$ \Petal_0\ni c_0$; 

\ssk\nin $ {\mathrm {(ii)}}$
 The {\it flower} $ \Flower = \bigcup \Petal_i$ is forward invariant
(and hence contains the postcritical set $\OO_{ f}$); 

\ssk\nin  ${\mathrm {(iii)}}$
  It is is contained in an $\eps(\kappa, \FF)$-neighborhood of the Siegel
   disk $S_\circ$, where $\eps(\kappa, \FF)\to 0$ as $\kappa\to \infty$
   with the rate depending only on the geometry of $\FF$;%
\footnote{meaning that the bound is uniform as long as $\FF$
stays within a precompact family of families} 

\ssk \nin $ {\mathrm {(iv)}} $
  It projects to two ends of the repelling  Ecall\'e-Voronin cylinder
(so, the complement of the flower on the cylinder is a bounded
annulus);  


\ssk\nin $ {\mathrm {(v)}} $
For a given $\kappa$, 
$ \Pi:=  \CC^r \sm  \Flower$ is a rectangle with bounded geometry,
depending only on the geometry of  $\FF$;

\end{prop}

 We let $ \Petal^\sym_0$ be the domain attached to $c_0$ which is
dynamically symmetric to the petal $ \Petal_0$ (in some neighborhood of
$c_0$ depending only on the geometry of $ f$).  

****************}

\section{Construction of an example}\label{construction sec}

\subsection{Outline}
Let us start with a rough description of 
our example. Take a big $l\in \N$, a bigger $\kappa\in \N $,
 and an even much bigger $m\in \N$. 
Begin with a {\it Siegel} quadratic polynomial 
$$
   \Bf =\Bf_\theta: z\mapsto e^{2\pi i \theta} z + z^2
$$ 
with a stationary rotation number of high type,
and consider its cylinder renormalization $f=R_\Sieg^{m-\kappa}\Bf $.  
It is a Siegel map of Inou-Shishikura class.

Moreover, $f$  has a  distinguished repelling periodic
point $\alpha= \alpha^l $ of period $q_l$
(that approximates the dynamics on $\di S_f$ in scale  $l$).
Perturb $f$ to a {\it parabolic}
approximand  $\tl f $ with rotation number $p_\kappa/q_\kappa$.  
Then $\alpha$ gets perturbed to a periodic
point $\tl \alpha $ with  the same period. 

Furthermore, using the theory of parabolic bifurcation,    one can perturb $\tl f$ to a
{\it Misiurewicz} map $f_{\mathrm{Mis}}$ for which  $\tl \alpha $ becomes a
postcritical point $\alpha_{\mathrm{Mis}}$. 
Since $\alpha_{\mathrm{Mis}}$ can be approximated with precritical points, 
$f_{\mathrm{Mis}}$  can be further perturbed to a  {\it
  superattracting} map $f_\circ$.

The last  map  can be anti-renormalized to obtain a superattracting
quadratic polynomial $\Bf_\circ$ such that $f_\circ=R_\Sieg^{m-\kappa} \Bf_\circ$.
This quadratic polynomial
determines a renormalization combinatorics. 
The unique infinitely renormalizable quadratic polynomial $\Bf_*$ with
this combinatorics  is  desired. 

Our construction depends on six large integer parameters $N, l,
\kappa, t$, and $ m, j$, selected consecutively as listed,
where the last two play somewhat different role than the first four.
 Once we select one of the first four parameters, 
we assume, sometimes  without saying, that all the rest  depends on this choice.
A statement  {\it For any consecutively selected} $(N,l, \kappa)>
(\underline N, \underline l, \underline \kappa)...$  (or {\it For any
consecutively selected sufficiently big $(N,l, \kappa)$}...) will mean
$$
   \exists \uN \quad \forall N> \uN \ \exists\, \ul=\ul(N) \quad \forall \, l> \ul\
   \exists\, \ukappa =\ukappa (N,l) \quad \forall\, \kappa > \ukappa \dots
$$  
We will also assume  that the choice $\ul(N)$ is made {\it
  monotonically increasing } in $N$, the choice of $\ukappa (N,l) $
  is monotonically increasing in each variable,   
and similarly for any other parameter in question. 

Let us now supply the details.

\subsection{Perturbed periodic points and trapping disks} 
\label{level l}

\sss{General perturbations}


Recall that $\Siegel (\bar N, \mu, K)$ stands for the space of Siegel maps
$f: (\Om, 0) \ra (\C, 0) $ introduced in \S \ref{def of Siegel maps}.
%

When we perturb $f$ below,  we will use the uniform metric on $\Om$. 

\begin{lem}\label{definite disc}
There exist natural numbers%
\footnote{In the polynomial case, we can let $\underline{l}=1$.}
$\underline{l}$ and  $\iota$ depending on  $(\bar N, \mu, K)$
such that for any   $l\geq  \underline{l}$,
there exists a  $\de_0= \de_0 (\bar N,\mu, K, l) > 0 $  
with the following property.
For any $\de<\de_0$,
if a holomorphic map $\tl f: \Om \ra \C $ is $\de$-close to a Siegel map
$f: \Om \ra \C$ of class $\Siegel (\bar N,\mu, K)$   
 then 

\ssk \nin {\rm{(i)}}
There exists  a periodic point $\tl \alpha^l $ of  period
 $q_l $  which  is a perturbation%
\footnote{meaning that $\tl \alpha^l$ is 
 $\eps(\de)$-near $\alpha^l $ 
where $\eps(\de)\to 0$ as $\de\to 0$. }
 of the $\alpha^l$;

\ssk \nin {\rm{(ii)}}
There exists a  collar%
\footnote{Objects associated with $\tl f$ are usually marked with
  ``tilde'', but it can be skipped if the object is independent of
  $\tl f$, e.g., $\tl A^l \equiv A^l$,   $\tl D^l\equiv  D^l$, etc.}      
$A^l$  in $\Om \sm \bar S_f$ such that:
it is impossible to jump over it under $\tl f$:
$$
 \mathrm{ If}\  z\in  \Comp_0 (\C\sm A^l) , \ 
  \tl f(z)\not\in \Comp_0 (\C\sm A^l) , \ 
  \mathrm{then} \ \tl f(z)\in A^l  ;
$$

\ssk \nin {\rm { (iii) } }
There exists a trapping quasidisk
$\Disc^l \Subset \Om \sm \bar S_f$  
with  bounded shape  around $\tl\alpha^l$    
whose hyperbolic diameter  in $\Om \sm \bar S_f$ is  of order 1;  moreover, 
$$
  D^l \cap D^{l+\iota}= \emptyset;
$$

\ssk \nin {\rm{(iv)}}
A definite part of the disk $ D^l $ is contained in 
$\tl f^{-1} (S_f) \sm \bar S_f$;  
moreover, there is a point $\tl\beta\in \tl f^{-1} (\di S_f)\sm  \bar S_f$
that lies in the middle of $\Disc^l $;

\ssk \nin {\rm{(v)}}
If $z \in A^l  $ then at  some moment $k< q_{l+1}$, 
 $f^ k z$ lands in the middle of   $ \Disc^l $, while
$$
   f^i z  \in \Comp_0(A^{l- 2\iota}) \sm D^{l-\iota}, \quad i=0,1,\dots, k.
$$ 

\ssk
All geometric bounds depend only on $N$,  $\mu$, and $K$.
\end{lem}

\begin{proof}  
  The properties of Proposition \ref{disk D for Siegel} are manifestly robust under
  perturbations, keeping the same collars $A^l$ and trapping disks $D^l$. 
(The auxiliary collars $A_0^l$ and disks $D^l_1$,
as well as the  collars $A^{l - 2\iota}$ in the  last statement,    
were designed to secure robustness.)
\end{proof}

As before, we say that the trapping disk $D=D^l$ is centered at
$\alpha^l$, or that $\depth D= l$.  

\comm{***
We let $\tl A^l $ be the union of the  pullbacks of the
domain  $Y^l =Y^l_f$ under the same number of iterates of $\tl f$  
as was used to construct the collars
$A^l$ in the circle and Siegel cases. Since only bounded number of iterates
are used, the collar $\tl A^l$
is a small perturbation of the collar $A^l$. 
Hence   properties (T0)--(T3) 
are preserved under this perturbation.

Using these collars,
the domain $\tl \Disc^l$ can be constructed  in the same way as
the domain $\Disc^l$ in Proposition \ref{disc D} 
or in its Siegel counterpart \ref{disk D for Siegel}. 

All other asserted properties  are clearly robust under perturbations
of any $f\in \Siegel (\bar N, \mu, K)$.
Since the space $\Siegel (\bar N, \mu, K)$ is compact,
all of them are inherited by any $\tl f$ under consideration.
%
%
***********}


\comm{****
\sss{Transit from $\tl\CC^r$ to $\tl \alpha^l$}

In case of a parabolic perturbation,  we can make a transit from the 
repelling crescent to a neighborhood of  $\tl\alpha^l$:

\begin{lem}\label{point in petal}
Given $\bar N$,  $\mu>0$, $K>0$, and $l\in \N$,  
let   $\de_0= \de_0 (N,\mu,K,l) $  be from Lemma \ref{definite disc}, 
and let $\de<\de_0$.
Let   $f: \Om \ra \C$ be a Siegel map of class $\Siegel (N,\mu, K)$, 
and let   $\tl f: \Om \ra \C $ be a  $p_\kappa/ q_\kappa$-parabolic approximand
which is $\de$-close to $f$. 
Then  there exists $\bar s = \bar s(\kappa, \de )$ and  a point $\tl a=\tl a_{\tl f} $ in the repelling
crescent $\tl \CC^r$  of $0$ such
that  $\tl f^s (\tl a)\in  \Disc^l $ for some $s\leq \bar s$,
and this happens before the orbit of $\tl a$  passes through the collar 
$ A^{ l-\iota } $.
Moreover, the point $\tl a$ can be selected  
in the middle of  $\tl \CC^r$.%
\footnote{
 i.e., so that  $\dist (\tl a, \di \tl C^r) $ is comparable with $\diam \tl \CC^r$
(with  the constants depending only  on $N,\mu,K,l$ and $\kappa$)} 
\end{lem}

\begin{proof}
Let us first fix such a parabolic map $\tl f: \Om \ra \C$.
 Then  the crescent $\tl\CC^r$  contains a  point
  $\tl a$  that escapes  the domain $\Om $. 
(For otherwise, 
the union of the repelling and attracting petals would
form a neighborhood of $0$  on which
the family of iterates, $\{\tl f^n\}_{n=0}^\infty$, would be well defined and normal.)
 This forces  $\orb \tl a$ to pass through the 
quasidisk $ \Disc^l $ at some moment $s=s_{\tl f} $. 

If we fix $\kappa$ and $\de$, then we obtain a compact family of maps $\tl f$, 
and the fundamental crescent  $\tl\CC^r $ can be selected in a continuous way.
This allows us to make a continuous choice of $a_{\tl f}$,
which puts $\tl a$ in the middle of $\tl \CC^r$ and  makes the escaping time $s $ bounded.
\end{proof}

**************}

\sss{Expansion} 

For a perturbation $\tl f$ of a Siegel map $f$,
we will use notation 
$R^l_\Sp \tl f : X^l_+\cup X^l_-\ra \C$ for the corresponding
perturbation of the butterfly renormalization
$R^l_\Sp f$.  

Away from the Siegel disk,
Corollary \ref{expansion for Siegel maps} is robust under
perturbations:

\begin{lem}\label{expansion for perturbed maps}
Let $f: \Om \ra \C$ be a Siegel map of class $\Siegel (\bar N,\mu, K)$.  
For  any $\eps>0$
there exists   $\de= \de (\bar N,\mu, K; \eps) > 0 $  
with the following property.
Let  $\tl f: \Om \ra \C $  be a holomorphic map which  is $\de$-close
to $f$,   and  let $z\in X^l_+\cup X^l_-$  
be a point with the property that 
$R^l f (z)\in Y^l $ and $\dist (R^l f(z) , \bar S) \geq  \eps$.  
Then 
$$
     \| D( R^l f) (z) \|_\hyp \geq \rho >1 
$$ 
with $\rho$ depending only on $(\bar N, \mu, K)$ and $\eps$.
\end{lem}

In turn, it implies a perturbed version of Corollary \ref{expansion
  for Siegel maps-2}:

\begin{cor}\label{expansion for perturbed maps-2}
Let $f: \Om \ra \C$ be a Siegel map of class $\Siegel (\bar N,\mu, K)$.  
For  any $\eps>0$,
there exist $a>0$,   $\rho >1$, and $\de$
with the following property.
For any holomorphic map $\tl f: \Om \ra \C $   which  is $\de$-close
to $f$,  
if $z\in Y^m $,  $f^n z\in Y^{m-k}$ for some $n\in \N$,  $0< k< m$
(with $m-k > \ul $),  while 
$$
  \dist (f^i z , \bar S) \geq \eps,\quad  i=0,1,\dots,n,
$$
then 
$$
    \| Df^n (z) \|_\hyp \geq a \rho^k,
$$ 
where the norm is measured in the hyperbolic metric of  $\C\sm \bar S$.
\end{cor}

\sss{ Cylinder renormalization of polynomial maps}

To make the exposition more transparent, we will focus on the
stationary case when $\theta=\theta_N$ is a stationary rotation
number with   $N > \underline N$.
 Let $  \Bf = \Bf_\theta: z\mapsto e^{2\pi i \theta}z + z^2$   be the corresponding
Siegel quadratic polynomial,
and let $\tl\Bf=\Bf_{\tl \theta}$ be its polynomial  perturbation
(where $\tl\theta$ is not necessarily real).
By the Inou-Shishikura theory, all cylinder renormalization of $\Bf$
are well defined and belong to the IS class:
\begin{equation}\label{m-2 kappa}
    f_i   =   R_\Sieg^i  (\Bf )\in \IS_\theta ,\quad i = 1, 2, \dots
\end{equation}
 Moreover, for any $n$,
if $\tl\theta$ is sufficiently close to $\theta$, 
then the same is true for the first $n$ cylinder
renormalizations of $\tl \Bf$.
In this case, we let
\begin{equation}\label{tl f}
    \tl f_i  =  R_\Sieg^i  (\tl \Bf )\in \IS_{G_*^i  \tl \theta} ,
      \quad  i = 1, 2, \dots, n, 
\end{equation}
where $ G_*: \gamma \mapsto -1/\gamma\ \mod \Z$ 
is the modified and complexified Gauss map. 


Theorem \ref{hyp fixed point} and  its Corollary
\ref{Siegel renorm of famiies} provide us with a good
control of the maps $\tl f_i $:

\begin{lem} \label{shadowing}
There exist positive  $\mu,  K, \eps_0, C$,  and $\rho>1$ 
depending  only   on $N$
such that:

\ssk \nin $\bullet$ $f_i \in \Siegel( N, \mu, K), \quad i=0,1,\dots$;

\ssk \nin $\bullet$ 
For any $\gamma\in \C$ which is $\eps_0$-close to $\theta$  and any  $n \in \N$, there exists a
unique $\tl \theta$ such that the cylinder renormalizations $\tl f_i $,
$i =0,1,\dots, n$, 
are well defined, and $\tl f_n $ has
complex rotation number $\gamma$; 

\ssk \nin $\bullet$
  $$
   \dist (f_i , \tl f_i) \leq  C\dist (f_n, \tl f_n)  \, \rho^{-(n-i)}, \quad
   i =0,1,\dots, n. 
  $$

\ssk \nin $\bullet$ The Siegel maps $f_i$ converge to the Siegel renormalization
fixed point $f_\infty$, while the nearby maps $\tl f_i $ converge to a  map
$\tl f_\infty$ in the unstable manifold $\WW_\Sieg (f_\infty)$.   
\end{lem}

\comm{****
\sss {Safe trapping disks}

\begin{lem}\label{safe trap disks}
  For any natural $j\in [l, m-\kappa+l]$, the map $\Bf_*$ has a collar
  $\BA^j$ and a trapping disk
$\BD^j\ni \Bal^l$ satisfying the properties of  
 Lemma \ref{definite disc}.  In particular, 
$\BD^j$ is disjoint from 
$\BD^{j-\iota}$ for any $j\in [l,m-\kappa+l-\iota]$, where $\iota=\iota(N)$ .  
\end{lem}

\begin{proof}
Let $\de<\de_0$ and $\iota$ be  as in Lemma \ref{definite disk}.
  Let us consider the Siegel renormalizations $R_\Sieg^{m-\kappa-s\iota} \Bf_*$ with
 $k\geq \kappa$, $s=0,1,\dots, [(m-kappa)/\iota]$. 
Theorem   ...ensures that these renormalizations
stay  $\de$-close to the corresponding renormalizations 
 $R_\Sieg^{m-\kappa- s\iota} \Bf $  of the Siegel polynomial $\Bf$. 
Hence for each  $R_\Sieg^{m-\kappa- s\iota} \Bf_*$,  and $j\in [l, l-\iota]$, 
 Lemma \ref{definite disk}
provides us with collars a $A^j (R_\Sieg^{m-\kappa-s\iota} \Bf_*)$ and trapping
disks  $D^l(R_\Sieg^{m-\kappa-s\iota} \Bf_*)$.   
 
\end{proof}
***********}

\sss{Parabolic approximand $\tl\Bf$}
We will now specialize a perturbation $\tl\Bf=\Bf_{\tl \theta}$ of the 
Siegel polynomial $\Bf=\Bf_\theta$.
Take two natural numbers $\kappa< m$.
Let $\tl\theta= p_m/q_m$  be  the (modified)  continued fraction
approximand to $\theta$, so that $\tl \Bf$ is the parabolic quadratic
polynomial with rotation number $p_m/q_m$ at $0$.
It is  $m$ times cylinder  renormalizable with all the
renormalizations $\tl f_i = R_\Sieg^j \tl\Bf$, $i\geq 1$,  in the IS class.
Moreover, $f_i $ is parabolic with rotation number $p_{m-i }/q_{m-i }$ at
$0$.  We will consider the maps 
\begin{equation}\label{f_m}
       f_{m-\kappa} =   R_\Sieg^{m-\kappa}  (\Bf ) \in
     \IS_\theta, \quad
     \tl f_{m-\kappa} =   R_\Sieg^{m-\kappa}  (\tl \Bf ) \in \IS_{p_{\kappa}/ q_{\kappa}},
\end{equation}
and their limits $f_\infty$ and $\tl f_\infty$. 
To simplify notation, we will often skip the subscript $m-\kappa\in \bar \N$ letting
$$
    f  \equiv  f_{m-\kappa},  \quad \tl f  \equiv \tl f_{m-\kappa},\ m\in \bar \N_\kappa. 
$$

By Lemma \ref{shadowing}, $\tl f$ is $\de$-close to 
$f: \Om\ra \C$ for $\kappa$ big enough, so Lemma
\ref{definite disc} is
applicable, providing us with  the trapping discs $ D^l $
and the collars $A^l$.

\comm{*****
\begin{lem}\label{delta-close}
For any $\eps>$ there exists $\underline{\kappa}$
such that for any natural $\kappa>\underline{\kappa}$ 
and any $m>\kappa$, we have:
$  \dist (f_{m-\kappa} , \tl f_{m-\kappa}) < \eps$, $m\in \bar \N$.
\end{lem}

\begin{proof}
Let $\BFF$ be the quadratic family.
Then $R_\Sieg (\BFF)$ is a family of IS class. 
   By Corollary \ref{Siegel renorm of famiies}, 
its renormalizations $R^{m-\kappa} \BFF$ belong to a compact space of
families.  
Since $p_\kappa/ q_\kappa\to \theta$ as $\kappa\to \infty$,
where the $p_\kappa/q_\kappa$ are the rotation numbers  of the 
$\tl f_{m-\kappa}\in R_\Sieg^{m-\kappa} (\BFF) $,
we conclude that 
$$
\dist(\tl f_{m-\kappa}, f_{m-\kappa} ) \to 0 \ \mathrm{as}\  
\kappa\to \infty \quad
{\mathrm{uniformly\ in}} \ m\in \N.
$$
Since the families $R^{m-\kappa}_\Sieg (\BFF) $ converge to the unstable
manifold $\WW^u(f_\infty)$ (Corollary \ref{Siegel renorm of famiies}), the assertion is also
valid for $m=\infty$.
\end{proof}
**********}

%

\sss{Transit from $\tl\CC^r$ to $\tl \alpha^l$}

For the parabolic map $\tl f=\tl f_{m-\kappa}$,
we let:

\ssk\nin $\bullet$
$\tl \CC^r$  be its the repelling crescent;

\ssk\nin $\bullet$
 $ \tl \De^\kappa$ be the domain of the renormalization change of
 variable $\tl \pi_\kappa$, see \S \ref{telescope};

\begin{lem}\label{point in petal}
For any consecutively selected $N$ and $l$, 
there exists $\underline \kappa$ such that  for any natural
 $m> \kappa> \underline \kappa$,
 the parabolic map $ \tl f= \tl f_{m-\kappa} $ has the following property.
There exists $\bar s = \bar s(N,l,\kappa )$ and  a point $\tl a\in \tl
\CC^r \cap \tl \De^\kappa$
such that  $\tl f^s (\tl a)\in  \tl \Disc^l $ for some $s\leq \bar s$,
and this happens before the orbit of $\tl a$  passes through the collar 
$ A^{ l-\iota } $, where $\iota=\iota(N)$.
Moreover, the projection $\tl \pi_\kappa (\tl a) $ lies
in the middle of  the repelling crescent $ \CC^r(\tl f_m)$,
with a constant depending on $\kappa$ but independent of $N$ and $l$.
\end{lem}

\begin{proof}
The range  $\tl \pi_\kappa (\tl\De^\kappa) $ contains an annulus 
$\{ \eps<  |z|< r \} $ 
with a definite $r$ and a small%
\footnote{
How small it is depends on  the truncation level $t$ defining the domains
$\De^\kappa$, see \S \ref{telescope}. }
$\eps$, slit along the straight ray $i\R_-$.
Moreover, 

\ssk\nin $\bullet$
The ray does not intersect the repelling
crescent $\CC^r(\tl f_m)$ (since  the crescent is contained in 
the $\R_+$-symmetric wedge  of size $\pi/2$);

\ssk\nin $\bullet$
$\eps$ is so small that the truncated crescent 
$$
  \CC^r_\tr  (\tl f_m)  := \CC^r(\tl f_m)\cap \D_\eps
$$
is contained in an  attracting crescent  of $\tl f_m$ 
(by property (C2) of \S \ref{parabolic  sec} and  compactness of $\bar \IS_0$). 

\ssk
The above truncated crescent lifts under $\pi_\kappa$ 
to  a truncated  crescent $\tl\CC^r_\tr$ for $\tl f$.
The latter   contains a  point $\tl a$  that escapes  the domain $\Om $. 
(For otherwise, 
the union of the repelling and attracting petals would
form a neighborhood of $0$  on which
the family of iterates, $\{\tl f^n\}_{n=0}^\infty$, would be well defined and normal.)
By Lemma \ref{definite disc},  this forces  $\orb \tl a$ to pass through the 
trapping disk $ \Disc^l $ at some moment $s $ before it passes through
the collar $A^{l-\iota}$ with $\iota=\iota(N)$.

If we fix $\kappa$, then we obtain a compact family of maps
$\tl f\in \bar \IS_{p_\kappa/q_\kappa}$, 
and the fundamental crescent  $\tl\CC^r $ can be selected in a locally
continuous way.
This allows us to make a locally continuous choice of $\tl a_{\tl f}$,
which, by compactness,  
makes the escaping time $s $ bounded and 
puts $\tl a$ in the middle of $\tl \CC^r_\tr$.
Since $\tl \pi_\kappa$ has a bounded distortion on $\tl \De^\kappa$, 
this puts  $\tl \pi_\kappa(\tl a)$ in the middle of 
$\CC^r_\tr (\tl f_m)$.  
\end{proof}

\sss {Pullback of $ \Disc$}

\begin{lem}\label{univalent pullback}
For any consecutively selected $N$ and $l$, 
there exists $\underline \kappa$ such that  for any natural
 $m> \kappa> \underline \kappa$
 the parabolic map $ \tl f_{m-\kappa} $ 
 has the properties of  Lemma~\ref{point in petal}, and 
 the trapping disc $ D=  D^l$ 
can be univalently and with bounded distortion pulled back
to $\tl a$ along the orbit $\{\tl f^i \tl a\}_{i=0}^{  s}$.
Moreover, the whole  pullback $\{ \tl D_{-k} \}_{k=0}^s$
is contained in 
$\Comp_0   (\C\sm A^{l-\iota } )   $  
for some $\iota=\iota(N) $, 
while the last domain  $\tl D_{-s}$ is contained in the repelling crescent
$\tl \CC^r$.  
\end{lem}


\begin{proof} 
By Proposition \ref{confinement of PP},  for  $\kappa$ big enough,
 the postcritical set $\tl \OO$ of $\tl f$ stays close to $S=S_f$. Since
 $\Disc $ is
 contained  well inside $\Om \sm \bar S$, it is also contained well inside
 $ \Om  \sm \tl \OO$.  So it has a bounded hyperbolic diameter in
 $ \Om \sm \tl \OO $. 

Let us consider the parabolic map
$\tl f_m= R^m_\Sieg \tl\Bf = R^\kappa_\Sieg (\tl f_{m-\kappa}) $
with multiplier $1$ at the origin.    
By Lemma  \ref{point in petal},  
there is an escaping  point  $\tl a$ in  $\tl \De^\kappa$
such that $\tl\pi_\kappa(\tl a)$ lies in the middle of the repelling
crescent   $\CC^r  ( \tl f_m ) $,  while $\tl a_s \equiv\tl f^s (\tl a)$ lands in
 $D=D^l$. 

 Corollary \ref{expansion for Siegel maps-2} implies that
 for any $\eps>0$, if  $\kappa$ is sufficiently big,
there is  $k\leq  s$ such that

\ssk\nin (i)
 $ \tl a_{s-k} $ is $\eps$-close to the critical point  $\tl c_0$;

\ssk\nin (ii) 
 $D$ can be univalently pulled back along the orbit 
$\{ \tl a_{s-n}  \}_{n=0}^k$; let $\tl D_{-n}$ denote the corresponding disks;

\ssk\nin (iii) 
The hyperbolic diameter of $\tl D_{-k}$ in  $\Om \sm \tl\OO$ is less than $\eps$. 

\ssk
The last property  allows us to  enlarge $\tl D_{-k }$ to a  disk 
$ \tl D_{-k }' \Subset \Om\sm \tl \OO$ such that
\begin{equation}\label{pair of disks}
\mod (\tl D_{-k }' \sm \tl D_{-k} ) > \mu,  \quad
  \diam_\hyp \tl D_{-k}' < \eps,
\end{equation}
where  $\mu = \mu (\eps) \to \infty$ as $\eps\to 0. $
%

Properties (i) and (iii) ensure that  $\tl a$ lies in the range of the 
renormalization change of variable $\tl \pi_{m-\kappa}$, so 
it can be lifted  to a point $\tl \Ba$ in the domain 
$\tl \BDe^{m-\kappa}$ of $\tl \pi_{m-\kappa}$. 
By equivariance of $\tl \pi_{m-\kappa}$,
there exist moments $\Bs$ and $\Bs-\Bk$ such that points
$\tl \Ba_\Bs $ and $\tl \Ba_{\Bs-\Bk}$ belong to $\tl \BDe^{m-\kappa}$ 
and project by $\tl \pi_{m-\kappa}$ to $\tl a_s$ and $\tl a_{s-k}$ respectively. 

Moreover,  the  disks $\tl D\ni \tl a_s $ and $\tl D_{-k}' \supset \tl D_{-k}
\ni  a_{s-k} $ lift by $\tl \pi_{m-\kappa}$ to disks $\tl \BD\ni \tl \Ba_\Bs$
and   $ \tl \BD_{-\Bk}' \supset  \tl \BD_{-\Bk} \ni \tl \Ba_{-\Bk} $ in $\C\sm \tl\BO$.
Since $\tl\Bf$ is a global polynomial map, the disks 
$\tl \BD_{-\Bk}' \supset \tl \BD_{-\Bk}$ can
be further pulled back to  disks 
${\tl \BD_{-\Bs} }'  \supset \tl \BD_{-\Bs} \ni \Ba_{-\Bs} $
in $\C\sm \tl \BO$.

As we know (see \S \ref{expansion sec}), 
this pullback contracts the hyperbolic diameter in
$\C\sm \tl \BO$. Since $\BD_{-\Bk}' $ has a small hyperbolic diameter,
see (\ref{pair of disks}), 
so does $\tl \BD_{-\Bs}'$.
Hence it has a small  Euclidean diameter
compared with $\dist( \tl\Bc_{-\Bs}, \tl\Bc_{q_m -\Bs} ) $,  
\note{$q_m$, center ?} 
where $\tl\Bc_{-\Bs} $ is the center of $\tl \BDe^m$.  
On the other hand, $\diam \tl\BDe^m$ is comparable with the latter
distance,
and we conclude that $\tl \BD_{-\Bs}' \subset \tl \BDe^m$. 

We can now apply to $\tl\BD_{-\Bs}'\supset \tl\BD_{-\Bs}$ the renormalization change of variable
$\tl \pi_{m-\kappa}$ to obtain disks $\tl D_{-s}'\supset \tl D_{-s} \ni \tl a$ in
$\tl \De^\kappa\sm \tl \OO$ which are univalent pullbacks of the disks
$\tl D_{-k}'\Supset \tl D_{-k}$. 
Moreover, the change of variable $\tl \pi_\kappa$ is well defined on these
disks,  and
$$
    \mod ( \tl \pi_\kappa (\tl D_{-s}' ) \sm  \tl \pi_\kappa (\tl D_{-s}))  =    
   \mod  (\tl D_{-s}'\sm \tl D_{-s})  =  \mod  (\tl D_{-k}'\sm \tl D_{-k}) 
\geq \mu,
$$
with a  big $\mu$, see (\ref{pair of disks}).
Hence the hyperbolic diameter of $\tl \pi_\kappa (\tl D_{-s} ) $ in $\Om\sm
\tl \OO(\tl f_m) $ is small.
Since $\tl \pi_\kappa ( \tl a ) $ lies in the middle of the repelling crescent of
$\tl f_m $,  the disk $\tl\pi_\kappa (\tl D_{-s} ) $ lies  inside the crescent.
\comm{
**************
Since $\{\tl f^i \tl a\}_{i=0}^{ s} \subset \Comp_0 (\C\sm  A^{ l- 2\iota } ) $,
it is contained in some $\tl \Om^{l- \nu }$,  
where $\nu =\nu (N) $ is independent of $l$ and $\kappa$.
By equivariance of the renormalization change of variable, 
 the backward orbit  $\tl a_{ s-k}$, $k=0,1,\dots s$,  lifts to 
a backward orbit 
 $$
    \{\tl \Bf^{ \Bs -\Bk} \tl \Ba \}_{\Bk\in \mathbf K} \subset 
 \tl    \BOm^{m-\kappa+l   -\nu }, 
$$
where $\mathbf K\subset \N$ is the set of $s+1$  moments corresponding
to $k=0,1,\dots, s$.

Since the disk $\Disc $ is close to $\tl c_\circ$ for $l$ big enough,
it is contained in the range of the change of variable
$\tl\Pi_{m-\kappa}$ (\ref{Pi_m}). Let $\tl\BD $ be its lift under $\tl\Pi_{m-\kappa}$.
The disk $\tl\BD$ has a bounded hyperbolic diameter in 
$\C\sm  \tl{ \mathbf{ O}}$  (where $\tl {\mathbf{ O}} =  \OO_{\tl \Bf}$).
and hence it can be univalently and with bounded distortion  pulled
back along  the  orbit 
 $\{ \tl \Bf^{\Bs - \Bk} \tl\Ba\}_{\Bk\in \mathbf K} $.
Let  $\{ \tl \BD_{-k}\}  $ stand for this pullback.  

Since the pullback  procedure  contracts the hyperbolic metric, 
the discs $\tl \BD_{-k} $ have a bounded hyperbolic diameter as well. 
Hence the Euclidean $\diam \tl \BD_{-k}$ are  bounded (with some constant) by  
$\dist (\tl \BD_{-k},  \tl{ \mathbf{O} } )$. 

Since  the distance from the boundary of 
$\tl \BOm^{m-\kappa +l - j  }$ to $\tl {\mathbf{O}} $ grows
exponentially with $j$ (and uniformly in $m, \kappa$ and $l$),
the disks $\tl  \BD_{-k}$ are contained in 
$\tl \BOm^{m- \kappa+l - j }$ for $ j $
sufficiently big (depending only on $N$).
By the equivariance of the renormalization change of variable. 
the disks $\tl D_{-k}$ are contained in $\tl \Om_{l-j}$. 

In fact, by Lemma \ref{},  the hyperbolic diameter of $\BD_{\Bs-k}$
shrinks by  a definite factor every time this domain returns to $\Ups$ while staying away from
the Siegel disk $\BS$. Hence, for any $\eps >0$,  if $\kappa$ is sufficiently big, then
then the  $\diam_\hyp (\BD_{\Bs-k} ) < \eps$ 
for some  domain $\BD_{\Bs-\Bk_0 }\Subset \Ups\sm \BO$ (and the hyperbolic
diameter is measured in $\Ups\sm \BO$).  
This allows us to  enlarge $\BD_{\Bs-\Bk_0 }$ to a domain 
$ \hat \BD_{\Bs-\Bk_0 }$ with 
$$
\mod (\hat \BD{\Bs-\Bk_0 } \sm \BD{\Bs-\Bk_0 }) > \mu (\eps) \to \infty\
\mathrm{as}\ \eps\to 0. 
$$
Let $\hat \BD_{\Bs - k}$, $k=\Bk_0, \dots \Bs$, be the corresponding 
pullback of the bigger disk.


To see that $\tl D_{-s}$ is contained in $\tl \CC^r$, let us apply
Lemma \ref{perturbative expansion}.
If $\kappa$ is sufficiently big, 
then there are many moments $k $ such that  the disk $ \tl\BD_{-k }$ 
stays  bounded hyperbolic distance from $\C\sm \tl \Bf^{-1}( \tl \BO )$
(measured in $\C \sm \tl {\mathbf O} $).
At these moments it is
contracted by a definite factor  under the pullback. Hence the final
pullback $\tl \BD_{-\Bs}$ has a small hyperbolic diameter in $\C\sm \tl\BO$.

The crescents $\tl \CC^r$ and $\tl\BC^r$ are related by
(an appropriate branch) of  the power
change of variable $\Bz= z^{q_\kappa/q_m}$ ($z\in \tl\CC^r$),  where
$m\geq \kappa$.  
Since $\tl a$ is located in the middle of $\tl\CC^r$,  it follows that
$\tl \Ba$  is located in the middle of $\tl\BC^r$.
Since the disk $\tl \BD_{-\Bs}$ has a small hyperbolic diameter, 
it is contained well inside $\tl\BC^r$. 
Using  the power relation between the crescents once again,
we conclude that $\tl D_{-s} $ is contained well inside $\tl\CC^r$. 
******************}
\end{proof}

Passing to the limit as $m\to \infty$ (using Lemma \ref{shadowing}), we conclude:

\begin{lem}\label{univalent pullback for f_infty}
There exists $\underline \kappa$ such that  for any natural
 $\kappa> \underline \kappa$
 the map 
$$
  \tl f_\infty= \lim_{m\to \infty}  \tl f_{m-\kappa}\in \WW^u_\Sieg(f_\infty) 
$$
has the properties listed in Lemma \ref{univalent pullback}.
\end{lem}

\subsection{Various connections}

By a {\it connection} between two points, $z$ and $\zeta$, 
we mean a trajectory passing from a small neighborhood of $z$ to a
small neighborhood of $\zeta$. 

\sss{Connection between $\tl c_0$ and $0$}  

Property (P3) of the Inou-Shishikura class (\S \ref{IS class sec}) 
and compactness of the space $\Siegel ( N, \mu, K)$
(with $\mu=\mu(N)$ and $K=K(N)$ as in Lemma \ref{shadowing})
 imply:

\begin{lem}\label{n}
  There exists an $\bar n= \bar n (N,\kappa) $ such that  for any  parabolic map
$\tl f= \tl f_{m-\kappa}$, $m\in \bar \N_\kappa $,
 we have: $\tl f^n (\tl c_0)\in  \tl \CC^a$ for some $n \leq\bar n $.
\end{lem}

\sss{Connection between $\tl \alpha^l$ and $\tl c_0$}

Let us now make a connection between the periodic point $\alpha^l$ and
the critical point $c_0$:

\begin{lem}\label{u-mand of zeta}
   For any $(N,\mu,K)$ there exists $\underline{l}$ with the following property.
For any natural $l> \underline {l}$ and  any $\rho>0 $ there exists $\underline{t}$
   such that for any  $t>\underline{t}$,
  any Siegel map
   $f\in \Siegel (N, \mu, K)$  has a
$t $-precritical point $c_{-t}$ in the $\rho $-neighborhood of  the
periodic point $\alpha^l $.
Moreover, the orbit $\{ c_n \}_{n=-t}^0$ is contained  
well inside $\Om^{l-\nu} $ with
$\nu$ depending only on $ (N,\mu, K) $.
 In particular, all these properties are valid uniformly
for the maps $f_{m-\kappa}$, $m\in \bar \N_\kappa$.   
\end{lem}

\bignote{Add comments on $m=\infty$ throughout}

\begin{proof}
Let $\eps  =\si\cdot \dist(\alpha^l, c_0)$ with a small $\si\in (0,1)$, 
and let $W$ be the
$\eps^2$-neighborhood of the critical value $c_1$.
 Any point $z\in W\cap \di S$, except $c_1$ itself, has a preimage
    $z_{-1}\not\in \bar S$. Let $k$ be the first moment when the backward
    orbit $\{ c_{-n} \} $ of $c_0$ (along $\di S$) lands in $W$. Then
    $k =k(N,\mu,K; l)$ 
and $\dist(c_1,  c_{-k})\asymp \eps^2$ (with a constant depending on
$(N,\mu,K)$ only).

   The point $c_{-k}$ has a  preimage  $c_{-k-1}\not\in \bar  S$ such that
$$    
        \dist (c_{-k-1}, c_0) \asymp \dist (c_{-k-1} , \bar S ) \asymp \eps.
$$ 
It follows that if $\si$ is sufficiently small then $c_{-k-1}\in Y^l$
and the hyperbolic distance  
$d: = \dist_\hyp  (c_{-k-1}, \alpha^l) $  in $Y^l$  is bounded.
(Here  $Y^l$ corresponds through the surgery to 
  the range of the  holomorphic circle pair from Theorem 
 \ref{bounds for comm pairs}).

Let $ D\ni c_{ -k-1}$ be the hyperbolic disk in $Y^l$ of radius $2d$ centered at $\alpha^l$.
By the Schwarz Lemma, 
  $f^{-q_l} (D)$ is a subset of  $D$ of bounded  hyperbolic  diameter
(where $f^{-l_q}$ is the inverse branch fixing $\alpha^l$). 
  A few more (of order $-\log \rho$)
pullbacks of $c_{-k-1}$ by $f^{-q_l}$  will bring our point to the
$\rho$-neighborhood of $\alpha^l$. 

Since this backward orbit stays in $D$, it is trapped inside
$\Comp_0(\C\sm A^{l-\iota} )$ with $\iota=\iota(N)$.  Since  points of $\di A^{l-\iota}$ lie
on depth $l-\iota$, while those of $\di \Om^{l-\nu}$ lie on
depth $l-\nu$,  we see  that   $A^{l-\iota}$ is contained well inside
 $ \Om^{l-\nu}$ for $\nu$
big enough (depending on $(N,\mu, K)$ only).  
The conclusion follows.
\end{proof}

 The above  connection is robust: 

\begin{cor}\label{perturb u-mand of zeta}
 For any $(N,\mu,K)$ there exists $\underline{l}$ with the following property.
 For any  natural  $l>\underline l$ and any 
$\rho >  0  $ there exist  $\underline t$ and  $\de_0>0$  
such that for any $\de<\de_0$  and any natural $t>\underline{t}$,   the following holds. 
If a map $\tl f $ is $\de$-close to a Siegel map  $f\in \Siegel(N,\mu,K)$ 
 then it has a
$t$-precritical point $\tl c_{-t}$ in the $\rho $-neighborhood of  the
periodic point $\tl \alpha^l$.
Moreover, the orbit $\{ \tl c_n \}_{n=-t}^0$ is contained in 
$\tl \Om^{l-\nu}$ with $\nu=\nu(N,\mu, K)$. 
In particular, these properties are valid for any parabolic map 
$\tl f_{m-\kappa}$, $m\in \bar \N$.
\end{cor}

\sss{Connection between $0$ and $\tl \alpha^l$}


\begin{lem}\label{c_(-s-t)}
For any consecutively selected sufficiently big   $N,l, \kappa $ 
and any  $\rho >  0  $,  there exist  $\underline t$ and $\bar s$  
such that for any natural $t>\underline{t}$ and some $s\leq \bar s$,   the following holds. 
For any parabolic map $\tl f= \tl f_{m-\kappa}$, $m\in \bar \N$,  
there exists a precritical point $\tl c_{-s-t}$ lying in the middle of the repelling crescent
$\tl \CC^r $ such that $\tl f^s (\tl c_{-s-t} ) = \tl c_{-t}$ where
$\tl c_{-t}$ is $\rho$-close to the periodic point $\tl \alpha^l$,  
and the orbit $\{ \tl f^i (c_{-s-t}) \}_{i=0}^s$ is contained in
$\tl \Om^{l-\nu}$ with some $\nu=\nu(N)$. 
\end{lem}

\begin{proof}
By Lemma \ref{shadowing},
for $\kappa$ sufficiently big, 
 $\tl f_{m-\kappa} $ is close to $f_{m-\kappa} $, uniformly in 
$m\in \bar\N$. 
Hence  we can apply 

\ssk\nin
-- Lemma \ref{definite disc}
to find a trapping disk $D\equiv D^l$ around $\alpha^l$;

\ssk\nin 
-- Lemma \ref{point in petal} 
to find  $\bar s$ and a point $\tl a\in   \tl \CC^r$ such that
$\tl f^s \tl a \in D$ for some $s\leq \bar s$;

\ssk\nin
-- Corollary \ref{perturb u-mand of zeta} to find,
 for any $t> \underline t$,  
a precritical point $\tl c_{-t} \in D $ which is $\rho$-close to $\alpha^l$.

\ssk
   By Lemma \ref{univalent pullback}, 
the trapping disk $D$ can be univalently pulled back to the
 point $\tl a $. Moreover, this pullback is contained in $\tl
 \Om^{l-\nu}$ for some $\nu=\nu(N)$,  and the last domain $\tl D_{-s}$ 
is compactly contained in the repelling crescent $\tl \CC^r$. 
The  corresponding pullback of the precritical point 
$\tl c_{-t}\in D$ gives us the desired point $\tl c_{-s-t}$.
\end{proof}

\sss{Transit from the repelling crescent to the attracting  one}

Combining the last lemma with  Lemma \ref{n} 
and Corollary \ref{perturb u-mand of zeta}, we obtain:

\begin{cor}\label{precritical point in rep cres}
For any consecutively selected sufficiently big  $N,l,\kappa$,
 and for any $\rho>0$,
there exist  $\underline  t$, $\bar n$ and $\bar s$ with the following properties. 
For any $m\in \bar\N_\kappa$ and any  $t> \underline t$,  
there exist  $n\leq  \bar n $ and $s\leq  \bar s $  
such that  the parabolic map $\tl f= \tl f_{m-\kappa}$ has 
a precritical point  $\tl c_{-s-t}\in \tl \CC^r$  and a postcritical point
$\tl c_n \in \tl \CC^a$ such that  the
whole orbit $\{ \tl c_k\}_{k=-s-t}^n$ is trapped  in $\tl \Om^{l-\nu}$
with some $\nu=\nu(N)$. 
\end{cor}

Recall that $\tl \Bf=\Bf_{p_m/q_m}$ is the parabolic  quadratic
polynomial with rotation number $p_m/q_m$,
and  that  $\tl \BC^{a/r}$ stand for the attracting and repelling crescents for
$\tl\Bf$. As  $\tl f_{m-\kappa} = R_\Sieg^{m-\kappa} \Bf$,
we obtain:

\begin{cor}\label{precritical point in rep cres for polynom}
The points $\tl c_{-s-t}$ and $\tl c_n$ from   
Corollary \ref{precritical point in rep cres} lift to
 a precritical point  $\tl \Bc_{-\Bs-\Bt}\in \tl \BC^r$ and a postcritical point $\tl
\Bc_\Bn \in \tl \BC^a$ for the parabolic polynomial $\tl \Bf$
such that  the whole orbit $\{ \tl\Bc_k\}_{k=-\Bs-\Bt}^\Bn$ is trapped in
$\tl \Om^{m-\kappa+l-\nu }$ with $\nu=\nu(N)$. 
\end{cor}



\subsection{Quadratic-like Renormalization}
\sss { Superattracting parameter}\label{super par}


Let us now perturb the parabolic  map $\tl f\equiv \tl f_{m-\kappa}$, $m\in \bar\N_\kappa$, to a
superattracting  map $f_\circ\equiv f_{m-\kappa, j;\circ}$, $j\in \N$, that  will determine the desired 
renormalization combinatorics.
Its superattracting cycle%
\footnote{We will mark the objects related to $f_\circ$ with a
  subscript or superscript ``$\circ$''.
On the other hand, 
for  the (pre-/post-) critical points $c_k$,
we skip (here and below) subscripts indicating their dependence  on various parameter's:
$m, j$, etc. } 
 $\{ c_k^\circ\}_{k=0}^{p-1}$ follows the following route:

\ssk\nin $\bullet$
first, it  passes from the critical point
  $ {c}^\circ_0$ to a postcritical point $  {c}^\circ_n$ in the
 attracting crescent $\CC^a_\circ $
 (where $n$ comes from Lemma  \ref{n});

 \ssk\nin $\bullet$
 then it  goes through the parabolic
 gate to a precritical point  $c_{-t-s}^\circ $ 
in the   repelling crescent  $\CC^r_\circ$
(where $s$ and $t$ come from Lemmas \ref{point in petal} 
and \ref{u-mand of zeta});

\ssk\nin $\bullet$
then it penetrates trough the
 boundary of the virtual Siegel disk $S_f$, approaches a periodic point
    $\alpha_\circ$ just  missing it  to land  at $ c^\circ_{-t}$;

\ssk\nin $\bullet$ 
   and finally, it  returns to  $ c^\circ_0$. 

\ssk Here is a formal statement: 

\begin{lem}\label{superattracting map}
Let $\theta=\theta_N$ be a stationary rotation number of high type $N>\uN$, 
 and let $l>\ul$ be a level selected in Lemma \ref{definite disc}.
For any $\de>0$, 
there exists $\underline{\kappa}= \underline{\kappa} ( N,l; \delta) $ 
such that for any $\kappa>\underline{\kappa}$, some
  $n< \bar n(\kappa) $,  $s< \bar s(\kappa)$, 
and any $t> \underline{t}(\underline {\kappa} ) $ and
 $m \geq  \kappa $, 
  there exists a superattracting  map 
$$
   f_\circ\equiv f_{m-\kappa,j;\circ} =R_\Sieg^{m-\kappa} ( \Bf_\circ) : \Om \ra \C
$$  
$\de$-close to the parabolic map $\tl f $  (\ref{f_m}),
with a superattracting cycle   of period $p=n+j+s+t$,
such that near the critical point  $c_0^\circ$ we have
$$
   f_\circ ^p=  f_\circ^{s+t}\circ I_\circ\circ f_\circ^n, 
$$
where $I_\circ:  \CC^a_\circ \to \CC^r_\circ$ is a transit map
between the crescents of $f_\circ$.
Moreover, the whole  cycle  of $c_0^\circ$  is contained in 
$\Comp_0 (\C\sm  A_\circ^{l-\iota} ) $ with some $\iota=\iota(N)$.

The same properties hold for the limit map 
$$
    f_{\infty,j; \circ} =\lim_{m\to \infty}   f_{m-\kappa, j; \circ}
$$
in the unstable manifold of the renormalization fixed point
 (compare  Lemma \ref{univalent pullback for f_infty}). 
\end{lem}

\begin{proof}
Let us consider postcritical  point $\tl\Bc_\Bn\in \tl \BC^a$ and
a precritical point $\tl \Bc_{-\Bs-\Bt}\in \tl \BC^r$
from  Corollary \ref{precritical point in rep cres for polynom}.
Let $\BI : \Cyl^a\ra \Cyl^r $ be the isomorphism between the
cylinders such that $\BI (\Bc_\Bn) = \Bc_{-\Bs-\Bt} $.
%
By the Parabolic  Bifurcation Theory (Theorem \ref{parabolic bifurcation}) 
for any $j$ sufficiently big,   $\tl \Bf$ can be perturbed to a 
superattracting polynomial map $\Bf_\circ\equiv \Bf_{j; \circ} $
for which
$$
 \Bf_\circ^j (\Bc_\Bn^\circ) = \Bc_{-\Bs-\Bt}^\circ .
$$
Let $f_\circ=R_\Sieg^{m-\kappa} (\Bf_\circ)$ for $m\in \N_\kappa$.
The desired properties for these maps, and their limit is $m\to
\infty$, are evident.
\end{proof}

\sss {Quadratic-like families for parabolic maps}

Similarly, we can  construct the whole quadratic-like family 
with  the desired renormalization combinatorics:

\begin{lem}\label{ql family on Cyl}
For any consecutively selected  natural 
$( N,l, \kappa, t ) > (\underline{N}, \underline{l},  \underline{\kappa},
\underline{t}) $, 
and any  $m\in \bar\N_\kappa$, 
any parabolic  map $\tl f_{m-\kappa } $   admits  a family of  transit maps $ I_\la: \tl \Cyl^a\ra \tl \Cyl^r$,
  $\la \in \La$, 
with the following properties.
There is a family of disks
 $\tl U_\la\subset \tl V$ around $\tl c_0$ and moments 
$ (n,s )\leq  (\bar n ,\bar s)$  
(from Lemmas \ref{n} and  \ref{point in petal}) 
such that:

\ssk\nin {\rm { (0) }} $V$ is a quasidisk with bounded dilatation and
definite size depending only on $(N, l, \kappa)$;%
\footnote{In fact, for given $(N,l,\kappa)$, 
the disk itself can be selected independently of $t$; 
for $m$ sufficiently big,
it can be selected independently of  $m$ either.}

\ssk\nin {\rm{(i)}}
 The maps 
\begin{equation}\label{F_la}
   \tl F_\la = \tl f^{s + t} \circ  I_\la \circ \tl f^n  : 
        \tl U_\la    \ra  \tl V 
\end{equation}
form  a proper  unfolded quadratic-like family over $\La$;

\ssk\nin  {\rm{(ii)}}
  The closures of all intermediate disks,
$$
   \tl f^k (\tl U_\la), \ k=0, 1, \dots, n, \quad \mathrm{and} \quad
   \tl f^k \circ I_\la\circ \tl f^n  (\tl U_\la), \ k=0,1, \dots s+t-1,
$$
 that appear in  composition (\ref{F_la})
  are pairwise disjoint;

\ssk\nin  {\rm{(iii)}}
$$
 \bar\mu(N,l,\kappa,t) \geq   \mod ( V\sm \tl U_\la)  
   \geq \underline{\mu} ( N, l,  \kappa, t)\to \infty
   \ \mathrm{as}\ t \to \infty \ \mbox{ with $N,l,\kappa $ fixed }.
$$

\ssk\nin  {\rm{(iv)}}
In case of connected Julia set $J(\tl F_\la)$ 
(i.e., when $\la$ belongs to the Mandelbrot set
$\MM'_{N,l,\kappa,t,m}$ of q-l family (\ref{F_la})),  
the disk $\tl U_\la$ is an $L( N,  l, \kappa,t)$-quasidisk with   
$$
     \area \tl U_\la \geq  c ( N,  l,  \kappa, t )>0 .
$$


\ssk
All constants and bounds are independent of $m$.
\end{lem}

\begin{proof}
In the $\tl f$-plane, select a disk  $\tl V\ni\tl  c_0$,
and let $\tl V_{-i}$, $i=0,1,\dots, t$,  be its  pullback to $\tl
c_{-t}$. Let us show that if $\tl V$ is small enough, 
depending on $N,l,$ and $ \kappa$ but independently of $t$,
then the closures of these disks are pairwise disjoint.
Consider a linearization domain $W$ around the
periodic point $\tl a_l$ (so, $f^{q_l}$ maps $W$ univalently onto
$f^{q_l} (W)\Supset W$). Note that  its size depends on $N$ and $l$ only. 
It takes a bounded number of iterates ( $\leq t_0 (N,l,\kappa)$ )
for the pullback in question to get trapped in $W$. 
 By adjusting  $W$ and selecting  $\tl V$
 sufficiently small, we ensure that  the first $t_0$ pullbacks
$\tl V_{-i}$, $i < t_0$, are pairwise disjoint and disjoint from $W$, while
$\tl V_{-t_0}\Subset W$. Then the further pullback $\tl V_{-i}\subset W , \ t_0 \leq t$,
will stay pairwise disjoint  and disjoint form the first ones.
So, independently of $t$, the whole pullback $\tl V_{-i}$, $i=0,1,\dots t$, will consists of
pairwise disjoint domains.   Moreover,
\begin{equation}\label{V_-t}
  \diam \tl V_{-t}\to 0\ \mathrm{as} \ t\to \infty\ 
   \mbox{with  $(N,l,\kappa)$ fixed}.
\end{equation}

Let us now pull $\tl V_{-t}$  further to $c_{-t-s}$.
The number $s$ of iterates is bounded by  $\bar s ( N,l,\kappa)$ from
Lemma \ref{point in petal}, so for  $t$ sufficiently big, 
 (\ref{V_-t}) ensures that these pullbacks stay small and pairwise disjoint. 
Since $c_{-t-s}$ lies in the middle of the repelling crescent $\tl
\CC^r$, the final domain $\tl V_{-t}$ is trapped well inside $\CC^r$.
Hence it projects to a disk compactly contained  in the repelling cylinder $\Cyl^r\isom \C/\Z$.
(We will keep the same notation for it.)

Consider a parameter domain $ \La\subset \C/\Z$ such  that 
$I_\la (\tl c_n)\in \tl V_{-t-s}$ for any transit parameter $\la$ in
$\La$ (in fact, under  our normalizations and notational  conventions, $\La=\tl V_{-t-s}$).
Pull $\tl V_{-t-s}$ further back by this transit
  map,  and then further back to $\tl c_0$  by the iterates of $\tl f$.    
Call the corresponding pullbacks $\tl V_{\la, -t-s-I-i}$, $i\leq n$. 
By Lemma \ref{n}, it involves at most $\bar n (N,l, \kappa)$ iterates.
Hence all these  pullbacks have
  a small diameter  and stay pairwise disjoint and disjoint from the
  initial pullbacks, which proves assertion (ii).
 It also follows that  the disc $\tl U_\la  :=  \tl V_{\la, -t-s-I-n}$ is
  trapped well inside  $\tl V$, which implies  that the maps $F_\la$
  defined by (\ref{F_la}) are quadratic-like. 

Since the transit map $I_\la: \CC_\la^a\ra \CC^r$ 
depends holomorphically on $\la\in \La$, these q-l maps form a
quadratic-like family. For the same reason, the domains 
$\tl V_{\la, -t-s-I}$, and hence their further pullbacks $\tl V_{-s-t-I-i}$,
 move holomorphically with $\la$, so our family is equipped
(see \S \ref{ql families}). For $\la\in \di \La$, we have
$I_\la(c_n)\in \di \tl V_{-t-s}$, 
and  hence $F_\la(c_0)\in \di \tl V$. Thus, our q-l family is proper.    
Finally, as $\la$ goes once around $\di \La$ then $I_\la(c_n)$ 
goes once around $\tl V_{-t-s}$
(recall that with our normalizations,   $\La=\tl V_{-t-s} $). 
So, our q-l family is unfolded. This completes the proof of (i).

The upper estimate in item (iii) and item (iv) follow from the
property that the total number of $\tl f$-iterates involved  in our
construction is bounded in terms of $(N,l,\kappa)$ and $t$,
while the transit maps $I_\la$, $\la \in \bar\La$, form a compact
family. Hence the size of $U_\la$ is definite in terms of $(N,l,
\kappa)$ and $t$. 

On the other hand,  as $t\to \infty$ with $(N,l, \kappa)$ fixed,
(\ref{V_-t}) implies that $\diam U_\la \to 0$. This yields
the lower bound in item (iii).
\end{proof}

For notational convenience,
let us shift the $m$-parameter:
$$
   \m= m-\kappa\in \bar\N =\{1,2,\dots, \infty\} ,
$$
and let $\FF_\m= R_\Sieg^\m \FF$,
where $\FF$ is the quadratic family $(\Bf_\gamma)$.
By Theorem \ref{hyp fixed point}, these are holomorphic curves converging to the
unstable manifold $\FF_\infty=\WW^u(f_\infty)$ for the Siegel renormalization. 

\sss {Quadratic-like families for parabolic perturbations}

Perturbing our parabolic maps $\tl f_\m$ within the families $\FF_\m$,
we can construct genuinely renormalizable maps: 

\begin{lem}\label{ql family}
Under the circumstances of Lemma \ref{ql family on Cyl},
for any $\m\in \bar\N_0 $ and $j> \underline {j} ( N, l, \kappa, t)$,
there exists a holomorphic subfamily $\FF_{\m , j} = (f_{\m, j; \la} )$ of $\FF_\m$
parametrized by some  domain $\La_{\m,,j}$  with the following
properties:
 

\ssk\nin {\rm {(i)} } Each family $\FF_{\m, j}$
gives rise to  a primitive proper unfolded q-l family
$$   
    F_{\m, j; \la}   = f_{\m,j;\la}^\per:  U_{\m,j;\la}\ra   V_\m, \quad
    \la\in \La_{\m , j}, \ 
$$     
with period  $\per= n+j+s+t$;

\ssk\nin {\rm {(i)} } 
As $\m\to \infty$, the families $\FF_{\m, j } $ converge, uniformly in $\m$, 
to the families $\FF_{\infty,j}$ in $\FF_\infty= \WW_\Sieg^u(f_\infty)$;

\ssk \nin \rm { (iii) }
$$
  \bar \mu( N,  l, \kappa, t) \geq  \mod (V_\m \sm U_{\m, j; \la})   
  \geq \underline{\mu} ( N,  l, \kappa, t)\to \infty\
   \mathrm{as} \ t  \to \infty  \ \mbox{ with $N,l,\kappa $ fixed } .
$$ 

\ssk \nin \rm { (iv) }
In case of connected Julia set $J(F_\la)$ 
(i.e., when $\la$ belongs to the corresponding little Mandelbrot set 
$\MM'_{N,l, \kappa,t,m,j}$), 
the disks $U_\la^j $ are
$L( N, l, \kappa, t )$-quasidisks with   
$$
   \area U_{\m, j; \la }  \geq c ( N,  l, \kappa, t)>0 .
$$ 

All geometric  constants and bounds are independent of $\m$ and $j$.
\end{lem}

\begin{proof}
Throughout this argument,  $(N,l,\kappa,t)$ will be fixed,
and dependences on them  will not be mentioned. 
Parameters $m$ and $j$ will be free.   


By Corollary~\ref{Siegel renorm of famiies}, 
the families $ \FF_\m $
stay within a compact collection of families crossing the Siegel class
$\{ f\in \IS :  f'(0) =e^{2\pi i \theta} \}$ transversally at points
$
       f_\m =  R_\Sieg^\m \Bf_{\theta} . 
$
In fact, they converge, as $\m\to \infty$, 
to the unstable manifold $\WW^u(f_\infty) \equiv \FF_\infty$ of
the Siegel fixed point.
Moreover, the  parabolic maps
$$
       \tl f_\m = R_\Sieg^\m (\Bf_{p_m/q_m} )
      = f_{m; \,  p_\kappa/q_\kappa} \in \FF_\m,
$$
converge  to  $\tl f_\infty\in \FF_\infty$.
This allows us to apply  the Parabolic Bifurcation Theory 
 in a uniform way  to the families $\FF_\m$ 
near the maps $\tl f_\m $.

Let us start with  the limiting parabolic  map  $\tl f= \tl f_\infty$.
Let  $  V   \ni c_0$ be the  disk  
selected for this map in  Lemma \ref{ql family on Cyl},
and let  $\tl V_{-s-t} \ni c_{-s-t} $ be its pullback constructed in that 
lemma.
 It is compactly contained in the repelling crescent $\tl \CC^r$,
and hence it is compactly contained in some smooth disk $\La\Subset
\tl \CC^r$.

There is a neighborhood $\Upsilon\subset \FF_\infty$ of $\tl f$
such that for any map $f_\gamma\equiv f_{\infty, \gamma} \in \Ups$,
the pullback $V^\gamma_{-s-t}\ni c_{-s-t} $ of $ V$ under $f_\gamma^{s+t}$
is compactly contained in $\La$ as well (uniformly over $f_\gamma\in \Upsilon$).  
Moreover, since  the disks  $\tl V^\gamma_{-s-t}$  are univalent pullbacks of a
fixed disk $ V$ by a holomorphic family of  maps $f^{s+t}_\gamma$, 
they   move holomorphically in $\gamma$; let 
$$ 
        h_\gamma : \tl V_{-s-t} \ra V^\gamma_{-s-t}
$$  
be this holomorphic  motion (based at $\tl f$).  

\ssk 
By Theorem \ref{parabolic bifurcation}, for any sufficiently big $j$,
there exists a holomorphic function  $\gamma= \gamma_j (\la)$ on $\La$ such that 
 the transit  maps $I_\gamma^j : \C/\Z \ra \C/\Z $
induced by $f_{ \gamma }^j$, have the following  properties:

\ssk\nin $\bullet$
  $I_\gamma^{j} (0) = \la$ (recall that the uniformizations of the
  Douady cylinders $\Cyl^{a/r}$
by $\C/\Z$ are selected so that
 $c_n\in \Cyl^a$ and $c_{-s-t}\in
\Cyl^r $ correspond to  $0\in \C/\Z$);

\ssk\nin $\bullet$
 As $j\to \infty$, the transit maps $I_{\gamma(\la)}^j$ converge uniformly on
 compact sets of $\C/\Z $
and uniformly in $\la\in \La $ 
to the transit map $I_\la: z\mapsto z+\la$  between the
Ecall\'e-Voronin cylinders for the parabolic map $\tl f$.

\ssk  
By the Argument Principle%
\footnote{This is an occasion of the standard Phase-Parameter
  relation.}
for any $z\in \di \tl V_{-s-t}$
there exists a unique $\la\in \La$  such that
$$
   h_\gamma  (z) = I^j_\gamma (0) ,  \quad  \mathrm{with}\
   \gamma= \gamma_j (\la),
$$
and these $\la$'s go around a Jordan curve $\Gamma^j \Subset \La$.
This implies that each quadratic-like family  
\begin{equation}\label{perturbed q-l family} 
     F_{j;  \la}  = f^{s+t}_\gamma \circ I_\gamma^j \circ
     f^n_\gamma : U_{j; \gamma} \ra  \tl V
\end{equation}
is proper and unfolded over the disk bounded by $\Gamma^j \Subset \La$,
where $\gamma=\gamma_j (\la)$ and $U_{j; \gamma}  \ni c_0$ is  the pullback of
 $V^\gamma_{ -t-s} $ by $I_\gamma^j \circ f_\gamma^n$.  
We obtain assertions (i) and (ii) for $\m=\infty$.

 Assertions (iii) and (iv) for $\m=\infty$ follow from the corresponding assertions of
 Lemma \ref{ql family on Cyl} since the quadratic-like families (\ref{perturbed q-l family})
are small perturbations (for big $j$) of the family
$\tl F_\la$  (\ref{F_la}).

\ssk For each finite $\m$, we can apply the same argument to the family $\FF_\m$,
which provides us with  quadratic-like families $F_{\m, j ; \la }$ with desired properties,
except that the geometric constants and bounds may depend on $\m$.
To make them uniform, we can apply a perturbative argument near $\FF_\infty$.
Namely, let us start with the same disk $ V \ni c_0$ as for $\tl
f\equiv \tl f_{\infty}$,
and pull it back by $f_{\m; \gamma}^{s+t}$.   
We obtain  a holomorphically moving family of disks 
$V^{\m;\gamma}_{-s-t} \Subset \CC^r (f_{\m; \gamma}) $ 
which is a  small perturbations of the above  family
$(V^\gamma_{-s-t} )$ for the $f_\gamma$. In particular,  for $\m$ big enough,
all these disks are uniformly compactly contained in the domain $\La$ used for $\m=\infty$.

Moreover, by  Theorem \ref{parabolic bifurcation}, as  $m, j \to \infty$, the transit maps
$I_\gamma^{\m, j }$, with $\gamma= \gamma_{\m,j} (\la)$, associated with  $f_{\m;  \gamma}$, converge to $I_\la$.
It follows that for $\m$ and $j$ sufficiently big,
the quadratic-like families  $(F_{\m, j; \la} )$ are small perturbations
of the family $( \tl F_\la )$ from Lemma \ref{ql family on Cyl}.
The uniformity of the  geometric bounds follows. 
\end{proof}

\sss{Renormalizations in the quadratic family}

Lifting the above renormalization to the quadratic family by means of
change of variable $\Pi_{m-\kappa}$ (\ref{restricted Pi}) 
 we obtain:

\begin{cor}\label{ql renormalization}
Let  $\underline {N}, \underline {l}, \underline{\kappa}$,
and $\underline{t}$ be as above.
Then for any natural 
$( N,l, \kappa, t ) > (\underline{N}, \underline{l},  \underline{\kappa},
\underline{t}) $,  there exist $\underline{m}$ and $\underline{j}$ 
with the following properties.
For  each  natural  $(m,j)> (\underline{m}, \underline{j})$,  
consider the holomorphic family ${\boldsymbol{\FF}}_{m,j}  = (\Bf_{m,j; \la} ) $  of quadratic polynomials
such that 
$$
   f_{m-\kappa , j; \la} = R_\Sieg^\m ( \Bf_{m-\kappa ,j; \la} ) , 
$$
where $( f_{m-\kappa , j; \la} )$ is the family from Lemma \ref{ql family},  
Then:

\ssk\nin {\rm {(i)} } Each family ${\boldsymbol{\FF}}_{m, j}$
admits  a primitive proper unfolded q-l renormalization 
$$   
   \BF_{m, j; \la}   = \Bf_{m, j;\la}^\p:  \BU_{m,j;\la}\ra   \BV_m, \quad
    \la\in \La_{m , j}; \ 
$$     

\ssk \nin \rm { (ii) }
$$
  \bar \mu( N,  l, \kappa, t) \geq  \mod ( \BV_\m \sm \BU_{\m, j; \la})   
  \geq \underline{\mu} ( N,  l, \kappa, t)\to \infty\
   \mathrm{as} \ t  \to \infty  \ \mbox{ with $N,l,\kappa $ fixed } .
$$ 

\ssk \nin \rm { (iii) }
In case of connected Julia set $\BJ\equiv J(\BF_\la)$ 
(i.e., when $\la$ belongs to the corresponding little Mandelbrot set 
$\MM'_{N,l, \kappa,t,m,j}$), 
the disks $\BU_{m, j; \la} $ are
$L( N, l, \kappa, t )$-quasidisks with   
$$
   \area \BU_{\m, j; \la }  \geq   c\, \area \BV_m,
   \quad \mathrm{where}\ c= c( N,  l, \kappa, t)  >0  .
$$ 

All geometric  constants and bounds are independent of $\m$ and $j$.

\end{cor}

The little Mandelbrot copies 
$\MM'=\MM'_{ N,  l, \kappa,t,m,j }\subset\MM$ generated by these quadratic-like families 
 determine the desired renormalization combinatorics.  Below, a map $\Bf_\la$ will be called {\it
  renormalizable} if it is DH renormalizable with these combinatorics
(and similarly, for the Siegel map $f_\la$).

\subsection{A priori bounds}

Along with lower thresholds 
 $(\underline{N}, \underline{l}, \underline{\kappa}) $
let us  select some   upper bounds 
$
  (\bar N, \bar l, \bar \kappa)  > 
( \underline{N}, \underline{l},   \underline{\kappa} ) 
$
satisfying the following requirements:
$$
    \bar N> \uN,\quad    \bar l > \ul= \ul (\bar N),\quad 
\bar\kappa >\ukappa =    \ukappa (\bar N, \bar l). 
$$

Let $\Bf_*: \BU  \ra \BV $ be an infinitely
renormalizable quadratic polynomial with
bounded  combinatorics $(\MM^i)_{i=0}^\infty$, where 
$\MM^i=M'_{N_i, l_i, \kappa_i , t_i  , m_i ,  j_i } $ are the little Mandelbrot
copies constructed above with 
\begin{equation}\label{N,l,kappa}
    ( \underline{N}, \underline{l},   \underline{\kappa} ) <   (N_i, l_i, \kappa_i ) \leq 
  (\bar N, \bar l, \bar  \kappa )
\end{equation}
(while the bounds on $t_i$, $m_i$ and $j_i$ are not yet specified%
\footnote{In fact, in this section  one can consider maps $\Bf_*$ with
  unbounded $t_i, m_i, j_i$}).

\begin{prop}\label{a priori bounds}
For any sequence $ (N_i, l_i, \kappa_i ) $ satisfying (\ref{N,l,kappa})
  there exists $\underline{t}$ such that if 
$$
t_i>\underline{t}, \quad i=0,1,\dots,
$$
 then  the quadratic polynomial  $\Bf_* $ has 
 {\it  a priori} bounds $\nu(\bar N, \bar l, \bar \kappa )>0$
independent of   $(t_i, m_i, j_i ) $.
\end{prop}

\begin{proof}
   If $g$ is a quadratic-like map with   $\mod g > \mu $ then it is $K$-qc conjugate to a quadratic  polynomial
   $\Bf_\theta $, where $K=K(\mu)\searrow 1$ as $\mu \to \infty$.  
Hence, if $g$ is DH renormalizable with any combinatorics
$\MM'=\MM_{N,l,\kappa,t,m,j }'$ under consideration, 
then  its renormalization $Rg $ has modulus at least 
$K^{-1}  \underline{\mu} $, 
where $\underline{\mu}=\underline{\mu}(N,l, \kappa,t)$ is from Corollary \ref{ql renormalization},
and $K=K(\underline{\mu} )$.

Let us select  $\nu$ so that $K( \nu )<2$ and then $\underline{t}$ so that
$\underline{\mu} (N,l,\kappa,t)>2\nu $ for any $t> \underline{t}$ 
 (which is possible by Corollary \ref{ql renormalization}).
Then for any  quadratic-like map $g$ with $\mod g > \nu $   which is renormalizable with
combinatorics $\MM'$,  we have $\mod Rg > \nu $ as well.

It follows that $\nu$ gives  {\it a priori bound} for  any quadratic-like map
$g $ with  $\mod g >\nu$
which is infinitely renormalizable with  combinatorics
$(\MM^i)$.
\end{proof}

\subsection{Landing probability}

Let $f_*  = R_\Sieg^{m-\kappa} \Bf_* $, 
and let $R f_*: U_*\ra V_*$ be its DH pre-renormalization
(with the combinatorics constructed in \S \ref{super par}). 

The next  lemma shows that there is a definite probability of landing
in the renormalization domain $U_*$ of the  map $f_*$.


\begin{lem}
Let $\underline{l}$ and $\iota$ be as in Lemma  \ref{definite disc}.
Let $l> \underline{l} + \iota$ and let $D_*= D_*^{l-\iota}$ be the trapping
disk for $f_*$ constructed in that lemma.   
Then $D_*$  contains  domains $U'\subset V'$ of comparable
(with $D_*$)  size 
(with constants depending on $N, l, \kappa$, and $t$)
which are mapped respectively  to
$U_*$ and $V_*$ under some iterate of $f_*$.  
Moreover, $D_*$ is contained well inside $\Dom f_* \sm \OO_*$
(with a lower bound depending on $N$ only),
where $\OO_*=\OO_{f_*}$ is the postcritical set for $f_*$. 
\end{lem}

\begin{proof}
Recall that $f_*$ is a small perturbation of the Siegel map $f$
whose Siegel disk is called  $S=S_f$.
Let $S'$ be the component of  $f^{-1} (S)$
which is different from $S$.   
The trapping disk $\Disc^{l-\iota}$ for $f$  contains in the middle 
 some point  of $\di S' $. If $f_*$ is sufficiently close to  $f$
the  $D_*= D^{l-\iota}_*$  contains in the middle some point of 
$f_*^{-1}(\di  S)$. Hence $f_*(D_*)$ contains in the middle some point
of $\di S$.

The renormalization range $V_* $ can be selected at a much
 deeper 
(but still depending only on $N, l, \kappa$, and $t$)
dynamical scale    than $f_*(\Disc_*)$.  Then $f_*(\Disc_*)$ contains many
(in fact, we need only one) 
univalent and bounded distortion  pullbacks of $V_*$ under the
 Siegel map $f$.  Moreover, these pullbacks have size comparable with 
$\diam f_*(\Disc_*)$.
Selecting  $f_*$  sufficiently close to $f$,
we ensure the same property for  $f_*$.
Then  $\Disc_*$ also contains a comparable  pullback of
$V_*$. The corresponding pullback of $U_*$
 has a comparable size as well (all in terms of $N,l, \kappa$, 
and $t)$ .

The last assertion follows from the property that the postcritical set
$\OO_*$ lies well inside $A^{l-1}_*$ while $D_*$ lies outside $A_*^{l-1}$.
\end{proof}

We call the disk $D =D^{l-\iota}_*$ 
(and similar disks that appear below)
a  {\it safe trapping disk}
since it can be ``safely'' pulled back, with a bounded distortion
(depending on $N$ only),
along any orbit landing in it.
As before, we say that $D$ is {\it centered} at  $\alpha^{l-\iota}$, or that
$\depth D = l-\iota$.

Lifting  this disk by the renormalization change of variable
 $\Pi_{m-\kappa}$ (\ref{restricted Pi}),  
we obtain: 

\begin{cor}\label{safe trapping disk}
The quadratic polynomial  $\Bf_*$ has
a   safe trapping disk  $\BD :=\BD_*^{m-\kappa+ l -\iota }$ 
that contains domains $\BU'\subset  \BV'$ of comparable (with $\BD$)  size
which are mapped  respectively to
$\BU_* $ and $\BV_*$  under some iterate of $\Bf_*$.  
The constant depends on $N, l, \kappa$, and $t$
but is  independent of $m$.
\end{cor}

We will refer to  the above disk  $\BD$ as the {\it base}  safe trapping disk.

Spreading the disks $\BU'\subset \BV'$ around by the landing map, 
we obtain: 

\begin{cor}\label{gaps U(z)}
For any point $z$ whose orbit  passes through the safe trapping disk  $\BD$ 
under the iterates of  $\Bf_*$, 
 there  exist  quasidisks $\BU(z) \subset \BV(z)$
with bounded dilatation  
whose size is comparable with   $\dist(z, \BV(z))$, and such that 
$$ \Bf_*^n (\BU(z) )= \BU_*, \quad 
 \Bf_*^n (\BV(z) )= \BV_*  \quad  \mathrm {for \ some}\  n=n(z).
$$
All constants and bounds depend on $N$, $l$, $\kappa$ and $t$,
but not on $m$.
\end{cor}

\comm{***
\begin{proof}
Since the postcritical set $\BO_\la$ of the maps $\Bf_\la$ does not go
beyond  the collar $\BA^{m-\kappa+ l-1}$, the trapping disk $\BD=\BD^{m-\kappa+\l/2}$ 
is disjoint of $\BO_\la$ (with a definite collar around it).
Hence it can be pulled back along the orbit of $z$.
 The pullback of the domain $BW_\la$ from Corollary 
\ref{safe trapping disk} gives us a desired disk $\BW_\la(z)$.
 
\end{proof}
******}

We are now ready to show the map $\Bf_*$ has 
 a definite landing  probability $\eta$.

\begin{prop}\label{def eta}
For the  polynomial  $\Bf_*$, 
 the landing probability $\eta$ is bounded from below in terms
of $N$, $l$, $\kappa$, and $t$,  uniformly in  $m$.  
\end{prop}

\begin{proof}
 It is known that  almost all point of the Julia set $\BJ_*= J(\Bf_*)$ land in $\BU_*$
\cite{DAN},  so it is sufficient to deal with the Fatou set. 
Since the Siegel disk $\BS= \BS_\Bf$ occupies certain area, it is
sufficient to check that a definite portion of points $z\in \BS \sm \BJ_*$ land
in $\BU_*$.  But any point $z\in \BS\sm \BJ_*$ on its way  from $\BS$ to $\infty$  must pass
through  the base safe trapping disk $\BD$.  Then Lemma \ref{gaps U(z)}
provides us with a domain $\BU (z)$ of points landing in $\BU_*$ that 
occupies  a definite portion  of some neighborhood of  $z$. 
The conclusion follows.  
\end{proof}


\subsection{Escaping  probability $\xi$}\label{escaping par sec}

\sss{Porosity}
Let us  start with a general measure-theoretic
lemma asserting that if a set $X$ has density less than $1-\eps$ in many
scales then it has small area.

By a {\it gap} in $X$ of radius $r$ we mean a round disk of radius
$r$ disjoint from $X$. 

\begin{lem}\label{gaps in many scales}
  For any $\rho\in (0,1)$, $C>0$ and $\eps>0$ there exist  $\si\in
  (0,1)$ and $C_1>0$  
with the following property.  
Assume that a measurable set $X\subset \D_r$
 has the property that for any $z\in X$ there are $n$ disks 
$\D(z,r_k)$ with radii  
$$
      C^{-1} \rho^{l_k} \leq r_k/r \leq C\rho^{l_k}, 
   \quad  l_k\in \N, \  l_1< l_2< \dots < l_n,
$$ 
containing gaps in $X$ of radii  $\eps\, r_k$. Then 
$
    \area X \leq C_1 \si^n\, r^2. 
$
\end{lem}

\begin{proof}
Since the assertion is scaling invariant, 
we can assume without loss of generality that $r=1$.
 We can also assume that $X$ is compact,
   and we can work with  squares instead of disks.
Using the first scale $l_1$ for  points of $X$, 
we can subdivide  the unit square $\Q$ into dyadic squares $Q^1_i$
(of varying scales) such that each $Q^1_i$ contains a comparable
dyadic square $B^1_i$ (of relative scale depending on $\eps$) disjoint from
$X$.   Let $\Q^1\supset X$ be the union of $Q^1_i\sm B^1_i$. Then 
$$
  \area \Q_1  \leq \si_0 \area \Q,
$$
where $\si_0 \in (0,1)$ is roughly
equal to  $1-\eps^2$.

Then we can subdivide each $Q^1_i$ into squares of size $B^1_i$ and repeat
the construction with all non-empty squares of this subdivision
(using a deeper scale $l_j$ with a sufficiently big but bounded $j$). 
It will produce a set $\Q_2\supset X$ such that
$$
  \area \Q_2 \leq \si_0 \, \area \Q_1.
$$

We can repeat this procedure roughly $ n/j$ times, which implies the desired.
\end{proof}

\sss{Landing branches} 



Let us consider   a  safe trapping disk $\BD=\BD^\Bl $  for $\Bf_*$.
centered at the periodic point $\Bal_\Bl $. 
By definition, it has a bounded  hyperbolic%
\footnote{Below, ``hyperbolic''  will always refer to the hyperbolic metric in
$\C\sm \BO_*$.}
 diameter of order 1 in  $\C\sm \BO_*$:
\begin{equation}\label{d}
   d^{-1} \leq      \diam_\hyp \BD \leq d \quad \mathrm{with}\ d=d(N).
\end{equation}
For instance, $\BD$ can be  the base trapping disk of depth
$\Bl=m-\kappa+l-\iota$  from  Corollary~\ref{safe trapping disk},
but we will also consider much more shallow disks.

For any point $z$, let 
$$
  0\leq r_1(z) < \dots < r_n(z)< \dots 
$$
be all  {\it landing times} of $\orb z$ at $\BD$, i.e. the moments for which 
$\Bf_*^{r_n} (z) \in \BD$ listed consecutively  (this list can be
infinite, finite, or empty).  
Let  $T^n :  \Dom T^n \ra \BD$ be the corresponding {\it landing maps},
i.e.,   for a point $z\in \Dom T^n $, the landing moment $r_n(z)$ is
well defined and  $T^n (z)= \Bf_*^{r_n} (z) $.
Let $P^n (z)\ni z$ be the pullback of $\BD$ along the orbit
$\{\Bf_*^i  (z) \}_{i=0}^{r_n}$. 
Since $ \BD\Subset \C\sm \BO_*$, the  maps 
\begin{equation}\label{return branches}
      \Bf_*^{r_n } : P^n  (z) \ra \BD
\end{equation}
are univalent 
 We will refer to these maps as the {\it landing branches}.

For a domain  $P =P^n(z) $, 
we will also use notation $r_P$ for the landing time $r_n(z)$
(which is independent of $z\in P$), and will will use  notation
$T_P= \Bf_*^{r_P}$ for the corresponding landing branch $P\ra \BD$.

Let $\PP(\BD) $ be the family of all  domains $P=P^n (z)$.

\comm{***
 Let $T :   \Dom T \ra \BD$ be the first return map to $\BD$,
and let $T^n :  \Dom T^n \ra \BD$ be its iterates. 
For a point $z\in \Dom T^n $,
we have $T^n (z)= \Bf_*^{r_n} (z) $ for some return time $r_n =r_n(z)$.
Let $P^n (z)\ni z$ be the pullback of $\BD$ along the orbit
$\{\Bf_*^i  (z) \}_{i=0}^{r_n}$. 
Since $ \BD\Subset \C\sm \BO_*$, the  maps 
\begin{equation}\label{return branches}
      \Bf_*^{r_n } : P^n  (z) \ra \BD
\end{equation}
are univalent 
 We will refer to these maps as {\it return branches}
(slightly inconsistently as the domains $P^n(z)$ are  not
necessarily contained in $\BD$).

Let $\PP $ be the family of all the domains $P^n (z)$. 
For $P=P^n(z) \in \PP$, 
we will also use notation $r_P$ for the return time $r_n(z)$
(which is independent of $z\in P$), and will will use  notation
$T_P= \Bf_*^{r_P}$ for the corresponding return branch $P\ra \BD$.
************}

\begin{lem}\label{distortion trivia}
\ssk \nin $\bullet$
   The landing branches $T_P : P\ra  \BD$, $P\in \PP (\BD) $,  
have uniformly bounded distortion; the domains $P\in \PP(\BD)$ have  a bounded shape
and are well inside $\C\sm \BO_*$
(with  bounds and constants  depending only on
$\bar N$); 

\ssk\nin $\bullet$
Each domain $P\in \PP(\BD)$ contains a pullback of $\BV_*$ of
 comparable size (with the constant depending  only on the parameters
$\bar N, \underline{l}, \underline{\kappa}, \underline {t}$).
%
\end{lem}

\begin{proof}
  The first assertion follows from the property that $\BD$ is well inside
  $\C\sm \BO_*$ and  the Koebe Distortion Theorem.
 Together with 
Corollary \ref{safe trapping disk},
it implies the second assertion.
\end{proof}

Along with $\BD$, let us consider another trapping disk $\BD'$
(which is allowed to coincide with $\BD$).   
Let $\PP_{\BD'} (\BD) $ be the family of all the domains $P=P^n (z)\in \PP(\BD)$ 
intersecting $\BD'$.

\begin{lem}\label{distortion trivia-2}
 For any domain $P\in \PP_{\BD'} (\BD)$, 
$$
  \diam P \leq C_0\diam \BD' \quad \mathrm{ with}\  C_0= C_0 (\bar N), 
$$
where $\diam\equiv \diam_\Euc$ stands for the Euclidean diameter;
\end{lem}

\begin{proof}
By Lemma \ref{expansion},
the inverse branch $T_P^{-1}: \BD\ra P$
is a hyperbolic contraction.
Hence  $\diam_\hyp P \leq \diam_\hyp \BD \leq d $.
 Since $P\cap \BD' \not=\emptyset$ and $\diam_\hyp \BD \leq d$ as well, we have: 
\begin{equation}\label{diam}
    \diam_\hyp (\BD\cup P)\leq 2d.
\end{equation}
 It follows that the conformal factor $\rho(z)$ between the
 hyperbolic and Euclidean  metrics has a  bounded oscillation on
 $\BD' \cup P$:
$$
  \sup_{z\in \BD' \cup P}  \rho(z) \leq 
  C \inf_{z\in \BD' \cup P}   \rho (z), \quad C=C( N). 
$$
Hence
\begin{equation}\label{hyp-Euc}
    \frac  {\diam_\Euc P } { \diam_\Euc \BD' } 
   \leq C \,  \frac{\diam_\hyp  P} {\diam_\hyp \BD' } \leq Cd^2.
\end{equation}

\end{proof}

\bignote{Expanding Lemma is hidden here}

\comm{*****
\begin{lem}\label{distortion trivia-3}
For any $\la>1$, there exist $ \nu = \nu (N, \la) \in \N$  
 with the following property.
Let us consider $\nu+1$ trapping disks%
\footnote{not necessarily distinct}
 $\BD_i$, $i=0,1,\dots, k$, 
such that 
\begin{equation}\label{depth}
\depth \BD_0 \geq \depth \BD_\nu .
\end{equation}
Let $P\in \PP(\BD_\nu )$ with  $\Bf_*^{r_\nu } (P) = \BD$. 
If   $\Bf_*^{r_i} (P)\cap \BD_i \not=\emptyset  $ for some moments
$0=r_0< r_1<\dots < r_\nu$,
then 
$$\mbox{
$|D\Bf_*^{r_\nu } (z)| \geq \la$ \rm{for all}  $z\in P$. 
}$$
\end{lem}

\begin{proof}
Let $P_i :=  \Bf_*^{r_i} (P) $, so $P_i\cap \BD_i\not=\emptyset$,
$i=0,1,\dots, \nu$. 
By (\ref{diam} ), 
\begin{equation}\label{diam-2}
   \diam_\hyp \BD_i  \cup P_i \leq 2d ,
\end{equation} 
which implies (\ref{z-position}) for all $z\in P_i$.    
This allows us to apply   Lemma \ref{perturbative expansion} and to
conclude that all the maps $\Bf_*^{r_{i+1}-r_i}: P_i\ra P_{i+1}$ are
hyperbolic expansions  by  some factor $\mu =\mu (N) >1$. 
Hence the map 
\begin{equation}\label{expanding map}
\Bf_*^{r_\nu } : P\ra P_\nu
\end{equation}
 is a hyperbolic expansion by  $\mu^\nu $. 

By (\ref{diam-2}) 
the oscillation of the conformal factor $\rho(z)$ on each
$\BD_i \cup P_i$ is bounded by some $C=C(N)$. 
Hence  map (\ref{expanding map}) is a Euclidean expansion by factor
$$
        \geq     C^{-2} \mu^\nu \rho ( \alpha_{l_0} )  / \rho ( \alpha_{\Bal_\nu} ) , 
$$ 
where  the $\BD_i$ are centered at  the periodic points $\Bal_{\Bl_i}
$ respectively.
Finally,  (\ref{depth}) implies that   
$$
 \rho ( \Bal_{\Bl_0} ) \geq c\,  \rho ( \Bal_{\Bl_\nu } ) \quad
 \mathrm{with \ some}\  c=c (N)>0.
$$   \note{need a comment?}
The conclusion follows.
\end{proof}

*************}

The following lemma shows that pullbacks of trapping disks to some
point $z$ lie in different scales: 

\begin{lem}\label{exp decay}
For any $\si \in (0, 1)$,  there exists $\nu = \nu(N, \si)\in \N$ 
 with the following property. 
Let $\BD_i$, $i= 1,\dots, \nu$, be safe trapping disks,
not necessarily distinct. Consider a point $z$  landing at the $\BD_i$ at 
moments $r_i$, where $0   \leq  r_1  < \dots < r_\nu $, and
let $P^i\ni z$ be the corresponding pullback of the $\BD_i$. 
Then 
$$
   \diam P^{\nu}  < \si \diam P^1 .  
$$

\end{lem}

\begin{proof}
Let $P\equiv P^\nu $, and 
let $P_i :=  \Bf_*^{r_i} (P) $, $i=1,\dots, \nu$.
Then  $P_i\cap \BD_i\not=\emptyset$.
By property (\ref{diam}), 
\begin{equation}\label{diam-2}
   \diam_\hyp \BD_i  \cup P_i \leq 2d ,
\end{equation} 
which implies (\ref{z-position}) for all $z\in P_i$.    
It allows us to apply   Lemma \ref{perturbative expansion} 
and to conclude that all the maps $\Bf_*^{r_{i+1}-r_i}: P_i\ra P_{i+1}$ are
hyperbolic expansions  by  some factor $\la  =\la  (N) >1$. 
Hence the map 
$ \Bf_*^{r_\nu -r_1} : P_1 \ra P_\nu$
 (which is the same as $ \Bf_*^{r_1} (P)  \ra  \BD_\nu $)
 is a hyperbolic expansion by  $\la^{\nu-1} $. 
Hence 
$$
  \diam_\hyp (\Bf_*(P))    \leq \la^{-\nu+1} \diam_\hyp \BD_\nu \leq d\, \la^{-\nu+1}. 
$$
On the other hand, 
$
   \diam_\hyp  (\Bf_* (P^1))  \equiv \diam_\hyp \BD_1 \geq d^{-1},
$
so $$   \diam_\hyp (\Bf_*(P))   \leq d^2 \la^{-\nu+1} \,  \diam_\hyp  (\Bf_* (P^1)).$$  
Property  (\ref{diam-2})  with $i=1$ allows us to switch in the last estimate from the
hyperbolic diameters  to the Euclidean ones (like in (\ref{hyp-Euc}))
and then to apply the Koebe
Distortion Theorem to the map 
$ \Bf_*^{r_1} $ on $ P\cup P^1 $. The conclusion follows. 
\end{proof}

\sss{Truncated Poincar\'e series}
Let us now fix a safe trapping disk $\BD$
(in applications, it will be the base trapping disk), and let
$\PP :=\PP_\BD(\BD)$
Of course, a domain $P \in \PP$ can admit several representations
as $P^n (z)$. 
Let 
$$
  \chi(P) = \max  \{n : \exists\, z\in P \ \mathrm{such \ that} \ 
P= P^n  (z) \}. 
$$
Let $\PP^n$ be the family of domains $P\in \PP$ with $\chi(P)\leq  n$. 
We also let
$$
   \P = \bigcup_{\PP} P ,\quad    \P^n = \bigcup_{\PP^n } P 
$$

\begin{lem}\label{bounded multiplicity}
  There exists $C= C( N)$ such that 
$$
   \sum_{\PP^n}  \area P  \leq C n \area \BD . 
$$
\end{lem}

\begin{proof}
Note that  the family  $ \PP^n  $ has 
the intersection multiplicity at most $n$. Indeed, if some point $z$
is contained in  $k$ sets $P_i$  of this family then 
$P_i = P^{n_i} (z)$ with $n_i= n_i(z) \leq n$. But since the $n_i$
are pairwise distinct,  $\max n_i \geq k$.  
  
Hence 
\begin{equation}\label{sum vs union}
  \sum_{\PP^n }  \area P \leq  n\,   \area \P^n \leq n \area  \P.  
\end{equation}

By Lemma \ref{distortion trivia-2} (i), 
 $\P$  is contained in  a Euclidean  neighborhood of $\BD$ of size
$\leq C_0 \diam \BD$.  Since $\BD$ has a bounded shape, 
$
  \area \P  \leq C\area \BD ,
$ 
with $C=C(N)$. 
Together with (\ref{sum vs union}), this implies the desired.  
\end{proof}

Let us consider the following {\it truncated Poincar\'e series}:
for $\zeta\in \BD$, let 
$$
    \phi_n (\zeta) =  \sum_{P\in \PP^n}  
     \frac 1{| D T_P (\zeta_P )|^2 },  \quad
\mbox{where $\zeta_P\in P$ and $T_P  (\zeta_P) = \zeta$.}
$$

\begin{lem}\label{bound of phi}
We have
$   \phi_n(\zeta) \leq C n $, where $C=C(\bar N)$. 
\end{lem}

\begin{proof}
  We have:
$$
       \int_\BD \phi_n (\zeta)\,  d\area(\zeta) =  \sum_{\PP^n} \area P
       \leq C n \area \BD,
$$
where the last estimate is the content of Lemma \ref{bounded multiplicity}. 
But 
since the branches  $ T_P : P\ra \BD$ have a bounded distortion,
$\phi_n(\zeta)\asymp \phi_n (\zeta')$ for any $\zeta, \zeta' \in \BD$
(with constants depending only on $ N$). 
The conclusion follows. 
\end{proof}

\sss{Probability of few returns to the base}

Let us start with an observation that for $m$ big enough,
our quadratic polynomial  $\Bf_*$ has  plenty of safe trapping disks:

\begin{lem}\label{many trapping disks}
  For any natural $\tau\in \N$,
there exists $\underline{m}=\underline{m}(N,l.\kappa, t, \tau)$ such that
for any $m>\underline{m}$, the polynomial $\Bf_*$ 
has at least $\tau$ safe trapping disks $\BD_i $ satisfying properties of Lemma \ref{definite disc}.
Moreover, these trapping disks are  pairwise disjoint
and disjoint from the base safe trapping disk $\BD=\BD^{m-\kappa+l-\iota}$.
\end{lem}

\begin{proof}
  By Lemma \ref{shadowing}, our polynomial $\Bf_*$  is $\eps_m$-close to the Siegel polynomial
 $\Bf$, where $\eps_m\to 0$ as $m\to \infty$   (keeping the other parameters, $N, l, \kappa$ and $t$,
frozen).  Hence for $m$ big enough,  Lemma \ref{definite disc} (applied directly to $\Bf_*$)
supplies us with arbitrary many safe  trapping disks $\BD_i $.   
\end{proof}

\bignote{Lemma on ``going high'' is hidden here} 
\comm{****
\begin{lem}\label{going high}
  Let $\Bf_* $ be a quadratic polynomial constructed above,
and let   $\BD\equiv \BD^{m-\kappa}$ be a trapping disk for $\Bf_*$. 
   For any natural  $\tau < m-\kappa-l$,
   let  $ Z(\tau) $ bet he set  of points $z\in \BD$  that 
   enter the trapping disk $\BD^{m-\kappa - \tau}$ before (possible)
   returning back   to $\BD$. Then 
$$
  \area Z(\tau) \leq C  \si^\tau \area \BD ,
$$
   with $\si\in (0,1)$ and $C>0$ depending only on $\bar N$.
\end{lem}

\begin{proof} 
 If a point $z\in \BD$ enters $\BD^{m-\kappa- \tau}$ before
 returning back to $\BD $, 
then  on its way to $\BD$ it cannot come back to the Siegel disk $\BS=S_\Bf$.
On the other hand,  
it must visit all intermediate trapping disks $\BD_{m-\kappa- j}$, $j \leq \tau$.
By Lemma \ref{}, \note{comment on the lift?}
these trapping disks were selected so that 
their definite parts are contained in $\Bf_*^{-1}(\BS)$.
Our orbit cannot go there, so this produces  $\tau$  gaps in $Z(\tau)$ near $z$. 
These gaps lie in exponentially decaying scales since 
by Lemma \ref{exp decay}  the landings to
the disks $\BD^{m-\kappa- j}$ are exponentially expanding in $j$.
\end{proof} 
****}

{\it From now on,  $\BD$ will stand for the base trapping disk. }  
Recall that $\BJ_*$ is the Julia set of $\Bf_*$.
 Let  $Z $ be the set of points $z\in \BD \sm \BJ_* $
   that under the iterates of $\Bf_*$
 never return back to 
   $\BD$.
The following lemma shows that for $m$ sufficiently big, 
 it is difficult to escape from 
 $\BD$:

\begin{lem}\label{no return}
 For any natural $\tau\in \N$,
there exists $\underline{m}=\underline{m}(N,l.\kappa, t, \tau)$ such that
for any $m>\underline{m}$,
$$
    \area Z\leq C  \si^\tau \area \BD,
$$
   with $\si\in (0,1)$ and $C>0$ depending only on $N$.
\end{lem}

\begin{proof}
Let $z\in Z$.
If $m$ is sufficiently big then 
 on its way from $\BD$ to $\infty$,
the orbit of $z$ must visit $\tau$ safe trapping disks $\BD_i $
from Lemma \ref{many trapping disks} at some moments
$r_1< r_2< \dots < r_\tau$. 
By Lemma \ref{definite disc}, 
definite parts $W_i $ of these trapping disks are contained in $\Bf_*^{-1}(\BS)$.
 Since $\orb z$  never 
 returns back to $\BD $, 
 it cannot visit the Siegel disk $\BS=S_\Bf$,
and hence it cannot land in  the domains $W_i $ either.

Since each disk $\BD_j$ is safe, it can be univalently and with bounded distortion 
pulled back to $z$.  The corresponding pullback of $W_i $ creates a gap
of definite size  in $Z$ near $z$. 
By Lemma \ref{exp decay},  these gaps lie in  $\asymp \tau$ different scales.
Lemma~\ref{gaps in many scales} completes the proof. 
\end{proof}

\comm{****

\begin{lem}\label{n returns}
  Let $ Z_n$ be the set of points $z\in \BD^{m-\kappa} \sm J( \Bf_* )$
  that  return back to the trapping disk 
   $\BD = \BD^{m-\kappa} $  at most  $n$ times. Then  
$$
   \area Z_n  \leq Cn \si^{m-\kappa} \area \BD,
$$
with $\si\in (0,1)$ and $C>0$ depending only on the bound $ N$
  of the rotation number $\theta$. 
\end{lem}

\begin{proof}
  Let $T :   Y\ra \BD$ be the first return map to $\BD$,
and let $T^k :  Y_k \ra \BD$ be its iterates. 
For a point $z\in Y_k $,
we have $T^k(z)= \Bf_*^{r_k} (z) $ for some return time $r_k=r_n(z)$.
Let $U^k(z)\ni z$ be the pullback of $\BD$ along the orbit
$\{\Bf_*^i  (z) \}_{i=0}^{r_n}$. 
Since $\bar \BD$ is well inside  $\C\sm \BO_*$, the  maps 
$$
      \Bf_*^{r_k} : U^k (z) \ra \BD
$$
have uniformly bounded distortion (with an absolute bound).

Let $\UU$ be the family of all domains $U^k(z)$. 
Any point $z\in Z_n$ is contained in no more than $n$ domains
 $U\in \UU$.  Hence
$$ 
    \sum_{U\in \UU} \area( Z_n \cap U) \leq n \area D.   
$$
Let us consider the following {\it Poincar\'e sum}:
for $\zeta\in D$, let 
$$
    \PP (\zeta) = \sum_{k+1}^n \sum_{z\in Z_n:\,  T^k z= \zeta}  
     \frac 1{|D\Bf_*^{r_k} (z)|}.  
$$
Then
$$
       \int_D \PP(\zeta) d\area(\zeta) =  \sum_{U\in \UU} \area( Z_n \cap U). 
$$
\end{proof}
*****}

Let 
$$
    Z_n = \bigcup_{P\in \PP^n} T_P^{-1} (Z),
$$
where $Z$ is from Lemma \ref{no return}.
Notice that points of $Z_n$ escape $\BD$ forever
after at most $n$ returns.  

\begin{lem}\label{Z_n}
For any natural $\tau\in \N$,
there exists $\underline{m}=\underline{m}(N,l.\kappa, t, \tau)$ such that
for any $m>\underline{m}$, 
$$
   \area Z_n \leq C\,n\, \si^\tau \area \BD,
$$
where $\si\in (0,1)$ and $C>0$ depend only on $ N$. 
\end{lem}

\begin{proof}
 Since
$$
    \area Z_n = \int_Z \phi_n (\zeta) \, d\area(\zeta),
$$
the conclusion follows from Lemma \ref{bound of phi} and Corollary \ref{no return}. 
\end{proof}

\sss{Many returns to the base} 

\bignote{Old scaling lemma is hidden here} 

\comm{*****
\begin{lem}\label{exp decay-old}
There exists $\si \in (0, 1)$  and $\iota \in \N$ (depending only on 
$ N$) such that 
for any $z\in\BD$ that returns to $\BD$ at least $n$ times,  we have:
$$
   \diam P^{k+\iota} (z) < \si \diam P^k(z), \quad k=1,\dots, n-\iota. 
$$

\end{lem}

\begin{proof}
For any $k \leq  n$, we have  $\Bf_*^{r_k } (P^k (z)) = \BD$. Hence 
$$
  \Bf_*^{r_k } (P ^j  (z)) \cap  \BD\not= \emptyset, 
\quad  j = k+1,\dots,  n .
$$
By  Lemma \ref{expansion},  the map
$   \Bf_*^{r_{k+1}-r_k } $ is expanding on the $ \Bf_*^{r_k }  (P^j  (z))  $ 
in the hyperbolic metric of $\C\sm \BO_*$
with some factor $\rho=\rho( N) >1$. 
Hence for $k \leq  n-\iota$, 
 the maps $   \Bf_*^{r_{k+\iota }-r_k } $ are hyperbolically expanding 
on  $ \Bf_*^{r_k }  (P^{k+\iota}  (z))  $  with  factor $\rho^\iota$. 
Since $\BD\cup  \Bf_*^{r_k }  (P^{k+\iota}  (z)) $ has a bounded hyperbolic
diameter (depending only on $ N$),  the map
$$
   \Bf_*^{r_{k+\iota} - r_k }  : \Bf_*^{r_k} ( P^{k+\iota}  (z)) \ra\BD 
$$
  is expanding with respect to the  Euclidean metric  by factor
$\geq C^{-1} \rho^\iota$, where $C=C( N)$.
Hence
$$
         \diam \Bf_*^{r_k}  (P^{k+\iota}  (z)  )
    \leq        C\rho^{-\iota}    \diam \BD. 
$$
Since  the map 
$$
     \Bf_*^{r_k} : P^k(z) \cup P^{k+\iota} (z) \ra \BD \cup \Bf_*^{r_k}  (P^{k+\iota}  (z) )
$$
has a bounded distortion,
$$
         \diam  P^{k+\iota}  (z)  \leq  C' \rho^{-\iota} \diam \BD \quad
         {\mathrm with} \ C' = C' ( N). 
$$
 The conclusion follows as long as $\iota$ is sufficiently big
 (depending on $ N$ only). 
\end{proof}
***********}

Let 
$$
   {\Bbb S}^n = \bigcup_{ \chi(P)  > n }  P = \bigcup_{\PP\sm \PP^n} P. 
$$

\begin{lem}\label{never landing in V'}  
There exist $C>0$ and $\si\in (0,1)$
depending  on $ N, \underline{l}, \underline{\kappa}, $ and $\underline{t}$
such that  for any $n\in \N$ 
 the area of the set of points of ${\Bbb S}^n $ that never land in $\BV$ is at
 most $C\si^n \area \BD $.
\end{lem}

\begin{proof}
Take a point $\zeta \in \Bbb S^n$.   It belongs to some domain
$P\in \PP$ with $\chi(P) >  n$. Then $P$ contains a point $z$ that
lands in  $\BD$ at least $n$ times, and $P^n(z)=P$.  
By Lemma \ref{exp decay}, the nest
$$
    P^1(z)\supset P^2(z) \supset\dots \supset  P^n(z) = P
$$
represents $\asymp n$ different scales.
By Lemma \ref{distortion trivia}, each of these domains contains a pullback of
$\BV$ of comparable size. 
Now  the desired follows from  Lemma \ref{gaps in many scales}. 
\end{proof}



\comm{***
\sss{Bernoulli return map}

A map 
\begin{equation}\label{Bernoulli map} 
     g:    \Dom g\equiv \bigsqcup U_i \ra D
\end{equation}
is called  {\it Bernoulli} (with range $D$) if  the closures $\bar U_i$ are
disjoint
and the branches $g: U_i \ra D$ are univalent uniformly  expanding
(with some factor $\geq \rho>1$ maps with uniformly bounded distortion
(by some $K$). The iterates 
$$
   g^n: \Dom g^n \equiv \bigsqcup U^n_j\ra D
$$ 
of  a Bernoulli map are
Bernoulli with the same range, expanding factor $\geq \rho^n$,
and distortion uniformly   bounded in  terms of $K$ and $\rho$. A Bernoulli map is 
a {\it return map } for $f$ if all the branches of $g$ are some iterates of
$f$. 

\begin{lem}\label{Bernoulli lemma}
  For any $\de>0$, there exists $\underline{m}$ such that for any $m>
  \underline{m}$, the quadratic map $\Bf_*$ admits a return map  $g$ 
  (\ref{Bernoulli map}) whose range is the trapping disk
   $ \BD = \BD^{m-\kappa}$ and such that 
$$
   \area (\BD\sm \Dom g ) < \de. 
$$ 
\end{lem}

\begin{proof}
    For a small $\eps>0$,  let $\NN_\eps$ be the tubular
    $\eps$-neighborhood of the boundary $\di D$, 
and let $\BD_\eps= D\cup \NN_\eps$.
For $\eps$ sufficiently small, it is a smooth quasidisk. 
\note{prepare  a bit?}

Let $T : \BD\sm Z \ra Z$ be the first return map to $\BD$,
and let $T^n : D\sm Z^n \ra D$ be its iterates. 
For a point $z\in \BD\sm Z^n$,
we have $T^n(z)= \Bf_*^{r_n} (z) $ for some return time $r_n=r_n(z)$.
Let $U^n(z)\ni z$ be the pullback of $\BD$ along the orbit
$\{\Bf_*^k  (z) \}_{k=0}^{r_n}$, and let $U^n_\eps(z)\Supset U^n(z)$ 
be the corresponding pullback of $\BD_\eps$.

 Since $\bar \BD$  does  not intersect
the postcritical set $\BO_*$, the  maps 
$$
      \Bf_*^{r_n} : U^n_\eps (z) \ra \BD_\eps
$$
have uniformly bounded distortion. Moreover, 
by Theorem \ref{}, they are 
$>\rho^n$-expanding with some $\rho>1$. 
Hence there exists $n$ such that  
$$
    \diam U^n_\eps (z) < \eps\quad \forall \ z\in \BD\sm Z^n.
$$
 It follows that if $z\in \BD\sm \NN_\eps $ then 
$U^n_\eps (z) \subset \BD\sm \NN_\eps$.  

Let us say that a domain $U^n(z)$ is {\it good} if
for any $k\in [0, r_n]$, either
$ \Bf_*^k (U^n_i) \cap \BD= \emptyset$ or 
$\Bf_*^k(z (U^n(z)) \subset \BD$.  

Any two good domains  $U^n(z)$ and $U^n(\zeta)$  are
either equal or disjoint.
Indeed, assume $U^n(z)\cap U^n(\zeta)\not= \emptyset$
for two different domains,
and assume  for definiteness that $r_n(z)\leq r_n(\zeta)$. 
Then  the domain $f^{r_n(z) } (U^n(\zeta))$ intersects 
$\BD= f^{r_n(z) } (U^n(z))$. But since the map
$     f^{r_n(z) } |\,  U^n(z) \cup U^n(\zeta) $ 
is  univalent. \note{justify}
$f^{r_n(z) } (U^n(\zeta))$ is not contained in 
$\BD$, contradicting the goodness of $U^n(\zeta)$. 

Let $\UU^n$ be the family of good domains $U^n(z)$.

Hence we obtain a Bernoulli map
\end{proof}
**********}

\sss{Escaping probability}

We are finally ready to show that the escaping probability $\xi$ 
for $\Bf_*$  can
be made arbitrary small by selecting $m$ sufficiently big 
(while  keeping the previously selected  parameters, 
$ N, l, \kappa$, and $t $, unchanged).


\begin{prop}\label{small xi}  
For any $\eps>0$ there exists $\underline {m}$ such that
$\xi< \eps$  for any $m> \underline{m}$. 
\end{prop}

\begin{proof}
Let $Y$ be the set of points in $\BD$
that never land in $\BV_* $.  
We will show first that for $m$ sufficiently big, 
\begin{equation}\label{Y}
   \area Y < \eps \area \BD.
\end{equation} 
For any  $n\in \N$, 
let us cover  $Y$ by  three sets:
$$\mbox{
$Y_0 = Y\cap J(\Bf_*)$, \quad 
$Y_1^n = Y\cap  {\Bbb S}^n$, \quad
   $ Y^n_2  = Y\sm (Y_0\cup  Y_1^n) .$}
$$

It is known that  almost all point of $J(\Bf_*)$ land in $\BV$
\cite{DAN},  so $\area Y_0 = 0$,

By Lemma \ref{never landing in V'}, 
$$
   \area Y_1^n  \leq C \si^n \area \BD <  (\eps/2)\area \BD.
$$
as long as $n$ is sufficiently big. 

Let  us take now any point $z\in Y_2^n$.  Then 
$$
    \chi(z)  : = \max \{ \chi(P): \ P\in \PP, \ P\ni z \}  \leq n,
$$
and $\orb z$ returns back to $\BD$ at most  $n$ times. 
Let $k\leq n$ be the number of returns, and let  $P:= P^k(z)$.
Since $P\ni z$, we have   $P\in \PP^{\chi(z)} \subset \PP^n $.
Moreover, under the return map  
$T_P : P \ra \BD$, the point $z$ must land in $Z$ since it will never
come back to $\BD$ again. Hence $z\in Z_n$. Thus $Y_2^n \subset Z_n$. 
Applying  Lemma \ref{Z_n},  we see that  
$\area Y^n_2 < (\eps/2)\area \BD $ for  $m$ sufficiently big, 
and estimate   (\ref{Y})  follows. 

\msk
To pass from (\ref{Y}) to an estimate of $\xi$, we need to transfer
the density estimate for $Y$  to the fundamental annulus 
$\BV_* \sm \BU_* $. 
Let $\YY$ be the set of points in $\BV_*\sm \BU_*$ that never land in $\BV_*$. 
Again, since almost all points of $J(\Bf_*)$ land in $\BV_*$, it is
sufficient to deal with the Fatou set $\YY \sm \BJ_*$. 
Any  point $z\in \YY \sm \BJ_*$  eventually lands in the ``middle'' of the
base trapping disk  $\BD$. Pulling $\BD$ back to $z$,
we obtain a domain $Q(z)$ of bounded
shape in which the set $\YY\cap Q(z) $ (the pullback of $Y$) 
has density  $\leq C\eps$.
 Applying the
Besikovich Covering Lemma (see \cite{Ma}), we conclude that $\YY$
has density $\leq C' \eps$ in $\BV_* \sm \BU_* $. 
\end{proof}

%

\comm{**** First attempt
Fix any $n\in \N$.  For the Siegel polynomial $P$,
we can construct (By Lemma \ref{disk D for Siegel})
 $n$ disjoint trapping disks $D_i$ on some bounded levels
$l_1< l_2< \dots < l_n$.  Since $\Bf_*$ is an arbitrary small
perturbation of $P$, these disks will serve as safe trapping 
disks for $\Bf_*$ as well  (compare Lemma \ref{definite disc}) . 

Let us define a partial transit map $T : D_{i+1} \ra D_i$ as follows:
for $z\in D_{i+1}$, $\zeta \in D_i$,  we let $\zeta= T (z)$ if 
$\zeta= \Bf_*^k z$ and $\Bf_*^j z\not \in D_i\cup D_{i+1}$ for $0< J< k$.
By Lemma \ref{perturbative expansion}, the transit maps $T : D_{i+1} \ra D_i$ under the 
 are uniformly expanding in the hyperbolic metric of
$\C\sm \bar \OO_*$ (where $\OO_*$ is the postcritical set of $\Bf_*$).
Making the levels $l_1$ sufficiently sparse, we can make the expanding
factor big.  

By Lemma \ref{landing prob}, each of $D_i$  contains a disk $\Ups_i$ of comparable size
that eventually lands in $V'$.   

Let $X$ be the set of points in
$(V'\sm U')\sm J(\Bf_*)$ that never return back to $V'$.   
The orbit of any point $z\in X $  must pass through all the trapping
disks $D_i$. 
Pulling the $\Ups_i$ back to $z$ along this orbit creates definite gaps
in the set $X$. Moreover, due to the expansion of the transit maps
$D_{i+1}\ra D_i$, these gaps lie in different scales. By Lemma \ref{gaps in many scales}, 
the set $X$ has an exponentially small (in terms of $n$)  area. 

On the other hand, almost all points of $J(\Bf_*)$ eventually land in
$V'$.%
\footnote{ We don't really need it as we can assume  without loss of
  generality that $\area J(\Bf_*)=0$.}
 It follows that the escaping probability $\xi$ is exponentially small
 in $n$. 
************}

\comm{****  An attempt to construct a nice return map
To make this work, let us create a convenient return map to a
slightly adjusted trapping disk $D$: 

\begin{lem}\label{returns to D} 
Take any $\de>0$.  If $P_*$ is sufficiently close to $P$
then  there exists a trapping disk $D_*\supset D$, a family of disjoint subdisks
$W_i\subset D_*$ and return times $n_i$ such that:

\ssk \nin  {\rm{(i) }}
   there exists a measurable subset $X\subset D_*$ such that 
$\area X\geq \xi_*  \, \area D_*$ for some $\xi_*>0$ independent of
$\de$, 
and the orbits of all points  $z\in X$ eventually  land in  $U'$;

\ssk \nin {\rm{(ii) }}
    $\area (\bigcup W_i) \geq (1-\de)\,  \area D_*$; 
and for any $i$,  the iterate $P_*^{n_i}$ maps $W_i$ onto $D_*$
 with uniform expansion and bounded distortion
(with absolute constants).
\end {lem}

\begin{proof}
  Let us slightly enlarge the trapping disk $D$ to obtain a disk 
$ D^* \Supset D$ with the Hausdorff distance between $\di D$ and 
$\di  D^*$ of order $\de$ (but small compared
with $\de$).%
\footnote{Meaning that the distance  is squeezed in between $\de/C$ and
  $\de/C'$,
where $C$ and $C'$ are sufficiently big constants, to be selected later.}
   On the other hand,  let $S_*\Subset S$ be a slightly
shrunk (in the same sense) Siegel disk. 

Let $Y\subset D$ be the set of points $z$ such that
there exist unique moments $m(z)= m<n=n(z) $ with the properties that
$f^m z\in S_* $,
$f^n z \in D$, while $f^k z\not\in  D^* $ for $k < n$
and $f^k z\not\in  S $ for $k\leq n$ and  $k\not m$. 
Pulling the disk $D_*$ back along the orbit $\{ f^k z\}_{k=0}^n$,
we obtain a disk $W= W(z)\ni z $ and a univalent map $f^n:  W\ra D^*$
with bounded distortion (with an absolute bound).
We claim that this family of disks and return times satisfy the
desired properties. 

As we pull $D^*$ back to $z$,  it enjoys definite contraction
in the complement of the postcritical set  every time it passes by
$c_0$. If $P_*$ is sufficiently close to $P$, the orbit  $\{f^k z\}_{k=m}^n$ it
makes a lot of revolutions near the boundary of $S$ passing close to
$c_0$. Hence the map $f^n : W \ra D^*$ is expanding by a big factor,
so $W\Subset D^*$.
 
Assume that  another domain, $W_1= W(z_1 )$ for  some $z_1\in Y$  
overlaps with $W$. Let $n_1= n(\zeta)$ be the corresponding return time;
assume for definiteness that $n_1\geq n$.  Then $f^n W_1$ overlaps
with $D^*$, while $f^m W_1$ overlaps with $S_*$. 
By the same expansion property,  the domain $f^m (W)_1)$ is exponentially small it terms of
$n_1-m$, and hence $f^m z_1\in S$.  By definition, we conclude that $m_1=m$.

\end{proof}
********************}

\comm{****** Old return machinery 
How does the Return machinery work?
Recall that $P$ is the Siegel polynomial.
Given an $\eps>0$,
find  a distance  $d>0$ such that $K(P)$ has density $1-\eps$ in the
outer $(d/2, d) $-collar $B$ around the Siegel disk $S=S_P$ (the ``toll belt''). Select $\kappa$
and $m$ big enough so that the postcritical set of $P_*$ and the disk
$V'$ lie much closer to $\di S$. 

Let $K^*=\bigcup P_*^{-n}(S)$. 
If $P_*$ is a sufficiently small perturbation of $P$ then $K^*$ has density at least  $1-2\eps$  in
the toll belt $B$.  

Let $\V$ be  the enlargement of  $V'$ by a small factor $1+\de$. 
Given  a set $Q\subset S$ consisting of pullbacks of $V'$,
 we let  $\Q$ be the set consisting of the corresponding pullbacks
of $\V$.  Then $\area \Q\leq (1+O(\de)) \area Q$.  

Let $ Q_1\subset S $ be the set of points  that land in $V'$ before they
enter the toll belt $B$.
We know that it has density $\geq \eta$.
 
Let $z\in S\sm \Q_1$.  If $z$ eventually lands in $B\cap K^*$ under some
iterate $P_*^n$,  let    $\Delta$ be the component of  $K^*$ containing
$ P^N_*(z)$.  When we pull $\Delta$ back to $z$ then it enjoys definite contraction
in the complement of the postcritical set  every time it passes by
$c_0$.  Hence it will arrive to a neighborhood of $V'$ being
much smaller than $\de \diam V'$.   
Since the orbit of $z$ misses $\V$,
the pullback of $\Delta$  misses  $ V'$.     

So, these pullbacks form a set $D_2 \subset S$ disjoint from $Q_1$ and such that 
in any component of $D_2$ we have the $\asymp \eta$  portion of $\XX$ (points
eventually returning  to  $V'$).  Let $Q_2= D_2\cap \XX$. 
The set $\Q_2$  occupies at  least  $\asymp \eta$ part of $D_2$, which
covers $(1-O(\de))$-part of the complement of $Q_1$.  
So, $Q_2 $ eats up a definite portion of the complement of $Q_1$.

Take now  a point $z\in S\sm (\Q_1\cup \Q_2)$, and proceed in the same
way, and so on. 
%
End of  of Old Return Machinery************}

\comm{
*******************************

\begin {lem}\label{P-plane}
  If $\kappa$ is sufficiently big then for any $m\geq l+\kappa$ the
  parabolic map $P_m$ possesses the following objects:

\ssk\nin  $\bullet$  a repelling periodic point $\zeta'\equiv \zeta^m_{m-k}$
  with period $q_{m-\kappa}$;

\ssk \nin $\bullet$ 
a quasidisk $\Disc' \equiv \Disc^m_{m-k}$ centered at
$\zeta'$ and separated from the
flower $ \FF' $ whose Euclidean  diameter is comparable to 
 $\dist (\zeta',  c_0)$ (in terms of $\kappa$);

\ssk\nin $\bullet$ 
a point $b$ in the repelling petal $\CC^r$ of $0$ such that $f^s(b) =\zeta'$
 at some moment $s$;

\ssk \nin $\bullet$  
   a precritical point $c_{-\tau}$  well inside $\Disc'$;

\ssk \nin $\bullet$ 
  a postcritical point $c_n$ in the attracting crescent $\CC^a$.
  \note{which depends on $\la$? } 

\ssk
Moreover, the moments $s$ and $\tau$ are bounded in terms of the
renormalization $\tl f$ (\ref{f_m}), 
and
\begin{equation}\label{derivative bounds}
    C^{-1} \leq  | DP_m^s (b)|  \leq C, \quad
   C^{-1} \leq  | DP_m^\tau (c_{-\tau}) |  \leq C,
\end{equation}
where all constants and bounds depend on $\kappa$ uniformly in $m$. 
\end{lem}

\begin{proof}
    Consider the renormalization $f$  \ref{m-2 kappa}) of $P$. 
By Corollary \ref{Siegel renorm of famiies},  the   parabolic map $\tl
f$ 
(with rotation number $p_{l+\kappa} / q_{l+\kappa}$)  
is $\de_0$-close to $f$  if $\kappa$ is big (uniformly in $m$).
So  Lemma \ref{definite disk} and Corollaries  \ref{point in petal} \& 
\ref{perturb u-mand of zeta} are applicable, as well as Proposition~\ref{postcrit set for IS}.
Lifting the objects constructed there to the 
$P_m$-plane using  Lemma~\ref{perturbed changes of variables},
we obtain the desired objects, except that we need to 
to adjust  the $a$-point slightly (to obtain the $b$-point).  
Namely, since $ \Disc'$ is well
inside $\C\sm \FF'$, it can be uniformly pulled back towards $0$
along the orbit of $a$.  The corresponding backward orbit of $\zeta'$
must pass, at some moment $s$,   through the repelling crescent
$\CC^r$. This is our $b$-point. 

To deduce bounds (\ref{derivative bounds})  we  use that $s$ and $\tau$ are
bounded in terms of $\tl f$ and the changes of variable $\pi_{m-\kappa-l}$  have  bounded
distortion at the endpoints of the corresponding orbits. 
\end{proof}

\begin{rem}\label{periodic pt: puzzle description}
The periodic point $\zeta' $  can be characterized as a landing point of a ray
  with rotation number $p'/q'\equiv p_{m-\kappa}/q_{m-\kappa}$
that was born in the parabolic explosion from
$P_{p'/q'}$, 
compare Remark \ref{parabolic explosion} .
It can also be described in terms of the top level of the puzzle:
$$
           U= \BY_0 \cup\bigcup_{i=1}^{q-1} ( \BY_i  \cup \BZ_i), \quad \Bq=\Bq_m,\
           \BY_0\ni c_0, 
$$  
where $\BY_i$ are the sectors attached to the $\alpha$-fixed point
and $\BZ_i= - \BY_i$ are the symmetric ``lateral'' sectors. 
 Let us label the sectors $\BY_k$, $k\in \Z/q\Z$,  according
to the dynamics:
$$
   0\in \BY_0\ra \BY_1\ra\dots \ra \BY_{q-1} .
$$ 
 Then $ \zeta' $  belongs to the  lateral sector $\BZ_{-q' }$, 
under the first iterate it jumps to the
sector $\BY_{-q' +1}$  attached to the $\alpha$-fixed point, and
then goes  through the  sectors $\BY_i$ until landing in  $\BY_{-1}$ and
then  closing up. 
\end{rem}

\comm{******
\begin{lem}\label {hyp contraction}

   Let $W'=W^m_{m-\kappa}$ be the pullback of the disk $V'=V^m_{m-\kappa}$ constructed in Lemma \ref{P-plane}
to the repelling petal. Then $W'$ is a
$L(\kappa)$-quasidisk  of a small hyperbolic diameter bounded in terms of
$\kappa$:  
$$
         0< \si_1 (\kappa) \leq      \diam_\hyp W' \leq \si_2 (\kappa)\to 0\  \mathrm{as} \
                  \kappa\to 0
$$ 
 (uniformly in  $m$), where the hyperbolic distance is measured in
 $\C\sm \FF'$. 
\end{lem}

\begin{proof}
  Let us go to the dynamical plane of the renormalization 
$\tl f =  \tl  f_m$  (\ref{f_m}), which is a small perturbation of the
  Siegel map $f$  (\ref{m-2 kappa}). 
The disk $\tl V\ni \tl f^s \tl a$ 
from Lemmas~\ref{definite disk} and \ref{point in petal} can be
univalently pulled back
along the orbit of $\tl a$, 
$$
 (\tl V, \tl f^s \tl a ) \ra (\tl V_{-1}, \tl f^{ s-1} \tl a)  \ra
 \dots  \ra (\tl V_{-T}, \tl f^{s-T} \tl a),   
$$
shadowing  closely (for some time $T=T(\kappa)\to \infty$ as $\kappa\to
\infty$) a backward  orbit of the Siegel map $f$. 
Hence we can find  $n$ moments of time $s-i$
(with $i\leq T$ and  $n=n  (\kappa) \to \infty$ as $\kappa\to \infty$)
when  $\tl f^{s-i} \tl a$  is closer (up to a constant)  in the
Euclidean metric to the
symmetric petal $\tl \FF^\sym_0  $ than to $\tl \FF_0$. 
At these moments, the hyperbolic  distance from  $\tl V_{-i}$ to 
$\tl \FF^\sym_0$ is bounded, 
and hence it enjoys a definite  hyperbolic contraction under the
pullback  (we apply here McMullen's argument \cite{McM1}).
Further pullbacks can only contract it more.

However, the result depends only on $\kappa$ (uniformly in $m$) by
(\ref{derivative bounds}).
\end{proof}
 
*******}

%
\comm{*****
Consider the  pullback $W'$ of $V'$ from Lemma  \ref{hyp contraction} . 
For a big $\kappa$, it has a small diameter in the complement of the
flower $(\FF^m)'$.   Hence it can be inscribed in a disk $W''$ of
small hyperbolic diameter with a big  $\mod (W''\sm W')$.  

Let us select  a repelling crescent $\CC^r$ so that $\CC^r \sm  \FF'$
is a rectangle with bounded geometry and whose  boundary stays a
definite hyperbolic distance away from some point of $W'$. Then
$W''\subset \CC^r$.  
In particular, the disks $W''\supset W'$ 
univalently projected onto  a disk $\BW''\supset \BW' $ in the repelling cylinder $\Cyl^r$. 

By Property (P3) of the  IS class, there is a postcritical point
${\mathbf c}_n$ in the attracting cylinder $\Cyl^a$.  Let $\La$ be the set of
$\la\in \Cyl^a$ such that $\BI_\la(\mathbf {c}_n)\in \BW'$. Obviously, it is an
$L$-quasidisk in $\Cyl^a$ that is mapped univalently onto $\BW'$ by the
map $\la\mapsto \BI_\la({\mathbf c}_n)$. (In the affine cylinder coordinates, $\La$
is just the  translation of $\BW'$  by $-{\mathbf c}_n$.)   

For $\la\in \La$, let $\mathbf{\Upsilon}_\la' := \BI_\la^{-1} (\BW')$  and $\mathbf{\Ups}_\la'' := \BI_\la^{-1} (\BW'')$. 
Let $\Upsilon_\la''\supset \Ups_\la'' \supset \Ups_\la' \ni c_n$  be the
lifts of these disks  to the  attracting
crescent $\CC^a_\la$.  
Pulling them back along the critical orbit to $c_0$ 
and then further to the precritical  point $c_{-\tau}$ from 
Lemma \ref{P-plane}, 
we obtain disks $U_\la''\supset U_\la'$ around $c_{-\tau}$. 

Let us show that the maps $F_\la: U_\la'\ra V'$  are
quadratic-like:

\ssk\nin $\bullet$ The map $P_m^s: W_\la'\ra V_\la'$ is univalent
since  $V'$ does not intersect the flower $\FF'$ which contains
the postcritical set $\OO'$;

\ssk\nin $\bullet$  The map $ T_\la: (\Ups_\la'', \Ups_\la') \ra (W'',
W')$ is univalent 
as a lift of the univalent map $\BT_\la:  (\mathbf{ \Ups}_\la'' ,
\mathbf{ \Ups}_\la')  \ra (\BW_\la'', \BW')$
between topological disks;
  
\ssk\nin $\bullet$  The pullback  $\Om_\la''\supset \Om_\la'$
of the disks $\Ups_\la''\supset \Ups_\la'$ to $c_0$ has degree
two since it corresponds to the first landing of the critical point to
$\Ups_\la'$. Moreover,  the Euclidean diameter of $\Om_\la'$ is small since it has a big
collar $\Om_\la''\sm \Om_\la'$ around itself such that the inner
radius of $\Om_\la''$ around $c_0\in \Om_\la'$ is bounded
(for instance,  $c_1\not\in \Om_\la''$).

\ssk\nin $\bullet$  The further pullback of $\Om_\la'$ to $U_\la'\ni
c_{-\tau}$ is univalent  if $\diam \Om_\la'$ is small enough
(since $\tau$ is independent of $\kappa$).

\ssk Composing the above maps, we obtain for any $\la\in \La$ 
a  branched covering $F_\la: U_\la'\ra V'$ of degree two.

\ssk 
Since $c_{-\tau}$ is well inside $V'$ while
the Euclidean diameter of $U_\la'$ is small (if $\kappa$ is big enough)
the disk $U_\la'$ is well inside $V'$ with 
$$
 \mod(V'\sm U_\la') > \mu(\kappa)\to \infty\quad  {\mathrm{as}} \
 \kappa\to \infty\quad \mbox{(uniformly in $m$)}.
$$

Thus, we have constructed a quadratic-like family with big modulus. 
This family is full and unfolded. \note{to be defined   in S 2}
Indeed, by construction, as  $\la$ goes once around  $\di \La$ then
$\BT_\la (\mathbf{ c}_n) = \la+ \mathbf{ c}_n$ (in affine
cylinder coordinates)   goes once around $\di \BW'$.
Since the lift  $\pi^{-1}: \BW'\ra   W'$ and the map  $P_m^s: W'\ra V'$ are  univalent, the critical value 
$F_\la(c_{-\tau}) =  P_m^s (\pi^{-1}  (\BT_\la( \mathbf{ c}_n))) $
goes once around $\di V'$.

\ssk 
By the Koebe Distortion,  $U_\la'$ is an $L'(\kappa)$-quasidisk provided
$c_n$ is not too close to $\di \Ups_\la$, which is OK in case of
connected Julia set.

The size of this quasidisk is bounded from below in terms of $\kappa$
(uniformly in $m$) by  (\ref{derivative bounds}).

*****************}

Perturbing the parabolic map $P_m$ (with the transit time  
bounded (albeit big) in terms of $\tl f$ (\ref{f_m}), 
we obtain a genuine quadratic-like renormalization: 

\begin{cor}
Let $\theta=\theta_N$ be a stationary rotation number.  
There exists $\underline{\kappa}= \underline{\kappa}(\theta) $ such that 
 for any $\kappa> \underline{\kappa}$  and $m>\kappa+ l$,  
there is a domain $\La\subset \C$ in the wake of $e^{2\pi i \Bp_m/\Bq_m}$ 
with the following property. There is a family of disks 
$U_\la'\subset V' $,  $\la\in \La$, 
and a renormalization period $k >1$ such that
the maps $F_\la=P_\la^k  : U_\la' \ra V' $ form  a full unfolded quadratic-like family over $\La$. 
Moreover, $\mod F_\theta \geq \mu(\kappa)\to \infty$ as
$\kappa\to \infty$. 
In case of connected Julia set $J(F_\la)$, the disk $U_\la'$ is an
$L(\kappa)$-quasidisk with   $\area U'_\la \geq c(\kappa)>0 $ (all uniformly in $m$).
\end{cor}

This quadratic-like family produces a little Mandelbrot copy
$\MM'=\MM'_{\kappa,m}$ that determines the desired renormalization
combinatorics.

\begin{rem}
  This combinatorics is primitive. In terms of the top level of the
  puzzle it can be  described as follows.   \note{can it ?} 
The critical orbit makes several revolutions 
  around the $\alpha$-fixed point, then escapes through the lateral
  sector (attached to $-\alpha$) containing the periodic point
  $\zeta^m_{m-\kappa}$, 
and then makes several revolutions following  combinatorially  this
cycle (see Remark \ref{periodic pt: puzzle description}), until it returns back $0$.   
\end{rem}

\begin{rem}
  In fact, $\kappa$ and $m$ do not determine the combinatorics $\MM'$ uniquely as there are various choices
made in Lemmas \ref{point in petal},  \ref{perturb u-mand of zeta}
and Corollary \ref{ql family}. Let us make some choice $\MM'_{\kappa, m}$:
any of them would do. 
\end{rem}
**************************************}



\subsection{Positive area}

\begin{thm}\label{main the}
   For any stationary rotation number $\theta=\theta_N$ of high type (i.e., $N> \underline{N}$),
there exist $\underline{l}$,$\underline{\kappa}$,
$\underline{t}$  and $\underline{m}$, $\underline{j}$ with the following property.
If  $\kappa, l, t, m, j$ are larger than the corresponding underlined parameters,
then the Feigenbaum polynomial $\Bf_*$ with stationary combinatorics
$\MM'_{N, l, \kappa,t,m, j }$ 
has  the Julia set of positive area.
\end{thm} 

 \begin{proof}
By Proposition \ref{def eta}, the map $\Bf_*$ has {\it a priori bounds} depending only
on $N, l, \kappa$, and $t$.

By Proposition \ref{def eta}, it has a definite {\it landing parameter} $\eta$ depending on the same four parameters only.

By Proposition \ref{small xi}, it has an arbitrary small {\it escaping parameter} $\xi$ as long as 
$m, j$ are sufficiently big (with frozen $N$, $\l$, $\kappa$, and $t$).

Now the Black Hole Criterion (Theorem \ref{black hole}) implies the desired.
\comm{****
By Theorem \ref{renorm fixed point},
 there is the DH Renormalization Fixed Point $g : U \ra V $
in the hybrid class of $\Bf_*$. Moreover, $R_{DH}^n  (\Bf_*) \to g$ implying that the {\it a priori bound}
$\nu(N, \kappa,l t)$ for $\Bf_*$ also works for $g$. 
 
We will  apply Theorem \ref{black hole}  to $g$, 
which will imply the desired result for $\Bf_*$
(as they are  qc equivalent).
To this end we need to estimate 
{\it  the landing parameter $\eta$}  (\ref{landing par})
and  {\it the escaping parameter $\xi$ }  (\ref{escaping par})
of this map.
In fact,  both of them have been already taken care of
in Propositions \ref{} and \ref{}. 
%
%

\ssk The Theorem is proven.
****}
\end{proof}



\section{Appendix: Further comments and open problems}

\subsection{Probabilistically balanced maps}
There is an interesting approach  to 
creating balanced (in some stronger sense) maps by variation of a continuous parameter (we thank
Jean-Christophe Yoccoz for this suggestion).  Consider a renormalization
horseshoe associated to a pair of renormalization combinatorics, such that
one of the fixed points is lean and the other is a black hole.  For each $0
\leq p \leq 1$, let $\mu_p$ be the Bernoulli measure on the horseshoe giving
probability $p$ to the ``Lean'' combinatorics and $1-p$ to the ``Black hole''
one.  Then for each $p$, the limit
$$
   c_p = \lim \frac {1} {n} \log \frac {\eta_n} {\xi_n}
$$
should exist  $\mu_p$-a.e. and be independent of a particular
$\mu_p$-typical  combinatorics.
Moreover, the dependence $p \mapsto c_p$ is  conceivably  continuous, 
and since $c_0<0<c_1$, we must have $c_{p_*}=0$ for some $0<p_*<1$ 
(justification of all those facts would depend on a suitable extension of the
analysis of \cite {AL}).  Let us call a $\mu_{p_*}$-typical Feigenbaum map
{\it probabilistically balanced}. (They are ``better balanced'' than generic 
{\it topologically balanced}  examples constructed  in \cite {AL}.)
The geometry of the probabilistically balanced  Julia sets would be
a good approximation to the geometry of  (perhaps, non-existing) balanced  Julia sets with periodic combinatorics.%
\footnote {Note however that $\mu_p$-a.e. Feigenbaum Julia set
has full hyperbolic dimension for every $0<p<1$ (see Lemma 7.2 and Theorem
8.1 of \cite {AL}), and while $c_p>0$ should imply positive area,
$c_p<0$ would not imply Hausdorff dimension less than $2$.}

\subsection{Computer experiments}

After identifying theoretically the main dynamical phenomena which should
lead to Black hole behavior, we have attempted an informal
numeric investigation of
a particularly simple sequence of renormalization combinatorics displaying
them.  Consider the quadratic map $p_c$ with a
golden mean Siegel disk, with rotation number
$[1,1,1,...]$, and let $p_m/q_m$ be the sequence of rational approximands
($p_m=q_{m-1}$ being the Fibonacci sequence).  Visual inspection of the
$(p_m/q_m)$-limb reveals a pair of largest primitive
Mandelbrot copies with period $q_m+q_{m-2}$.  Choosing one of them,
we explore in detail the parameter $z_m$ in this copy for which the first
renormalization has a golden mean Siegel disk.  This parameter is very
close to the actual Feigenbaum parameter with this stationary
combinatorics, and considerably easier to determine numerically.

In parameter space, one sees that $\frac {z_{2 m-1}-c} {z_{2 m+1}-c}
\to \beta=\frac {7+3 \sqrt {5}} {2}$.  Moreover, centering the Mandelbrot
copies at the superattracting parameter and rescaling by
$\beta^m$ shows manifest convergence of the copies in the Hausdorff topology.

In the dynamical plane, one sees that $p_{z_{2 m+1}}^{q_{2 m+1}+q_{2 m-1}}$
restricts to a quadratic-like map $g_{2 m+1}:U_{2 m+1} \to V_{2 m+1}$, where
$V_{2 m+1}$ is a disk of radius
$\sqrt {38} |w_{2 m+1}|$ and $w_{2 m+1}$ is the center of the Siegel disk for
$g_{2 m+1}$.  Moreover, $\frac {w_{2 m-1}} {w_{2 m+1}}$ converges to some
real constant greater than $1$, and up to rescaling by $|w_{2 m+1}|^{-1}$,
$g_{2 m+1}$ is seen to converge.  The proportion of $p_{z_{2 m+1}}$-orbits
starting in the original Siegel disk of $p_c$ that eventually land in
$V_{2 m+1}$ is clearly seen to approach $1$ (so that $\eta(2 m+1)$ is
bounded from below), while $\xi(2 m+1)$ appears to decay exponentially. 
Julia sets of positive area might already emerge then for period $2207$
($\xi \approx 0.0622$), and more likely for period $15127$ ($\xi \approx
0.0215$).\footnote {Those estimates are valid for the quadratic map and not
for the renormalization fixed point, so there is still some extra distortion
to consider.  Heuristically (ignoring distortion), $\xi$ should be small
compared to the relative area of the filled Julia set with a Siegel disk,
which near the fixed point is around $0.06$.}

All those observations would be justified by the existence of an
hyperbolic Siegel renormalization fixed point  \note{is it really enough?} 
 associated to the
golden mean, with one-dimensional unstable manifold containing
(up to straightening) the Mandelbrot copies we
explore.  While the existence of a hyperbolic Siegel renormalization fixed
point was  established by McMullen and  Yampolsky
\cite{McM3,Ya-posmeas} 
(and plays a central role in our argument), the current techniques do not go so far as to
to prove its hyperbolicity in the particular case of the golden mean.
But even this would not be enough: then  one needs to show that the unstable manifold is
large enough to contain those particular Mandelbrot copies that we
want, which  looks like a 
hard problem.

\subsection{More Julia sets of positive area?}

It remains an open problem whether Julia sets of positive area may exist for
real quadratic maps.  Any such example would have to be
infinitely renormalizable, and would imply their existence already in the
class of real Feigenbaum quadratic maps with periodic combinatorics.  A
natural candidate would be the ``original'' Feigenbaum map corresponding to
the period doubling bifurcation, since fixed points with
high (essential) period are known to be Lean.  However, in the doubling case
numerical experiments  still suggest that  the $\eta_m$ decay, 
 so this case appears to be Lean as well.  
It is thus plausible that all real quadratic Feigenbaum Julia sets are
Lean. However,  resolving  this problem one way or another  may depend
on  computer assistance.

In the higher degree case,  
the situation is even less conclusive. 
In this case, there is even a chance of existence of a
non-renormalizable   unicritical polynomial with positive area Julia
set (and even real):  see an attempt to prove it  by  Nowicki and van Strien 
for the  Fibonacci map of high degree,
see \cite {Buff}).

\end{document}